\documentclass[reqno,a4paper,12pt]{amsart}
\usepackage{amsmath,amsthm,amssymb,mathrsfs}
\usepackage[
	left = 2.9cm,
	right = 2.9cm,
	top = 2cm,
	bottom = 2cm
]{geometry}
\usepackage{bm,stmaryrd}

\usepackage[justification=centering]{caption}

\usepackage[toc,page]{appendix}
\usepackage[
	pdfencoding=auto,
	colorlinks = true,
	citecolor = blue,
	linkcolor = blue,
	anchorcolor = blue,
	urlcolor = blue, 
	filecolor=blue,
	pdfkeywords={}
]{hyperref}
\usepackage{enumitem}
\usepackage{tikz}

\newtheorem{theorem}{Theorem}[section]
\newtheorem{corollary}{Corollary}[section]

\newtheorem{proposition}{Proposition}[section]
\newtheorem{remark}{Remark}[section]

\numberwithin{equation}{section}

\newcommand{\bbr}{\mathbb{R}}
\newcommand{\bbi}{\mathbb{I}}
\newcommand{\bbn}{\mathbb{N}}

\newcommand{\bw}{\mathbf{w}}
\newcommand{\hw}{\hat{\bw}}

\newcommand{\bv}{\mathbf{v}}
\newcommand{\vf}{\bv_{f}}
\newcommand{\vs}{\bv_{s}}
\newcommand{\bu}{\mathbf{u}}

\newcommand{\hv}{\hat{\bv}}
\newcommand{\hu}{\hat{\bu}}
\newcommand{\hvf}{\hv_{f}}
\newcommand{\hvs}{\hv_{s}}

\newcommand{\bF}{\mathbf{F}}
\newcommand{\Fs}{\mathbf{F}_s}
\newcommand{\Fse}{\bF_{s,e}}

\newcommand{\hF}{\hat{\bF}}
\newcommand{\hFf}{\hF_{f}}
\newcommand{\hFs}{\hF_{s}}
\newcommand{\hFse}{\hF_{s,e}}
\newcommand{\hFsg}{\hF_{s,g}}
\newcommand{\Oft}{\Omega_{f}^{t}}
\newcommand{\Ost}{\Omega_{s}^{t}}
\newcommand{\Osg}{\Omega_{s}^{g}}

\newcommand{\Of}{\Omega_{f}}
\newcommand{\Os}{\Omega_{s}}

\newcommand{\OM}{\Omega}

\newcommand{\hnab}{\hat{\nabla}}
\newcommand{\hJ}{\hat{J}}
\newcommand{\hJf}{\hJ_f}
\newcommand{\hJs}{\hJ_s}
\newcommand{\hJse}{\hJ_{s,e}}
\newcommand{\hJsg}{\hJ_{s,g}}
\newcommand{\hg}{\hat{g}}
\newcommand{\rf}{\rho_f}
\newcommand{\rs}{\rho_s}
\newcommand{\hr}{\hat{\rho}}
\newcommand{\hrf}{\hr_f}
\newcommand{\hrs}{\hr_s}
\newcommand{\hc}{\hat{c}}
\newcommand{\hcf}{\hc_f}
\newcommand{\hcs}{\hc_s}
\newcommand{\hcss}{\hcs^*}
\newcommand{\nuf}{\nu_f}
\newcommand{\nus}{\nu_s}
\newcommand{\mus}{\mu_s}
\newcommand{\hnu}{\hat{\nu}}
\newcommand{\hnuf}{\hnu_f}
\newcommand{\hnus}{\hnu_s}
\newcommand{\hmu}{\hat{\mu}}
\newcommand{\hmus}{\hmu_s}
\newcommand{\vp}{\varphi}
\newcommand{\cL}{\mathcal{L}}

\newcommand{\wo}{\bw_0}

\newcommand{\bsigma}{\bm{\sigma}}
\newcommand{\sigf}{\bsigma_f}
\newcommand{\sigs}{\bsigma_s}
\newcommand{\sigse}{\sigs^{e}}
\newcommand{\sigsv}{\sigs^v}
\newcommand{\hsigma}{\hat{\bsigma}}
\newcommand{\hsigf}{\hat{\bsigma}_f}
\newcommand{\hsigs}{\hat{\bsigma}_s}
\newcommand{\hsigse}{\hsigs^{e}}
\newcommand{\hsigsv}{\hsigs^v}
\newcommand{\bn}{\mathbf{n}}
\newcommand{\hn}{\hat{\bn}}

\newcommand{\bT}{\mathbf{T}}
\newcommand{\hT}{\hat{\bT}}
\newcommand{\hpi}{\hat{\pi}}
\newcommand{\pif}{\pi_f}
\newcommand{\pis}{\pi_s}
\newcommand{\hpif}{\hpi_f}
\newcommand{\hpis}{\hpi_s}

\newcommand{\cD}{\mathcal{D}}

\newcommand{\YT}{Y_T}

\newcommand{\bS}{\mathbf{S}}

\newcommand{\bK}{\mathbf{K}}
\newcommand{\tk}{\tilde{\bK}}
\newcommand{\kf}{\tk_f}
\newcommand{\ks}{\tk_s}

\newcommand{\bH}{\mathbf{H}}

\newcommand{\FOinvtran}{\invtr{\hF} - \bbi}

\newcommand{\FfOinv}{\inv{\hFf} - \bbi}
\newcommand{\FfOinvtran}{\invtr{\hFf} - \bbi}
\newcommand{\FsO}{\hFs - \bbi}

\newcommand{\FsOinvtran}{\invtr{\hFs} - \bbi}
\newcommand{\TD}{T^{\delta}}

\newcommand{\hD}{\hat{D}}
\newcommand{\hDf}{\hD_f}
\newcommand{\hDs}{\hD_s}
\newcommand{\vo}{\hv^0}
\newcommand{\co}{\hc^0}

\newcommand{\tF}{\tilde{F}}
\newcommand{\tFf}{\tF_f}
\newcommand{\tFs}{\tF_s}

\newcommand{\cpi}{\check{\pi}}

\newcommand{\barpi}{\bar{\pi}}

\newcommand{\tv}{\tilde{\bv}}

\newcommand{\tpi}{\tilde{\pi}}

\newcommand{\nuSigma}{\nu_{\Sigma}}
\newcommand{\sL}{\mathscr{L}}
\newcommand{\sN}{\mathscr{N}}
\newcommand{\sM}{\mathscr{M}}

\newcommand{\cM}{\mathcal{M}}

\newcommand{\Smu}{S_{\mu}}
\newcommand{\ptial}[1]{ \partial_{#1} }
\newcommand{\tu}{\tilde{u}}
\newcommand{\tg}{\tilde{g}}
\newcommand{\tf}{\tilde{f}}
\newcommand{\baru}{\bar{u}}
\newcommand{\dW}[1]{\dot{W}^{#1}_{q}}
\newcommand{\dR}{\dot{\bbr}}
\newcommand{\hatu}{\hat{u}}
\renewcommand{\Bar}[1]{\overline{#1}}
\renewcommand{\Hat}[1]{\widehat{#1}}
\renewcommand{\Tilde}[1]{\widetilde{#1}}

\newcommand{\cu}{\check{u}}
\newcommand{\Sigtheta}{\Sigma_\theta}
\newcommand{\Gtheta}{G_\theta}
\newcommand{\Stheta}{S_\theta}

\newcommand{\nuStheta}{\nu_{\Stheta}}
\newcommand{\nuSigtheta}{\nu_{\Sigtheta}}
\newcommand{\cP}{\mathcal{P}}

\newcommand{\PSigtheta}{\cP_{\Sigtheta}}
\newcommand{\PSigma}{\cP_{\Sigma}}
\newcommand{\barx}{\bar{x}}
\newcommand{\bbe}{\mathbb{E}}
\newcommand{\bbeo}{{_0\bbe}}
\newcommand{\bbf}{\mathbb{F}}
\newcommand{\bbfo}{{_0\bbf}}
\newcommand{\hatf}{\hat{f}}

\DeclareMathOperator*{\supp}{\mathrm{supp}}
\newcommand{\ext}{\mathrm{ext}}
\newcommand{\cS}{\mathcal{S}}
\newcommand{\cR}{\mathcal{R}}

\newcommand{\onehalf}{\frac{1}{2}}
\newcommand{\tr}{\mathrm{tr}}
\newcommand{\rv}[1]{\left. #1 \right\vert}
\newcommand{\tran}[1]{ #1^{\top}}
\newcommand{\inv}[1]{ #1^{-1}}
\newcommand{\invtr}[1]{ #1^{-\top}}
\renewcommand{\d}{\mathrm{d}}
\newcommand{\pt}{\partial_t}
\newcommand{\jump}[1]{\left\llbracket #1 \right\rrbracket}
\newcommand{\abs}[1]{\left\vert #1 \right \vert}
\newcommand{\norm}[1]{\left\Vert #1 \right \Vert}
\newcommand{\seminorm}[1]{\left[ #1 \right]}

\newcommand{\Lq}[1]{L^{#1}}
\newcommand{\Lqa}[1]{L^{#1}_{(0)}}
\newcommand{\W}[1]{W^{#1}_{q}}
\newcommand{\WO}[1]{{_0W}^{#1}_{q}}
\newcommand{\WOO}[1]{W^{#1}_{q,0}}

\newcommand{\Bq}[1]{B^{#1}_{q,q}}

\newcommand{\Bqp}[1]{B^{#1}_{q,p}}
\newcommand{\tin}{\quad \text{in }}
\newcommand{\ton}{\quad \text{on }}

\newcommand{\doi}[1]{DOI: \href{https://doi.org/#1}{#1}}
\newcommand{\arxiv}[1]{arXiv: \href{https://arxiv.org/abs/#1}{#1}}

\newcommand{\FSI}{\eqref{Fluid}--\eqref{initialc}}

\DeclareMathOperator*{\Div}{\mathrm{div}}
\newcommand{\hdiv}{\widehat{\Div}}
\newcommand{\hDelta}{\widehat{\Delta}}

\allowdisplaybreaks[4]

\begin{document}

\title[Fluid-structure interaction problem for plaque growth]{On a fluid-structure interaction problem for plaque growth: cylindrical domain}

\author{Helmut Abels}
\address{Fakult\"at f\"ur Mathematik, Universit\"at Regensburg, 93040 Regensburg, Germany}
\email{Helmut.Abels@ur.de}
\author{Yadong Liu}
\address{Fakult\"at f\"ur Mathematik, Universit\"at Regensburg, 93040 Regensburg, Germany}
\email{Yadong.Liu@ur.de}

\date{\today}

\subjclass[2010]{Primary: 35R35; Secondary: 35Q30, 74F10, 74L15, 76T99}
\keywords{Fluid-structure interaction, two-phase flow, growth, free boundary problem, maximal regularity, contact angle problem} 

\thanks{Y. Liu was supported by the RTG 2339 ``Interfaces, Complex Structures, and Singular Limits'' of the German Science Foundation (DFG). The support is gratefully acknowledged.}

\begin{abstract}
	This paper concerns a free-boundary fluid-structure interaction problem for plaque growth proposed by Yang et al. [J. Math. Biol., 72(4):973--996, 2016] with additional viscoelastic effects, which was also investigated by the authors [arXiv preprint: 2110.00042, 2021]. Compared to it, the problem is posed in a cylindrical domain with ninety-degree contact angles, which brings additional difficulties when we deal with the linearization of the system.
	By a reflection argument, we obtain the existence and uniqueness of strong solutions to the model problems for the linear systems, which are then shown to be well-posed in a cylindrical (annular) domain via a localization procedure. Finally, we prove that the full nonlinear system admits a unique strong solution locally with the aid of the contraction mapping principle.
\end{abstract}

\maketitle


\section{Introduction}
In this paper, we focus on a 3d free-boundary fluid-structure interaction problem for plaque growth, which was also addressed in the authors' previous work \cite{AL2021}. To be more precise, the blood is assumed to be the incompressible Navier--Stokes equation and the artery is modeled by an elastic equation with viscoelasticity, while inside the blood flow and vessel wall react the cells, leading to the plaque formation, see e.g. Yang--J\"{a}ger--Neuss-Radu--Richter \cite{Yang2016}.
Define $ \OM^t := \Oft \cup \Ost \cup \Sigma^t \subset \bbr^3 $ (see Figure \ref{domain}), with three disjoint parts, where $ \Oft $, $ \Ost $ are piece-wise smooth domains for fluid and solid respectively, while $ \Sigma^t $ is a two dimensional sub-manifold of $ \bbr^3 $ with boundary $ \partial \Sigma^t $. 
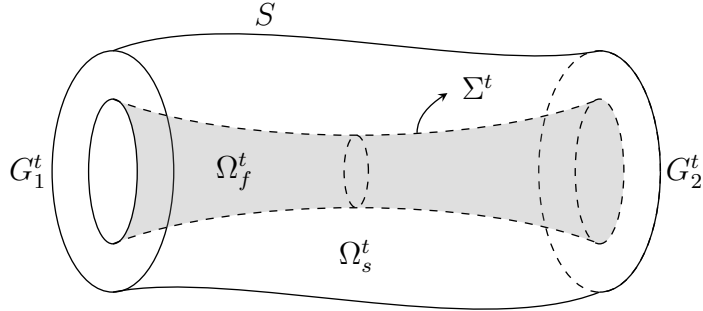
\begin{figure}[h]
	\centering
	\begin{tikzpicture}[scale=1.6,line width=0.5pt]
		\draw[rounded corners=35pt](0,0) .. controls (1,0.4) and (3,-0.2) .. (4,0);
		\draw[rounded corners=35pt](0,-2) .. controls (1,-1.8) and (3,-2.4) .. (4,-2);
		\draw (4.5,-1) arc (0:90:0.5 and 1);
		\draw[dashed] (4.5,-1) arc (0:360:0.5 and 1);
		\draw (4.5,-1) arc (0:-90:0.5 and 1);
		\draw (0.5,-1) arc (0:360:0.5 and 1);
		
		\fill [gray,opacity=0.25] (0,-1.6) .. controls (1,-1.2) and (3,-1.2) .. (4,-1.6) -- (4,-1.6) arc (-90:90:0.2 and 0.6) -- (4,-0.4) .. controls (3,-0.8) and (1,-0.8) .. (0,-0.4) -- (0,-0.4) arc (90:-90:0.2 and 0.6);
		
		\draw[dashed, rounded corners=35pt] (0,-0.4) .. controls (1,-0.8) and (3,-0.8) .. (4,-0.4);
		\draw[dashed, rounded corners=35pt] (0,-1.6) .. controls (1,-1.2) and (3,-1.2) .. (4,-1.6);
		\draw[dashed] (4,-1.6) arc (-90:270:0.2 and 0.6);
		\draw[dashed] (2.1,-1) arc (0:360:0.1 and 0.3);
		\draw (0.2,-1) arc (0:360:0.2 and 0.6);
		
		\node at (1, - 1) {$ \Oft $};
		\node at (2, - 1.7) {$ \Ost $};
		\node at (1.25, 0.3) {$ S $};
		\node at (-0.7, -1) {$ G_1^t $};
		\node at (4.7, -1) {$ G_2^t $};
		\node at (3, - 0.3) {$ \Sigma^t $};
		\draw [-stealth] (2.5, -0.68) .. controls (2.5, - 0.55) and (2.6, - 0.45).. (2.75, - 0.35);
	\end{tikzpicture}
	\caption{Deformed cylindrical domain.}
	\label{domain}
\end{figure}
In particular, $ \partial \OM^t = \Bar{G^t} \cup S $, $ \partial \Oft = \Bar{G_f^t} \cup \Sigma^t $ and $ \partial \Ost = \Bar{G_s^t} \cup \Sigma^t \cup S $, where $ G^t := G_1^t \cup  G_2^t \cup \partial \Sigma^t $ is a hypersurface with $ G_\beta^t := G_{1,\beta}^t \cup  G_{2,\beta}^t $, $ \beta \in \{f,s\} $, $ G_i^t = G_{i,f}^t \cup G_{i,s}^t $, $ \Bar{G_{i,f}^t} \subset G_i^t $, $ i \in \{1,2\} $, and $ S $ denotes the fixed surrounding surface, which is supposed to be perpendicular to $ G^t $ at $ \partial S $.  Moreover, $ \Sigma^t $ is assumed to be perpendicular to $ G^t $ at $ \partial \Sigma^t $ as well.
In such setting, the domain is endowed with the fixed contact line $ \partial S $ and moving contact line $ \partial \Sigma^t $ with ninety-degree contact angles for a short time, while in \cite{AL2021}, we considered a smooth domain without contact.

For $ T > 0 $, we model the motion of fluid by the classical incompressible Navier--Stokes equations, which is
\begin{equation}
	\label{Fluid}
	\left.
		\begin{aligned}
			\rf \left(\pt + \vf \cdot \nabla \right) \vf & = \Div \sigf(\vf,\pif) \\
			\Div \vf & = 0 
		\end{aligned}
	\right\} \quad \text{in}\ \Oft, \ t \in (0,T),
\end{equation}
where $ \rf > 0 $ is the fluid density. $ \sigf(\vf,\pif) = - \pif \bbi + \nuf ( \nabla \vf + \tran{\nabla} \vf )  $ denotes the Cauchy stress tensor, $ \pif $ is the unknown fluid pressure and $ \nuf > 0 $ represents the fluid viscosity.

Same as in \cite{AL2021}, the motion of vessel wall is captured by the incompressible neo-Hookean material with viscoelasticity
\begin{equation}
	\label{Solid}
	\left.
		\begin{aligned}
			\rs \left(\pt + \vs \cdot \nabla \right) \vs & = \Div \sigs(\vs,\pis) \\
			\left(\pt + \vs \cdot \nabla \right) \rs + \rs \Div \vs & = f_s^g
		\end{aligned}
	\right\} \quad \text{in}\ \Ost, \ t \in (0,T),
\end{equation}
where $ \rs > 0 $ is the solid density. $ \sigs = \sigse + \sigsv $ stands for the stress tensor of this kind of solid with 
\begin{align*}
	\sigse = - \pis \bbi + \mus \left( \inv{\Fse} \invtr{\Fse} - \bbi \right), \quad
	\sigsv = \nus \left( \pt \inv{\Fs} + \pt \invtr{\Fs} \right) \invtr{\Fs}.
\end{align*}
$ \pis $ denotes the solid pressure, $ \mus $, $ \nus $ represent the Lam\'e coefficient and viscosity respectively, which are all constant. 
The first equation is the balance equation of linear momentum.
$ \sigse $ is given by the constitutive relation of incompressible Neo-Hookean material, which is hyperelastic, isotropic and incompressible. This relationship was widely used to describe blood vessel wall by many investigators, see e.g. \cite{Tang2004,Yang2016}. Here $ \Fs $ is the inverse deformation gradient and $ \Fse $ denotes an inverse elastic tensor under the assumption of growth as in authors' previous work \cite{AL2021}. Similarly, the Kelvin--Voigt stress tensor $ \sigse $ is introduced, see e.g. Mielke--Roub\'{\i}\v{c}ek \cite{MR2020}. For more discussions about it, readers are referred to \cite[Remark 1.1]{AL2021}. Moreover, the second equation of \eqref{Solid} results from the mass balance, where $ f_s^g  = \gamma \beta c_s $ with $ \beta, \gamma > 0 $, is called growth function and represents the rate of mass growth per unit volume due to the formation of the plaque, which comes from the reaction of cells in the whole domain. See e.g. \cite{AM2002,JC2012,Yang2016}. 

The kinetic and dynamic conditions on the interface $ \Sigma^t $, $ t \in (0,T) $, are
\begin{align}
	\label{vjump}
	\jump{\bv} & = 0, \quad \text{on} \  \Sigma^t, \\
	\label{sigjump}
	\jump{\bsigma} \bn_{\Sigma^t} & = 0, \quad \text{on} \  \Sigma^t,
\end{align}
where $ \bn_{\Sigma^t} $ stands for the outer unit normal vector on $ \Sigma^t $ pointing from $ \Ost $ to $ \Oft $. $ \jump{f} $ denotes the jump of a quantity $ f $ defined on $ \Oft $ and $ \Ost $ across $ \Sigma^t $, namely,
\begin{equation*}
	\jump{f}(x) := \lim_{\theta \rightarrow 0} {f(x + \theta \bn_{\Sigma^t}(x)) - f(x - \theta \bn_{\Sigma^t}(x))}, \quad \forall x \in \Sigma^t.
\end{equation*}
Moreover, $ S $ is supposed to be the rigid part of the boundary, i.e.,
\begin{equation}
	\vs = 0 \quad \text{on} \  S,
\end{equation}
which means that the outside of the blood vessel is fixed.

Now for the boundary conditions on $ G^t $, we need a more careful consideration to make sure that they are physically meaningful and compatible to the conditions on $ S $ and $ \Sigma^t $ respectively. In this work, we impose the outflow conditions on $ G^t $, $ t \in (0,T) $, 
\begin{align}
	\cP_{G^t}(\bv) & = 0, \quad \text{on} \  G^t \backslash \Sigma^t, \\
	(\bsigma \bn_{G^t}) \bn_G & = 0, \quad \text{on} \  G^t \backslash \Sigma^t,
\end{align}
where $ \cP_{G^t} := I - \bn_{G^t} \otimes \bn_{G^t} $ denotes the tangential projection onto $ G^t $.
\begin{remark}\label{remark: G^t}
	$ G^t $ is a free surface. Since the vessel is cut in the Lagrangian coordinate, the cross sections then turn to be free when we recovered it in the Eulerian coordinate. In general, it can be fixed with suitable boundary conditions.
\end{remark}
\begin{remark} \label{Remark: outflow}
	Now we discuss more about the choice of the outflow boundary conditions on $ G^t $.
	\begin{enumerate}
		\item The boundary condition is not possible to be the Dirichlet type, i.e., \textit{no-slip} condition. On the one hand, a Dirichlet boundary implies the rigid part. For example $ \bv = 0 $ means that the displacement $ \int_{0}^{t} \bv \d t = 0 $ vanishes on the boundary, which contradicts to our consideration, as in remark \ref{remark: G^t}. On the other hand, a \textit{no-slip} condition leads to a incompatibility at the moving contact lines $ \partial \Sigma^t $. Specifically, on the interface $ \Sigma^t $ the normal velocity is $ V_{\Sigma^t} := \bv \cdot \bn_{\Sigma^t} $, while at $ \partial \Sigma^t $ one derives $ V_{\Sigma^t} = 0 $ according to the vanishing Dirichlet condition, contradicting to the ``moving contact line'' setting.
		\item In some literature, e.g. \cite{GT2018,GT2020}, the well-known Navier-slip condition was employed for the slip of the fluid along the solid with dynamic contact lines, while in \cite{Wilke2013}, Wilke showed that an incompatibility with moving contact lines happens for a two-phase Navier--Stokes problem in a cylindrical domain. So they considered a so-called \textit{pure-slip} condition. However, in our model, the surface $ G^t $ is not supposed to be fixed and it is physically meaningless if we take the \textit{pure-slip} condition into account. Thus, an outflow boundary condition is proper and fits to the reality of the blood flow in a vessel if we cut off the human vessels along the cross section.
		\item For the sake of analysis, the outflow condition allows for a reflection argument at the contact lines so that the reflected solutions of model problems are endowed with the symmetry, in view of the divergence-free condition of the solutions. See Section \ref{Section: Model problems} for more details.
	\end{enumerate}
\end{remark}
\noindent The initial data for velocities are prescribed as
\begin{equation}
	\label{initial}
	\vf(x,0) = \vf^0, \quad \vs(x,0) = \vs^0, \quad \text{in} \  \OM^0.
\end{equation}

Now as in \cite{AL2021,Yang2016}, we introduce the dynamics of monocytes in the blood flow and dynamics of macrophages and foam cells in the vessel wall, of which the concentrations are denoted by $ c_f $, $ c_s $, $ c_s^* $, respectively. The motion of monocytes in the blood is given by the transport-diffusion equation
\begin{equation}\label{CF}
	\pt c_f + \Div \left( c_f \vf \right) - D_f \Delta c_f = 0, \quad \text{in}\ \Oft, \ t \in (0,T),
\end{equation}
where $ D_f > 0 $ is the diffusion coefficient in the blood. Analogously, the motion of macrophages in the vessel wall is described by
\begin{equation}\label{CS}
	\pt c_s + \Div \left( c_s \vs \right) - D_s \Delta c_s = - f_s^r, \quad \text{in}\ \Ost, \ t \in (0,T),
\end{equation}
where $ D_s > 0 $ is the diffusion coefficient in the vessel wall. $ f_s^r := \beta c_s $ with $ \beta > 0 $ stands for the reaction function, modeling the rate of transformation from macrophages into foam cells. Thus, the transportation of $ c_s^* $ reads as
\begin{equation}\label{CSS}
	\pt c_s^* + \Div \left( c_s^* \vs \right) = f_s^r, \quad \text{in}\ \Ost, \ t \in (0,T),
\end{equation}
since the foam cells are not supposed to diffuse in vessels. 

To model the penetration of monocytes from the blood into the vessel wall, one gives the transmission conditions on $ \Sigma^t $ by the jump condition
\begin{align}
	\jump{D \nabla c} \cdot \bn_{\Sigma^t} & = 0, \quad \text{on} \  \Sigma^t, 
	\label{TC1} \\
	\zeta \jump{c} - D_s \nabla c_s \cdot \bn_{\Sigma^t} & = 0, \quad \text{on} \  \Sigma^t,
	\label{TC2}
\end{align}
where $ \zeta > 0 $ denotes the permeability of the interface $ \Sigma^t $ with respect to the monocytes, which is supposed to be a constant. Moreover, the conditions on $ G $ are given by
\begin{align} \label{CBoundary: G}
	D \nabla c \cdot \bn_{G^t} & = 0, \quad \text{on} \ G^t \backslash \partial \Sigma^t,
\end{align}
while on $ S $, 
\begin{equation} \label{CBoundary: S}
	D_s \nabla c_s \cdot \bn_S = 0, \quad \text{on} \  S.
\end{equation}
At initial state, the vessel is assumed to be healthy, which means that there is no foam cell in vessel, i.e.,
\begin{equation}
c_s^*(x,0) = 0, \quad \text{in} \  \Os^0.
\end{equation}
The initial values for concentrations of monocytes and macrophages are prescribed by
\begin{equation}
	\label{initialc}
	c_f(x,0) = c_f^0, \quad c_s(x,0) = c_s^0, \quad \text{in} \  \OM^0.
\end{equation}

For further analysis, we define the initial domain as $ \OM = \Of \cup \Os \cup \Sigma $, where $ \Of := \Of^0 $, $ \Os := \Os^0 $ and $ \Sigma := \Sigma^0 $ that is supposed to be perpendicular to the assumed flat initial surface $ G := G^0 $ at $ \partial \Sigma $ just like $ S $.
Then the deformation from initial configuration to current one is defined as
\begin{align*}
	\vp: \OM \times [0,T] \rightarrow \OM^t,
\end{align*}
with
\begin{align*}
	x = \vp(X,t) = X + \int_{0}^{t} \hv(X, \tau) \d \tau, \quad \forall X \in \OM,
\end{align*}
as well as $ \rv{x}_{t = 0} = \vp(X,0) = X $. Here $ \hv $ is the corresponding velocity in reference configuration.
In the sequel, the quantities or operators with a hat ``$ \hat{\cdot} $'' are all defined in reference configuration without any special clarification. Then deformation gradient has the form of
\begin{equation}\label{DG}
	\hF(\hv)(X,t) = \frac{\partial}{\partial X} \vp(X,t) = \hnab \vp(X,t) = \bbi + \int_{0}^{t} \hnab \hv(X,\tau) \d \tau,
\end{equation}
for all $ X \in \OM $ with initial deformation $ \hF |_{t = 0} = \bbi $ and by $ \hJ := \det \hF $ its determinant.
Conversely, we have the inverse deformation gradient by $ \bF(\bv)(x,t) = \inv{\hF} $.

To explain $ \Fse $ in \eqref{Solid}, which is the elastic tensor, we discuss the growth as in \cite{AL2021}. In general, the real deformation gradient $ \hFs $ is contributed by both growth and elastic deformation. Thus, a suitable treatment of the deformation is crucial to solve the system. In \cite{Yang2016}, Yang et al. decomposed the deformation gradient $ \hF $ based on the theory of multiple natural configurations, for which they introduced the so-called \textit{natural configuration}. For more details, see e.g. \cite{AM2002,JC2012,RHM1994,Yang2016}. 
In this article, we assume the decomposition of $ \hFs $ as
\begin{equation*}
	\hFs = \hFse \hFsg, \quad \text{in}\ \Os,
\end{equation*}
where $ \hFsg $ is the growth tensor and $ \hFse $ represents the elastic tensor. Then the corresponding determinants are 
\begin{equation*}
	\hJsg = \det \hFsg, \quad \hJse = \det \hFse, \quad \text{in}\ \Os,
\end{equation*}
with $ \hJs = \hJsg \hJse $.

To close the system, we still need a specific description of the growth tensor. Following \cite{AL2021,Yang2016}, a \textit{constant-density} type growth is taken into account since our solid is supposed to be incompressible. In the meantime, the second equation of \eqref{Solid} reads as
\begin{equation*}
	\rs \Div \vs = f_s^g.
\end{equation*}
Besides, the plaque is assumed to grow isotropically:
\begin{equation*}
	\hFsg = \hg \bbi, \quad \text{in}\ \Os,
\end{equation*}
where $ \hg = \hg(X,t) $ is the metric of growth, a scalar function depending on the concentration of macrophages. Hence, 
\begin{equation*}
	\hFse = \frac{1}{\hg} \hFs, \quad
	\hJsg = \hg^3,
\end{equation*}
where $ 3 $ is the dimension of space. Then as in \cite{AM2002,JC2012,Yang2016}, one obtains
\begin{equation}
	\label{gLagragian}
	\pt \hg = \frac{\gamma \beta}{3 \hrs} \hcs \hg, \quad \text{in}\ \Os,
\end{equation} 
which shows the specific dependence on $ \hcs $ of $ \hg $. At initial state, the growth tensor $ \hFsg $ is supposed to be the identity, i.e., 
\begin{equation} \label{initial-g}
	\hg(X,0) = 1, \quad \text{in}\ \Os.
\end{equation}

\subsection{Literature}
	Before stating our result, let us recall some previous works on the fluid-structure interaction problems, in which the incompressible Navier--Stokes equations coupled with elastic equations are considered. 
	
	In the case of 3d-3d model, the existence and uniqueness of strong solutions of such kind of models was firstly established by Coutand--Shkoller \cite{CS2005}, where they investigated the interaction of the Navier--Stokes equation and a linear Kirchhoff elastic material. The results were extended to the quasilinear elastodynamics case by them in \cite{CS2006}, where they regularized the hyperbolic elastic equation by a particular parabolic artificial viscosity and then obtained the existence of strong solutions by the a priori estimates. Subsequently, systems coupled the incompressible Navier--Stokes equation and the wave equation were continuously analyzed by Ignatova--Kukavica--Lasiecka--Tuffaha \cite{IKLT2014,IKLT2017}. More specifically, In \cite{IKLT2014}, a wave equation with several damping terms was considered and exponential decay of the energy was obtained. Later, they proved that the energy still decay without the boundary friction by introducing the tangential and time-tangential energy estimates. The coupling of the Navier–Stokes equation and the Lam\'{e} system was studied by Kukavica--Tuffaha \cite{KT2012} with initial regularity $ (v_0, w_1) \in H^3(\Of) \times H^2(\Os) $, while Raymond--Vanninathan \cite{RV2014} further obtained the same results by a weaker initial regularity $ (v_0, w_1) \in H^{3/2 + \varepsilon}(\Of) \times H^{1 + \varepsilon}(\Os) $, $ \varepsilon > 0 $ arbitrarily small, with periodic boundary conditions. Recently, Boulakia--Guerrero--Takahashi \cite{BGT2019} showed a similar result for the Navier--Stokes--Lam\'{e} system in a smooth domain with reduced demand of the initial regularity.
	
	Besides the standard models above, we refer to \cite{BG2010}, for compressible fluid coupled with elastic bodies, where Boulakia and Guerrero addressed the short time existence and the uniqueness of regular solutions with the initial data $ (\rho_0, u_0, w_0, w_1) \in H^3(\Of) \times H^4(\Of) \times H^3(\Os) \times H^2(\Os) $. Kukavica--Tuffaha \cite{KT2012compressible} improved the result by a weaker initial regularity $ (\rho_0, u_0, w_1) \in H^3(\Of) \times H^{3/2 + r}(\Of) \times H^{3/2 + r}(\Os) $, $ r > 0 $.
	More recently, Shen--Wang--Yang \cite{SWY2021} considered the magnetohydrodynamics (MHD)-structure interaction system, where the fluid is described by the incompressible viscous non-resistive MHD equation and the structure is modeled by the wave equation with superconductor material. They solved the existence of local strong solutions with penalization and regularization techniques.
	
	As for 3d-2d/2d-1d systems, numerous models and results were established during the last twenty years. The widely investigated case is the fluid-beam/plate systems where the beam/plate equations were imposed with different mechanical mechanism (rigidity, stretching, friction, rotation, etc.), readers are refer to e.g. \cite{CDEG2005,Grandmont2008,MC2013,TW2020} for weak solution results. Considering strong solutions, one can find related results in e.g. \cite{Beirao da Veiga2004,DS2020,GHL2019,Lequeurre2011,Lequeurre2013,MT2021NARWA,Mitra2020} and the references therein. Moreover, the fluid-structure interaction problems with linear/nonlinear shells were studied in e.g. \cite{BS2018,LR2013,MC2016} for weak solutions and in e.g. \cite{CCS2007,CS2010,MRR2020} for strong solutions respectively. It is worth mentioning that in recent works \cite{DS2020,MT2021NARWA}, a maximal regularity framework, which requires lower initial regularity and less compatibility conditions compared to the energy method, was employed.
	
	Concerning the contact angle problems in fluid dynamics, it is a challenging problem and is not yet well understood, especially with moving contact lines and dynamic contact angles. As far as we know, with a ninety-degree contact angle, Wilke \cite{Wilke2013,Wilke2020} gave a complete and outstanding analysis of two-phase Navier--Stokes equations with surface tension in cylindrical domains, using the framework of maximal regularity theory. Then Rauchecker \cite{Rauchecker2020} and Rauchecker--Wilke \cite{RW2020} extend the results to a two-phase Navier--Stokes/Stefan problem and a two-phase Navier--Stokes/Mullins--Sekerka system with boundary contact respectively. 
	The general methods in these papers are the localization procedure and the key element is the reflection argument for the model problems with respect to the linearized systems, thanks to the assumption of \textit{ninety}-degree contact angle. However, to the best of our knowledge, the problem with a dynamic contact angle still faces a large gap before completely being solved. Here we mention some works related to different contact angle problems in fluid dynamics. See e.g. \cite{GT2018,GT2020,MW2020,MW2021,TW2021,ZT2017} and references therein.
\subsection{Main features}
	Recently in \cite{AL2021}, we firstly established the short time existence of strong solutions to the considered plaque growth model including viscoelastic effects in a smooth domain under the maximal regularity theory. Motivated by these preceding results above, especially Wilke \cite{Wilke2020}, we are devoted to investigate the local existence of strong solutions to \eqref{Fluid}--\eqref{initial-g} in a cylindrical domain. In other word, we attempt to extend the results in \cite{AL2021} to the system in a cylindrical domain with ninety-degree contact angles, which is close to human arteries. 
	
	First of all, let us comment that we include the viscoelastic effects into the model as in \cite{AL2021}, which yields the parabolicity of the solid equation. To be precise, for a short time, the term describing the Kelvin--Voigt viscoelasticity leads to the principal regularity part of the equation, by which we can solve a two-phase Stokes type problem for the linearized fluid-structure interaction problem, which provides the solvabilities and maximal regularitie of the solutions. Throughout the proof based on the maximal regularity of type $ L^q $ that has been used in \cite{PS2010,PS2016,Wilke2013}, one obtains a semi-flow for the free boundary problem in a natural phase space. In particular, there is no loss of regularity and no more additional compatibility conditions needed. In this work, we only consider the three dimensional case for the sake of simplicity. In fact, it could be any dimension $ d \geq 2 $ as long as $ q $ has a modification restriction with respect to $ d $. This is also an advantage of the maximal $ L^q $-regularity theory.
	

	In the present paper, the boundary is not smooth any more since we have ninety-degree contact angles, so even for the linear system, there seem to be no well-posedness result that can be applied. To circumvent this problem, a localization procedure will be employed as in \cite{PS2016,Wilke2020}. More precisely, under suitable partition of unity for the cylindrical domain, we try to solve several kinds of model problems, for which we proceed using reflection arguments for the model problems in quarter (bent) spaces and half (bent) spaces, thanks to the ninety-degree contact angle. Then by a Neumann series argument, one can conclude the existence results of model problems and derive the well-posedness for the related linear system. However, we have no chance to consider all possible boundary conditions around the contact line, which play an essential role in both physics and mathematics. For example, as commented in \cite{Wilke2020} no-slip condition may lead to a paradoxon for the moving contact line, see also e.g. \cite{PS1983,GT2018,GT2020} and references therein. For the purpose they imposed a pure-slip condition on the boundary with contact, which guarantees that a reflection argument can be used. As in Remark \ref{Remark: outflow}, an outflow boundary condition is chosen on the surface $ G^t $ in our work. Here on the one hand, the outflow condition coincides with the situation of blood flow in vessels and does not violate the physical reality of a moving contact line. On the other hand, the symmetry of the reflected solutions on the boundary can be ensured from the perspective of mathematical analysis. For instant in the three dimensional case, if one carries out the reflection argument for Stokes equations in a quarter space with respect to $ x_2 = 0 $, the divergence-free condition $ \Div u = \ptial{1} u_1 + \ptial{2} u_2 + \ptial{3} u_3 = 0 $ should also hold on $ x_2 = 0 $. Then once $ u_1, u_3 $ are odd function with regard to $ x_2 = 0 $, $ \ptial{2} u_2 $ must be an odd function, which is exactly associated with the outflow boundary condition. 
	
	Moreover, due to the presence of the outflow boundary condition on $ G $ in \eqref{Linearized TwoPhase}, when we applied the localization argument, several auxiliary problems, namely, elliptic/Laplace transmission problem in a cylindrical domain with a Dirichlet boundary condition on $ G $ and parabolic transmission problem with a Neumann boundary condition, are involved, see Appendix \ref{AppSection: Auxiliary transmission problems}. Note that in \cite{Wilke2020}, similar auxiliary problems defined in a vertical cylindrical domain with a horizontal interface instead were analyzed. Our results in Appendix \ref{AppSection: Auxiliary transmission problems} are the extensions to \cite[Appendix 5.3]{Wilke2020}, where a Neumann boundary condition were impose on the vertical surface. 
	
	Let us finally point out that one of our main results obtained is the well-posedness of the nonstationary two-phase Stokes problem in a general cylindrical domain with an outflow boundary condition and a ninety-degree contact angle, as well as the heat equation in both a cylindrical domain and a cylindrical annular domain. Since we considered the general situation of the domain, i.e., smooth hypersurfaces $ S $ and $ \Sigma $, the results can be applied to solve broader nonlinear problems.
	
\subsection{Outline}
In Section \ref{Section: Preliminaries} we briefly introduce some function spaces and reformulate the system. Pullbacks from the deformed configuration to the reference one are discussed in Subsection \ref{Subsection: equivalent Lagrangian}, as well as the well-posedness theorems for the linearized systems in Subsection \ref{Subsection: Linearized system-Lagrangian}. 
Section \ref{Section: Model problems} is devoted to the analysis of the corresponding model problems. The main results of this section are well-posedness of these new model problems with outflow boundary conditions. In particular, the reflection argument is employed to investigate the solvability of a two-phase Stokes equation with boundary contact, as well as two heat equations. 
In Section \ref{Section: cylindrical domain}, we make an observation of additional regularity of the pressure for the first step. Then the reduction argument is carried out for the sake of analysis. Finally we prove the well-posedness of the two-phase Stokes equation in general cylindrical domain by a localization procedure.
Based on the solvability results in the previous section, in Section \ref{Section: Nonlinear well-posedness} the full nonlinear system is shown to be well-posed locally via the Banach fixed-point theorem. 
In addition, the existence of partition of our domain is given in Appendix \ref{AppSection: Partition} and the existence and uniqueness of auxiliary elliptic and parabolic transmission problems are analyzed in Appendix \ref{AppSection: Auxiliary transmission problems}.

\section{Preliminaries} \label{Section: Preliminaries}

\subsection{Function spaces} \label{Subsection: Function spaces}
If $ M \subseteq \bbr^d $, $ d \in \bbn_+ $ is measurable, $ \Lq{q}(M) $, $ 1 \leq q \leq \infty $ denotes the usual Lebesgue space and $ \norm{\cdot}_{\Lq{q}(M)} $ its norm. Moreover, $ \Lq{q}(M; X) $ denotes its vector-valued variant of strongly measurable $ q $-integrable functions, where $ X $ is a Banach space.  

Let $ \OM \subseteq \bbr^n $ be an open and nonempty domain, $ \W{m}(\OM) $, $ \dW{m}(\OM) $ denotes the usual and homogeneous Sobolev space respectively with $ m \in \bbn_+ $ and $ \Lq{q}(\OM) = \W{0}(\OM) $. Furthermore, we set 
\begin{gather*}
	\WOO{m}(\OM) = \overline{C_0^\infty (\OM)}^{\W{m}(\OM)}, \quad
	\W{-m}(\OM) := [ W^m_{q',0}(\OM) ]', \\
	\dot{W}_{q,0}^{m}(\OM) = \overline{C_0^\infty (\OM)}^{\dW{m}(\OM)}, \quad
	\dW{-m}(\OM) := [ \dot{W}^m_{q',0}(\OM) ]',
\end{gather*}
where $ q' $ is the conjugate exponent to $ q $ satisfying $ 1/q + 1/q' = 1 $. 

For $ k,k' \in \bbn $ with $ k < k' $, we consider the standard definition of the Besov spaces by real interpolation of Sobolev spaces (see e.g. Lunardi \cite{Lunardi2018})
\begin{equation*}
	\Bqp{s}(\OM) = \left( \W{k}(\OM), \W{k'}(\OM) \right)_{\theta, p},
\end{equation*}
where $ s = (1 - \theta)k + \theta k',\ \theta \in (0, 1) $. In special case $ q = p $, one obtains the Sobolev-Slobodeckij spaces by
\begin{equation*}
	\W{s}(\OM) := \Bq{s}(\OM) = \left( \W{k}(\OM), \W{k'}(\OM) \right)_{\theta, q}, \text{ if } s \notin \bbn,
\end{equation*}
which is endowed with norm $ \norm{\cdot}_{\W{s}(\OM)} = \norm{\cdot}_{\Lq{q}(\OM)} + \seminorm{\cdot}_{\W{s}(\OM)} $, where
\begin{equation*}
	\seminorm{f}_{\W{s}(\OM)} = \left(
		\int_{\OM} \int_{\OM} \left( \frac{\abs{f(x) - f(y)}}{\abs{x-y}^s} \right)^q \frac{\d x \d y}{\abs{x - y}^n} 
	\right)^{\frac{1}{q}}.
\end{equation*}

For an interval $ I \subseteq \bbr $, $ 0 < r < 1 $, $ 1 \leq q < \infty $, we recall the Banach space-valued Sobolev--Slobodeckij space as
\begin{equation*}
	\W{r}(I; X) = \left\{f \in \Lq{q}(I; X) : \norm{f}_{\W{r}(I; X)} < \infty \right\}
\end{equation*}
with the seminorm denoted by
\begin{equation*}
	\seminorm{f}_{\W{r}(I; X)} = \left(
		\int_{I \times I} \frac{\norm{f(t) - f(\tau)}_{X}^q}{\abs{t - \tau}^{n + rq}} \d t \d \tau
	\right)^{\frac{1}{q}}.
\end{equation*}
To be convenient, for $ T > 0 $ we define the corresponding space with vanishing time trace at $ t = 0 $ as
\begin{equation*}
	\WO{r}(0,T; X) = \left\{f \in \W{r}(0,T; X) : \rv{f}_{t = 0} = 0 \right\}.
\end{equation*}
\subsection{Reformulation in the reference configuration}
\label{Subsection: equivalent Lagrangian}
In this section, we transform the free-boundary fluid-structure problem for plaque growth \FSI\ from deformed configuration to the reference one and state the corresponding theorem. For quantities in different coordinates, define
\begin{equation}
	\label{equalofv}
	\begin{gathered}
		\hv(X,t) = \bv(x,t), \ \hpi(X,t) = \pi(x,t), \ \hsigma(X,t) = \bsigma(x,t), \\
		\hr(X,t) = \rho(x,t), \ \hmu(X) = \mu(x), \ \hnu(X) = \nu(x),
	\end{gathered}
\end{equation}
for all $ x = \vp(X,t), \ X \in \OM$ and $ t \geq 0 $. Then by simple calculations, it follows that
\begin{gather}
	\label{ptu}
	\pt \hat{\mathbf{u}}(X,t) = \left( \pt + \bv(x,t) \cdot \nabla \right) \mathbf{u}(x,t), \\
	\label{grad}
	\nabla \phi = \invtr{\hF} \hnab \hat{\phi}, \quad \nabla \mathbf{u} = \inv{\hF} \hnab \hat{\mathbf{u}}, \\
	\label{div}
	\Div \mathbf{u} = \tr ( \nabla \bu ) = \tr ( \inv{\hF} \hnab \hat{\mathbf{u}} ) = \invtr{\hF} : \hnab \hu, 	
\end{gather}
where $ \phi / \hat{\phi} $ are scalar-valued and $ \mathbf{u} / \hat{\mathbf{u}} $ are vector-valued functions in $ \OM^t / \OM $ respectively.
It is known that the Piola transform establishes a bridge between tensor field defined in deformed and reference configurations, which is
\begin{equation}
	\label{Piola}
	\hT(X,t) = \hJ(X,t) \bsigma(x,t) \invtr{\hF}(X,t), \quad \text{for all}\ x = \vp(X,t), \ X \in \OM,
\end{equation}
where $ \hT $ is the first Piola–Kirchhoff stress tensor. 

For the fluid part, it follows that
\begin{align*}
	\pt \hJf = \tr \left( \inv{\hF} \pt \hF \right) \hJf 
	= \tr \left( \inv{\hF} \hnab \hv \right) \hJf 
	= \Div \bv \hJf 
	= 0,
\end{align*}
which means
\begin{equation}
	\label{Jf=1}
	\hJf = \rv{\hJf}_{t = 0} = \det \bbi = 1, \quad \text{in}\ \Of.
\end{equation}
For the solid part, since the deformation from natural configuration $ \Osg $ to deformed configuration $ \Ost $ conserves mass, incompressibility yields
\begin{equation*}
	\hJse = 1, \quad \text{in}\ \Os,
\end{equation*}
and hence,
\begin{equation*}
	\hJs = \hJsg = \hg^3, \quad \text{in}\ \Os.
\end{equation*}

Now combining formulas \eqref{gLagragian}, \eqref{equalofv}--\eqref{Jf=1}, we rewrite the fluid–structure interaction problem \eqref{Fluid}--\eqref{initial-g} in Lagrangian coordinate and rearrange the equation to obtain the system for $ t \in J := (0, T) $, $ T > 0 $,
\begin{align} 
	\left.
		\begin{aligned}
			\hrf \pt \hvf - \hdiv \bS( \hvf, \hpif ) & = \bK_f \\
			\hdiv \hvf & = G_f
		\end{aligned}
	\right\} & \quad \text{in}\ \Of \times J,  \nonumber \\
	\left.
		\begin{aligned}
			\hrs \pt \hvs - \hdiv \bS( \hvs, \hpis ) = \bar{\bK}_s + \bK_s^g =: \bK_s \\
			\hdiv \hvs - \frac{\gamma \beta}{\hrs} \hcs = G_s
		\end{aligned}
	\right\} & \quad \text{in}\ \Os \times J,  \nonumber \\
	\jump{\hv} = 0, \quad
	\jump{\bS( \hv, \hpi )} \hn_\Sigma = \bH^1 & \quad \text{on}\ \Sigma \times J, \nonumber \\
	\cP_G(\hv) = \bH^2, \quad
	\bS( \hv, \hpi ) \hn_G \cdot \hn_G = H^3 & \quad \text{on}\ G \backslash \Sigma \times J,  \nonumber \\
	\hvs = 0 & \quad \text{on}\ S \times J,  \nonumber \\
	\rv{\hv}_{t = 0} = \vo & \quad \text{in}\  \OM \backslash \Sigma, 	\label{Nonlinear: Lagrangian} \\
	\pt \hcf - \hDf \hDelta \hcf = F_f^1 & \quad \text{in}\ \Of \times J,  \nonumber \\
	\pt \hcs - \hDf \hDelta \hcs = \bar{F}_s^1+ F_s^g =: F_s^1 & \quad \text{in}\ \Os \times J,  \nonumber \\
	\left.
		\begin{aligned}
			\hDf \hnab \hcf \cdot \hn_\Sigma = \hDs \hnab \hcs \cdot \hn_\Sigma + \bar{F}_f^2 =: F_f^2 \\
			\hDs \hnab \hcs \cdot \hn_\Sigma = \zeta \jump{\hc} + \bar{F}_s^2 =: F_s^2
		\end{aligned}
	\right\} & \quad \text{on}\ \Sigma \times J,  \nonumber \\
	\hD \hnab \hc \cdot \hn_G = F^3 & \quad \text{on} \ G \backslash \partial \Sigma \times J,  \nonumber \\
	\hDs \hnab \hcs \cdot \hn_S = F^4 & \quad \text{on}\ S \times J,  \nonumber \\
	\rv{\hc}_{t = 0} = \co & \quad \text{in}\  \OM \backslash \Sigma, \nonumber \\
	\pt \hcss - \beta \hcs = F^5, \quad 
	\pt \hg - \frac{\gamma \beta}{n \hrs} \hcs = F^6 & \quad \text{in}\  \Os \times J,  \nonumber \\
	\rv{\hcss}_{t = 0} = 0, \quad \rv{\hg}_{t = 0} = 1 & \quad \text{in}\  \Os, \nonumber
\end{align}
where $ \bS( \hv, \hpi ) = - \hpi \bbi + \hnu \left( \hnab \hv + \tran{\hnab}\hv \right) $
and
\begin{align}
	& \bK_f = \hdiv \kf, \quad \bar{\bK}_s = \hdiv \ks, \quad \bK_s^g = - \left( \hsigs \invtr{\hFs} \right) \frac{n \hnab \hg}{\hg}, \nonumber \\
	& G = - \left( \FOinvtran \right) : \hnab \hv, \quad \bH^1 = - \jump{\tk} \hn_\Sigma, \nonumber \\
	& \bH^2 = - \left( \bbi - \big((\invtr{\hF} \hn_G) \otimes (\invtr{\hF} \hn_G) - \hn_G \otimes \hn_G \big) \right) \hv, \nonumber \\
	& H^3 = - \hsigma \invtr{\hF} \hn_G \cdot (\invtr{\hF} \hn_G) + \bS(\hv, \hpi) \hn_G \cdot \hn_G, \quad
	F_f^1 = \hdiv \tFf, \label{Nonlinear} \\ 
	& \bar{F}_s^1 = \hdiv \tFs, \quad 
	F_s^g = - \beta \hcs \left( 1 + \frac{\gamma}{\hrs} \hcs \right) - \frac{n \hnab \hg}{\hg} \cdot \left( \hDs \inv{\hFs} \invtr{\hFs} \hnab \hcs \right), \nonumber \\
	& \bar{F}_f^2 = - \jump{\tF} \cdot \hn_\Sigma, \quad \bar{F}_s^2 = - \tFs \cdot \hn_\Sigma, \quad F^3 = \tF \cdot \hn_G, \quad F^4 = -\tFs \cdot \hn_S, \nonumber \\
	& F^5 = - \frac{\gamma \beta}{\hrs} \hcs \hcss, \quad
	F^6 = - \frac{\gamma \beta}{n \hrs} \hcs \left( \hg - 1 \right), \nonumber 
\end{align}
with
\begin{align*}
	& \hsigf = - \hpif \bbi + \hnuf \left( \inv{\hFf} \hnab \hvf + \tran{\hnab} \hvf \invtr{\hFf} \right), \quad
	\hsigs = \hsigse + \hsigsv, \\
	& \hsigse  = - \hpis \bbi + \hmus \left( \frac{1}{(\hg)^2} \hFs \tran{\hFs} - \bbi \right), \quad
	\hsigsv = \hnus \left( \hnab \hvs + \tran{\hnab} \hvs \right) \tran{\hFs}, \\
	& \kf = - \hpif \left( \FfOinvtran \right) 
	+ \nuf \left( \inv{\hFf} \hnab \hvf + \tran{\hnab} \hvf \invtr{\hFf} \right) \left( \FfOinvtran \right) \\
	& \qquad \qquad \qquad \qquad \qquad + \nuf \left( \left(\FfOinv\right) \hnab \hvf + \tran{\hnab} \hvf  \left(\FfOinvtran\right)\right), \\
	& \ks = - \hpis \left(\FsOinvtran\right) + \mus \left( \frac{1}{\hg^2} \left( \FsO \right) + \left( \frac{1}{\hg^2} - 1 \right) \bbi - \left( \FsOinvtran \right) \right), \\
	& \tF = \hD \left( \inv{\hF} \invtr{\hF} - \bbi \right) \hnab \hc.
\end{align*}
For \eqref{Nonlinear: Lagrangian}, we introduce the corresponding function spaces for the solutions
\begin{gather*}
	\bbe_{\hv}(J) := \left\{
		\bv \in \W{1}(J; \Lq{q}(\OM)^3) \cap \Lq{q}(J; \W{2}(\OM \backslash \Sigma)^3):  \jump{\bv} = 0, \ \rv{\bv}_S = 0
	\right\}, \\
	\bbe_{\hpi}(J) := \left\{
	\begin{aligned}
		& \pi \in \Lq{q}(J; \dW{1}(\OM \backslash \Sigma)): \\
		& \qquad \qquad
		\jump{\pi} \in \W{\onehalf - \frac{1}{2q}}(J; \Lq{q}(\Sigma)) \cap \Lq{q}(J; \W{1 - \frac{1}{q}}(\Sigma)), \\
		& \qquad \qquad
		\rv{\pi}_{G} \in \W{\onehalf - \frac{1}{2q}}(J; \Lq{q}(G)) \cap \Lq{q}(J; \W{1 - \frac{1}{q}}(G \backslash \partial \Sigma))
	\end{aligned}
	\right\}, \\
	\bbe_{\hc}(J) := \W{1}(J; \Lq{q}(\OM)) \cap \Lq{q}(J; \W{2}(\OM \backslash \Sigma)), \\
	\bbe_{\hcss}(J) := \W{1}(J; \W{1}(\Os)), \quad \bbe_{\hg}(J) := \W{1}(J; \W{1}(\Os)), \\
	\bbe(J) := \bbe_{\hv} \times \bbe_{\hpi} \times \bbe_{\hc} \times \bbe_{\hcss} \times \bbe_{\hg},
\end{gather*}
as well as the spaces for initial data
\begin{equation*}
	X_{\gamma, \hv} := \W{2 - \frac{2}{q}}(\OM \backslash \Sigma)^3, \quad X_{\gamma, c} := \W{2 - \frac{2}{q}}(\OM \backslash \Sigma), \quad
	X_\gamma := X_{\gamma, \hv} \times X_{\gamma, c}.
\end{equation*}
Then the main result of \eqref{Nonlinear: Lagrangian} reads as follows.
\begin{theorem} \label{Theorem: main}
	Let $ q >5 $, $ \OM \subset \bbr^3 $ be a domain defined as above with $ \Sigma, G, S $ of class $ C^3 $ and $ \partial G $, $ \partial \Sigma $ of $ C^4 $ as well. Given $ (\vo, \co) \in X_\gamma $ satisfying the compatibility conditions 
	\begin{gather*}
		\Div \vo = 0, \ 
		\cP_G(\vo)\big|_G = 0, \ 
		\vo\big|_S = 0, \  
		\jump{\vo} = 0, \ 
		\PSigma \jump{\mu (\nabla \vo + \tran{\nabla} \vo) \nuSigma} = 0, \\
		\zeta \jump{\co} - \hDs \hnab \co_s \cdot \hn_\Sigma = 0, \ 
		\jump{\hD \hnab \co} \cdot \hn_\Sigma = 0, \ 
		\hD \hnab \co \cdot \hn_G \big|_{G \backslash \partial \Sigma} = 0.
	\end{gather*}
	Then there exists a positive $ T_0 = T_0(\|(\vo, \co)\|_{X_\gamma}) < \infty $ such that for $ J := [0,T] $, $ 0 < T < T_0 $, \eqref{Nonlinear: Lagrangian} admits a unique solution $ (\hv, \hpi, \hc, \hcss, \hg) \in \bbe(J) $. 
\end{theorem}
The proof of Theorem \ref{Theorem: main} relies on the Banach fixed-point theorem. More specifically, we need to prove an isomorphism of linear part and construct a contraction mapping from the solutions to the nonlinear data.
\begin{remark}
	We comment here that in \cite{AL2021}, it was shown that the cells concentrations are nonnegative, which is not clear in the present case. The problem comes up with the contact line formulated in the problem, that is, the boundaries for both $ c_f $ and $ c_s $ are not smooth anymore.
\end{remark}

\subsection{Linearization} \label{Subsection: Linearized system-Lagrangian}
Let $ J := (0,T) $ with $ T > 0 $. Now we consider the linear systems separately, namely,

1). \textbf{Two-phase Stokes problem in a cylindrical domain}
\begin{equation} \label{Linearized TwoPhase}
	\begin{alignedat}{3}
		\rho \pt u - \Div \Smu(u, \pi) & = f_u, && \tin \OM \backslash \Sigma \times J, \\
		\Div u & = f_d, && \tin \OM \backslash \Sigma \times J, \\
		\jump{u} & = g_1, && \ton \Sigma \times J, \\
		\jump{- \pi \bbi + \mu (\nabla u + \tran{\nabla} u)} \nuSigma & = g_2, && \ton \Sigma \times J, \\
		\tran{(u_1, u_3)} & = g_3, && \ton G \backslash \partial \Sigma \times J, \\
		- \pi + 2 \mu \ptial{2} u_2 & = g_4, && \ton G \backslash \partial \Sigma \times J, \\
		u & = g_5, && \ton S \times J, \\
		u(0) & = u_0, && \tin \OM \backslash \Sigma.
	\end{alignedat} 
\end{equation}
where $ \Smu(u, \pi) = - \pi \bbi + \mu ( \nabla u + \tran{\nabla} u ) $. $ \rho, \mu > 0 $ are constants representing the density and viscosity respectively. $ \nuSigma $ denotes the unit normal vector on $ \Sigma $, pointing from $ \Of $ to $ \Os $.

For the maximal regularity setting, we define the suitable function spaces for solutions and data with $ q > 3 $, 
\begin{gather*}
	u \in \bbe_u(J) : = \W{1}(J; \Lq{q}(\OM)^3) \cap \Lq{q}(J; \W{2}(\OM \backslash \Sigma)^3), \\
	\pi \in \bbe_\pi(J) := \left\{
		\begin{aligned}
			& \pi \in \Lq{q}(J; \dW{1}(\OM \backslash \Sigma)): \\
			& \qquad \qquad  
			\jump{\pi} \in \W{\onehalf - \frac{1}{2q}}(J; \Lq{q}(\Sigma)) \cap \Lq{q}(J; \W{1 - \frac{1}{q}}(\Sigma)), \\
			& \qquad \qquad  
			\rv{\pi}_{G} \in \W{\onehalf - \frac{1}{2q}}(J; \Lq{q}(G)) \cap \Lq{q}(J; \W{1 - \frac{1}{q}}(G \backslash \partial \Sigma))
		\end{aligned}
	\right\}.
\end{gather*}
The additional regularities for $ \pi $ on $ \Sigma $ and $ G $ come from the Neumann trace of $ u $ on both boundaries. If $ (u, \pi) $ is a solution of \eqref{Linearized TwoPhase}, the necessary regularity classes for the data are subsequently given by
\begin{gather*}
	u_0 \in X_{\gamma, u} := \W{2 - \frac{2}{q}}(\OM \backslash \Sigma)^3, \quad
	f_u \in \bbf_1(J) := \Lq{q}(J; \Lq{q}(\OM)^3), \\
	f_d \in \bbf_2(J) := \W{1}(J; \dW{-1}(\OM)) \cap \Lq{q}(J; \W{1}(\OM \backslash \Sigma)), \\
	g_1 \in \bbf_3(J) := \W{1 - \frac{1}{2q}}(J; \Lq{q}(\Sigma)^3) \cap \Lq{q}(J; \W{2 - \frac{1}{q}}(\Sigma)^3), \\
	g_2 \in \bbf_4(J) := \W{\onehalf - \frac{1}{2q}}(J; \Lq{q}(\Sigma)^3) \cap \Lq{q}(J; \W{1 - \frac{1}{q}}(\Sigma)^3), \\
	g_3 \in \bbf_5(J) := \W{1 - \frac{1}{2q}}(J; \Lq{q}(G)^2) \cap \Lq{q}(J; \W{2 - \frac{1}{q}}(G \backslash \partial \Sigma)^2), \\
	g_4 \in \bbf_6(J) := \W{\onehalf - \frac{1}{2q}}(J; \Lq{q}(G)) \cap \Lq{q}(J; \W{1 - \frac{1}{q}}(G \backslash \partial \Sigma)), \\
	g_5 \in \bbf_7(J) := \W{1 - \frac{1}{2q}}(J; \Lq{q}(S)) \cap \Lq{q}(J; \W{2 - \frac{1}{q}}(S)^3).
\end{gather*}

Then it is necessary to consider the compatibility conditions at $ t = 0 $, that is,
\begin{equation} \label{Compatibility: TwoPhase-initial-cylinder}
	\begin{gathered}
		\Div u_0 = \rv{f_d}_{t = 0}, \quad
		\tran{((u_0)_1, (u_0)_3)}|_G = \rv{g_3}_{t = 0}, \quad
		\rv{u_0}_S = \rv{g_5}_{t = 0}, \\ 
		\jump{u_0} = \rv{g_1}_{t = 0}, \quad
		\PSigma \jump{\mu (\nabla u_0 + \tran{\nabla} u_0) \nuSigma} = \rv{\PSigma g_2}_{t = 0},
	\end{gathered}
\end{equation}
where $ \PSigma := I - \nuSigma \otimes \nuSigma $ denotes the tangential projection of $ \nuSigma $.

Moreover, the following compatibility conditions at the contact lines $ \partial G $ and $ \partial \Sigma $ must be satisfied as well.
\begin{equation} \label{Compatibility: TwoPhase-contact-cylinder}
	\begin{gathered}
		g_3 = \tran{((g_5)_1, (g_5)_3)}, 
		\text{ at } \partial G, \quad
		\jump{g_3} = \tran{((g_1)_1, (g_1)_3)},
		\text{ at } \partial \Sigma, \\
		\begin{aligned}
			(g_2)_1 & = \jump{2 \mu \ptial{1} (g_3)_1 \nuSigma \cdot e_1 + \mu (\ptial{1} (g_3)_2 + \ptial{3} (g_3)_1) \nuSigma \cdot e_3} \\
			& \quad + \jump{g_4 - 2 \mu f_d} \nuSigma \cdot e_1 + \jump{2 \mu (\ptial{1} (g_3)_1 + \ptial{3} (g_3)_2)} \nuSigma \cdot e_1,
			\text{ at } \partial \Sigma,
		\end{aligned} \\
		\begin{aligned}
			(g_2)_3 & = \jump{2 \mu \ptial{3} (g_3)_2 \nuSigma \cdot e_3 + \mu (\ptial{1} (g_3)_2 + \ptial{3} (g_3)_1) \nuSigma \cdot e_1} \\
			& \quad + \jump{g_4 - 2 \mu f_d} \nuSigma \cdot e_3 + \jump{2 \mu (\ptial{1} (g_3)_1 + \ptial{3} (g_3)_2)} \nuSigma \cdot e_3,
			\text{ at } \partial \Sigma.
		\end{aligned}
	\end{gathered}
\end{equation}
\begin{remark}
	For $ f_d $, we assume the regularity $ \W{1}(J; \dW{-1}(\OM)) $, which is due to the divergence equation in \eqref{Linearized TwoPhase} and the regularity $ u \in \W{1}(J; \Lq{q}(\OM)) $. A similar but different condition was employed in Wilke \cite[Section 1.2]{Wilke2020}, where they considered the Dirichlet and pure-slip boundary conditions around the contact line. Such setting gives rise to a compatibility identity for $ f_d $ and data associated with the normal component of velocities on each boundary and the interface. Readers are also referred to e.g. Pr\"uss--Simonett \cite[Section 7.3 and (8.2)]{PS2016}. The difference comes up with the consideration of outflow boundary condition, which is not endowed with the normal component of velocity. In our previous work \cite{AL2021}, the stress-free boundary is considered and hence we also have an additional regularity for $ f_d $ without the hidden compatibility identity, see also Pr\"uss--Simonett \cite[Page 338]{PS2016}.
\end{remark}

Let $ \bbe(J) := \bbe_u(J) \times \bbe_\pi(J) $ and $ \bbf(J) := \Pi_{j = 1}^7 \bbf_j(J) $. Now we give the theorem for the two-phase Stokes problem in a cylindrical domain, whose proof is postponed in Section \ref{Subsection: Localization procedure}.
\begin{theorem} \label{Theorem: Cylinder-TwoPhase}
	Let $ \rho, \mu > 0 $, $ q > 5 $. Assume that $ \OM \subset \bbr^3 $ is the domain defined in Theorem \ref{Theorem: main} with $ \Sigma, G, S $ of class $ C^3 $ and $ \partial G $, $ \partial \Sigma $ of $ C^4 $ as well. Then there exists a unique solution
	\begin{gather*}
		(u, \pi) \in \bbe(J)
	\end{gather*}
	of \eqref{Linearized TwoPhase} if and only if the data are subject to the following regularity and compatibility conditions:
	\begin{enumerate}
		\item $ (f_u, f_d, g_1, g_2, g_3, g_4, g_5) \in \bbf(J) $, 
		\item $ u_0 \in X_{\gamma,u} $, 
		\item compatibility conditions \eqref{Compatibility: TwoPhase-initial-cylinder} and \eqref{Compatibility: TwoPhase-contact-cylinder} hold.
	\end{enumerate}
\end{theorem}

2). \textbf{Heat equation in a cylindrical domain}
\begin{equation} \label{Linearized Heat_f}\\
	\begin{alignedat}{3}
		\pt c_f - D_f \Delta c_f & = f_{c,f}, && \tin \Of \times J, \\
		D_f \nabla c_f \cdot \nu_{G_f \cup \Sigma} & = g_6, && \ton G_f \cup \Sigma \times J, \\
		c_f(0) & = c_{0,f}, && \tin \Of,
	\end{alignedat} 
\end{equation}
where $ D_f > 0 $ is the diffusivity. As above, we need the regularity classes of solutions and data for the maximal regularity setting. More specifically, 
\begin{gather*}
	c_f \in \bbe_{c,f}(J) := \W{1}(J; \Lq{q}(\Of)) \cap \Lq{q}(J; \W{2}(\Of)),
\end{gather*}
with the data
\begin{gather*}
	c_{0,f} \in X_{\gamma, c_f} := \W{2 - \frac{2}{q}}(\Of), \quad
	f_{c,f} \in \bbf_8(J) := \Lq{q}(J; \Lq{q}(\Of)), \\
	g_6 \in \bbf_9(J) := \W{\onehalf - \frac{1}{2q}}(J; \Lq{q}(G_f \cup \Sigma)) \cap \Lq{q}(J; \W{1 - \frac{1}{q}}(G_f \cup \Sigma)).
\end{gather*}
Moreover, concerning the compatibility conditions at $ t = 0 $ and the contact line $ \partial \Sigma \cap G_{1,f} $, for $ q > 3 $ we have
\begin{equation} \label{Compatibility: Cylinder-initial-Heat_f}
	\begin{gathered}
		\rv{D_f \nabla c_{0,f} \cdot \nu_{G_2 \cup \Sigma}}_{G_2 \cup \Sigma} = \rv{g_6}_{t = 0}, \text{ on } G_f \cup \Sigma, \\
		\ptial{\nu_{G_f}}(\rv{g_6}_{\Sigma}) = \ptial{\nuSigma}(\rv{g_6}_{G_f}), \text{ at } \partial \Sigma.
	\end{gathered}
\end{equation}
\begin{theorem} \label{Theorem: Cylinder-Heat_f}
	Let $ D_f > 0 $, $ q > 5 $. Assume that $ \Of \subset \bbr^3 $ is the domain defined in Theorem \ref{Theorem: main} with $ \Sigma, G_f $ of class $ C^3 $ and $ \partial \Sigma $ of $ C^4 $ as well. Then there exists a unique solution
	\begin{equation*}
		c_f \in \bbe_{c,f}(J)
	\end{equation*}
	of \eqref{Linearized Heat_f} if and only if the data are subject to the following regularity and compatibility conditions:
	\begin{enumerate}
		\item $ (f_{c,f}, g_6) \in \bbf_8(J) \times \bbf_9(J) $, 
		\item $ c_{0,f} \in X_{\gamma, c_f} $, 
		\item compatibility condition \eqref{Compatibility: Cylinder-initial-Heat_f} holds.
	\end{enumerate}
\end{theorem}

3). \textbf{Heat equation in a cylindrical ring}
\begin{equation} \label{Linearized Heat_s}
	\begin{alignedat}{3}
		\pt c_s - D_s \Delta c_s & = f_{c,s}, && \tin \Os \times J, \\
		D_s \nabla c_s \cdot \nu_{\partial \Os} & = g_7, && \ton \partial \Os \times J, \\
		c_s(0) & = c_{0,s}, && \tin \Os,
	\end{alignedat} 
\end{equation}
where $ \partial \Os := \Sigma \cup S \cup G_s $, $ D_s > 0 $ is the diffusivity. As in above, we need the regularity classes of solutions and data for the maximal regularity setting. More specifically, 
\begin{gather*}
	c_s \in \bbe_{c,s}(J) := \W{1}(J; \Lq{q}(\Os)) \cap \Lq{q}(J; \W{2}(\Os)),
\end{gather*}
with the data
\begin{gather*}
	c_{0,s} \in X_{\gamma, c_f} := \W{2 - \frac{2}{q}}(\Os), \quad
	f_{c,s} \in \bbf_{10}(J) := \Lq{q}(J; \Lq{q}(\Os)), \\
	g_7 \in \bbf_{11}(J) := \W{\onehalf - \frac{1}{2q}}(J; \Lq{q}(\partial \Os)) \cap \Lq{q}(J; \W{1 - \frac{1}{q}}(\partial \Os)).
\end{gather*}
Moreover, concerning the compatibility conditions at $ t = 0 $, for $ q > 3 $ we have
\begin{equation} \label{Compatibility: Cylinder-initial-Heat_s}
	\begin{gathered}
		\rv{D_f \nabla c_{0,s} \cdot \nu_{\partial \Os}}_{\partial \Os} = \rv{g_7}_{t = 0}, \\
		\ptial{\nu_{G_s}}(\rv{g_7}_{\Sigma}) = \ptial{\nuSigma}(\rv{g_7}_{G_s}), \text{ at } \partial \Sigma, \quad
		\ptial{\nu_{G_s}}(\rv{g_7}_{S}) = \ptial{\nu_S}(\rv{g_7}_{G_s}) \text{ at } \partial S.
	\end{gathered}
\end{equation}

\begin{theorem} \label{Theorem: Cylinder-Heat_s}
	Let $ D_s > 0 $, $ q > 5 $. Assume that $ \Os \subset \bbr^3 $ is the domain defined in Theorem \ref{Theorem: main} with $ \Sigma, G_s, S $ of class $ C^3 $ and $ \partial \Sigma, \partial G $ of $ C^4 $ as well. Then there exists a unique solution
	\begin{gather*}
		c_s \in \bbe_{c,s}(J)
	\end{gather*}
	of \eqref{Linearized TwoPhase} if and only if the data are subject to the following regularity and compatibility conditions:
	\begin{enumerate}
		\item $ (f_{c,s}, g_8) \in \bbf_{10}(J) \times \bbf_{11}(J) $, 
		\item $ c_{0,s} \in X_{\gamma, c_s} $, 
		\item compatibility condition \eqref{Compatibility: Cylinder-initial-Heat_s} holds.
	\end{enumerate}
\end{theorem}

4). \textbf{ODEs in a cylindrical ring}
\begin{equation} \label{Linearized ODE: growth and foam cell}
	\begin{alignedat}{3}
		\pt c_s^* - \beta_1 c_s & = f_{c^*}, && \tin \Os \times J, \\
		\pt g - \beta_2 c_s & = f_g, && \tin \Os \times J, \\
		c_s^*(0) = 0,\  g(0) & = 1, && \tin \Os,
	\end{alignedat} 
\end{equation}
where $ \beta_1 = \beta $, $ \beta_2 = \gamma \beta/ 3 \rho_s $. Since both equations in \eqref{Linearized ODE: growth and foam cell} are ordinary differential equation, if $ c_s $ is the solution of \eqref{Linearized Heat_s}, then one obtains
\begin{gather*}
	c_s^* \in \W{1}(J; \W{1}(\Os)), \quad
	g \in \W{1}(J; \W{1}(\Os)),
\end{gather*}
provided that
\begin{gather*}
	f_{c^*} \in \Lq{q}(J; \W{1}(\Os)), \quad
	f_g \in \Lq{q}(J; \W{1}(\Os)).
\end{gather*}

\section{Model problems} \label{Section: Model problems}

As in \cite{PS2016} and \cite{Wilke2020}, a localization argument will be employed to prove the existence and uniqueness of a solution to \eqref{Linearized TwoPhase}--\eqref{Linearized ODE: growth and foam cell}. Since the linearized decoupled systems are a two-phase Stokes problem, two heat equations and two ODEs, we impose the localization procedure to each system prospectively. To this end, we study nine different types model problems, which are:
\begin{itemize}
	\item the two-phase Stokes equations with a flat interface and without any boundary condition
	\item the full space heat equations (without any boundary or interface conditions)
	\item the Stokes equations with outflow boundary conditions in a half-space and no interface
	\item the Stokes equations with no-slip boundary conditions in a half-space and no interface
	\item the heat equations with Dirichlet or Neumann boundary conditions in a half-space and no interface
	\item the Stokes equations in quarter-spaces with outflow conditions on one part of the boundary and no-slip boundary conditions on the other part
	\item the two-phase Stokes equations with outflow boundary conditions, a flat interface and a contact angle of 90 degrees in a half-space
	\item the heat equations in quarter-spaces with Neumann boundary conditions.
\end{itemize}
\begin{remark}
	As the first five problems are well understood, see e.g. Pr\"uss--Simonett \cite{PS2016}, we focus on the rest three ones. Note that in \cite{Wilke2020}, Wilke studied model problems for the Stokes equations with pure slip boundary conditions in a quarter-space, as well as for the two-phase Stokes equations with pure slip boundary conditions in a half-space, our consideration expends his results to such model problems with outflow boundary conditions, which is reasonable if we cut off the human vessels. 
\end{remark}


\subsection{The Stokes equations in quarter-spaces} \label{Section: Stokes-quarter}
For convenience, define $ \OM := \bbr \times \bbr_+ \times \bbr_+ $, $ G := \bbr \times \{ 0 \} \times \bbr_+ $ and $ S := \bbr \times \bbr_+ \times \{ 0 \} $. Now let us firstly consider the system
\begin{equation}\label{StokesEqs: quarter}
	\begin{alignedat}{3}
		\rho \pt u - \Div \Smu(u, \pi) & = f, && \tin \OM \times J, \\
		\Div u & = f_d, && \tin \OM \times J, \\
		\tran{(u_1, u_3)} & = 0, && \ton G \times J, \\
		- \pi + 2 \mu \ptial{2} u_2 & = g_2, && \ton G \times J, \\
		u & = g_3, && \ton S \times J, \\
		u(0) & = 0, && \tin \OM \times J,
	\end{alignedat}
\end{equation}
where $ \Smu(u, \pi) = - \pi \bbi + \mu ( \nabla u + \tran{\nabla} u ) $. Then we have the following theorem related to \eqref{StokesEqs: quarter}.
\begin{theorem} \label{Theorem: MP-Stokes-quarterspace}
	Let $ q > 5 $, $ T > 0 $, $ \rho, \mu > 0 $ and $ J = (0,T) $. Assume that $ G, S \in C^3 $ and $ \partial G $ is of class $ C^4 $. Then there exists a unique solution 
	\begin{gather*}
		u \in \WO{1}(J; \Lq{q}(\OM)^3) \cap \Lq{q}(J; \W{2}(\OM)^3), \\
		\pi \in \Lq{q}(J; \dW{1}(\OM)), \quad
		\rv{\pi}_G \in \W{\onehalf - \frac{1}{2q}}(J; \Lq{q}(G)) \cap \Lq{q}(J; \W{1 - \frac{1}{q}}(G))
	\end{gather*}
	of \eqref{StokesEqs: quarter} if and only if the data satisfy the following regularity and compatibility conditions:
	\begin{enumerate}
		\item $ f \in \Lq{q}(J; \Lq{q}(\OM)^3) $,
		\item $ f_d \in \WO{1}(J; \dW{-1}(\OM)) \cap \Lq{q}(J; \W{1}(\OM)) $,
		\item $ g_2 \in \W{\onehalf - \frac{1}{2q}}(J; \Lq{q}(G)) \cap \Lq{q}(J; \W{1 - \frac{1}{q}}(G)) $,
		\item $ g_3 \in \WO{1 - \frac{1}{2q}}(J; \Lq{q}(S)^3) \cap \Lq{q}(J; \W{2 - \frac{1}{q}}(S)^3) $,
		\item $ \tran{((g_3)_1, (g_3)_3)}|_{\Bar{G} \cap \Bar{S}} = 0 $. 
	\end{enumerate}
\end{theorem}
\begin{proof}
To deal with the well-posedness of the Stokes equations in quarter-spaces, we extend the data suitably to the half-space and solve the Stokes equations in half space. 
For this purpose, we firstly extend the function
\begin{equation*}
	f_d \in \WO{1}(J; \dW{-1}(\OM)^3) \cap \Lq{q}(J; \W{1}(\OM)^3)
\end{equation*}
with respect to $ x_3 $ by
\begin{equation} \label{Extension:Sobolev}
	\tf_d (t, x_1, x_2, x_3) = \left\{
		\begin{aligned}
			& f_d (t, x_1, x_2, x_3), & \text{ if } x_3 > 0, \\
			& - f_d (t, x_1, x_2, -2 x_3) + 2 f_d (t, x_1, x_2, - x_3 / 2), & \text{ if } x_3 < 0.
		\end{aligned}
	\right.
\end{equation}
to
\begin{equation*}
	\tf_d \in \WO{1}(J; \dW{-1}(\bbr \times \bbr^+ \times \bbr)^3) \cap \Lq{q}(J; \W{1}(\bbr \times \bbr^+ \times \bbr)^3).
\end{equation*}
By the same extension, we have the extended functions
\begin{gather*}
	\tf \in \Lq{q}(J; \Lq{q}(\bbr \times \bbr^+ \times \bbr)^3), \\
	\tg_2 \in \W{\onehalf - \frac{1}{2q}}( J; \Lq{q}(\bbr \times \{0\} \times \bbr) ) \cap \Lq{q}( J; \W{1 - \frac{1}{q}}(\bbr \times \{0\} \times \bbr) ). 
\end{gather*}
Now one solves the auxiliary half-space Stokes problem
\begin{equation}\label{StokesEqs: half-space_g}
	\begin{alignedat}{3}
		\rho \pt u - \Div \Smu(u, \pi)\mu & = \tf, && \tin \bbr \times \bbr_+ \times \bbr \times J, \\
		\Div u & = \tf_d, && \tin \bbr \times \bbr_+ \times \bbr \times J, \\
		\tran{(u_1, u_3)} & = 0, && \ton \bbr \times \{0\} \times \bbr \times J, \\
		- \pi + 2 \mu \ptial{2} u_2 & = \tg_2, && \ton \bbr \times\{0\} \times \bbr \times J, \\
		u(0) & = 0, && \tin \bbr \times \bbr_+ \times \bbr,
	\end{alignedat}
\end{equation}
and obtains a unique solution (see e.g. \cite[Theorem 6.1]{BP2007} or \cite[Theorem 7.2.1]{PS2016})
\begin{gather*}
	\tu \in \WO{1}(J; \Lq{q}(\bbr \times \bbr_+ \times \bbr)^3) \cap \Lq{q}(J; \W{2}(\bbr \times \bbr_+ \times \bbr)^3), \\
	\tpi \in \Lq{q}(J; \dW{1}(\bbr \times \bbr_+ \times \bbr)), \quad
	\rv{\tpi}_{x_2 = 0} \in \W{\onehalf - \frac{1}{2q}}(J; \Lq{q}(\bbr^2)) \cap \Lq{q}(J; \W{1 - \frac{1}{q}}(\bbr^2)).
\end{gather*}
If $ (u, \pi) $ is a solution of \eqref{StokesEqs: quarter}, then the function $ (u - \tu, \pi - \tpi) $ solves 
\begin{equation}\label{StokesEqs: quarter-reduced}
	\begin{alignedat}{3}
		\rho \pt u - \Delta u + \nabla \pi & = 0, && \tin \OM \times J, \\
		\Div u & = 0, && \tin \OM \times J, \\
		\tran{(u_1, u_3)} & = 0, && \ton G \times J, \\
		- \pi + 2 \mu \ptial{2} u_2 & = 0, && \ton G \times J, \\
		u & = g_3, && \ton S \times J, \\
		u(0) & = 0, && \tin \OM,
	\end{alignedat}
\end{equation}
with a modified data $ g_3 $ (not to be relabeled) in the right regularity classes having a vanishing trace at $ t = 0 $ and satisfying $ \tran{((g_3)_1, (g_3)_3)}|_{\Bar{G} \cap \Bar{S}} = 0 $. Note that the odd reflection does work for functions in both $ \WOO{1}(\bbr_+^n) $ and $ \W{2}(\bbr_+^n) \cap \WOO{1}(\bbr_+^n) $, the real interpolation implies that the function in $ \W{s}(\bbr_+^n) \cap \WOO{1}(\bbr_+^n) $ can be extended to $ \W{s}(\bbr^n) $, $ 1 < s < 2 $, as well. 
It follows from $ \Div u = 0 $ that $ \ptial{2} (g_3)_2 = \ptial{2} u_2 = - \ptial{1} u_1 - \ptial{3} u_3 = 0 $ on $ \Bar{G} \cap \Bar{S} $. 
Thus, we extend $ ((g_3)_1, (g_3)_3) $ via odd reflection and $ (g_3)_2 $ via even reflection along $ x_2 $ to 
\begin{equation*}
	\tg_3 \in \WO{1 - \frac{1}{2q}}(J; \Lq{q}(\bbr^2 \times \{0\})^3) \cap \Lq{q}(J; \W{2 - \frac{1}{q}}(\bbr^2 \times \{0\})^3).
\end{equation*}
By Bothe--Pr\"uss \cite[Theorem 6.1]{BP2007} or Pr\"uss--Simonett \cite[Theorem 7.2.1]{PS2016}, the half-space Stokes problem
\begin{equation}\label{StokesEqs: half-space_f}
	\begin{alignedat}{3}
		\rho \pt \baru - \Delta \baru + \nabla \barpi & = 0, && \tin \bbr \times \bbr \times \bbr_+ \times J, \\
		\Div \baru & = 0, && \tin \bbr \times \bbr \times \bbr_+ \times J, \\
		\baru & = \tg_3, && \ton \bbr \times \bbr \times \{0\} \times J, \\
		\baru(0) & = 0, && \tin \bbr \times \bbr \times \bbr_+,
	\end{alignedat}
\end{equation}
admits a unique solution satisfying
\begin{gather*}
	\baru \in \WO{1}(J; \Lq{q}(\bbr \times \bbr \times \bbr_+)^3) \cap \Lq{q}(J; \W{2}(\bbr \times \bbr \times \bbr_+)^3), \\
	\barpi \in \Lq{q}(J; \dW{1}(\bbr \times \bbr \times \bbr_+)).
\end{gather*}
With symmetry, we conclude that 
\begin{gather*}
	(\cu, \cpi)(t, x_1, x_2, x_3) := (\tran{(- \hatu_1, \hatu_2, - \hatu_3)}, - \hpi)(t, x_1, - x_2, x_3)
\end{gather*}
is a solution pair to \eqref{StokesEqs: half-space_f} as well. It follows from the uniqueness that 
\begin{equation*}
	\cu(t, x_1, x_2, x_3) = \hatu(t, x_1, x_2, x_3), \quad
	\cpi(t, x_1, x_2, x_3) := \hpi(t, x_1, x_2, x_3)
\end{equation*}
and this yields
\begin{equation*}
	(\hatu_1, \hatu_3)(t, x_1, 0, x_3) = 0, \quad
	\ptial{2} \hatu_2(t, x_1, 0, x_3) = 0, \quad
	\hpi(t, x_1, 0, x_3) = 0.
\end{equation*}
Therefore, the restricted pair $ (u, \pi) := (\tu + \baru, \tpi + \barpi) $ is the desired solution to \eqref{StokesEqs: quarter}. 

Now we give a brief proof of the uniqueness. To this end, for a solution $ u $ to \eqref{StokesEqs: quarter} with data all reduced in a quarter-space, one reflects it by suitable extension operators with respect to $ x_2 $ to a function $ \tu $ with $ \rv{\tu}_{x_2 > 0} = u $ in a half-space. Note here reflecting along $ x_3 $ is also an option without a difference. Then by a standard energy estimate one concludes the uniqueness of $ \tu $, which implies that of $ u $. This completes the proof.
\end{proof}

\subsection{The Stokes equations in bent quarter-spaces} \label{Section: Stokes-Bent-Quarter}
Now we consider the case of bent quarter spaces. Let $ \theta \in BC^2(\bbr^2) $ and $ x = (x', x_3) $, $ x' = (x_1, x_2) $ such that 
\begin{equation*}
	\OM_\theta := \left\{  x \in \bbr^3: x_2 > 0, x_3 > \theta(x') \right\}.
\end{equation*}
Let $ \nabla_{x'} = \tran{(\ptial{1}, \ptial{2})} $. Assume that $ \norm{\nabla_{x'} \theta}_\infty \leq \eta $ and $ \norm{\nabla_{x'}^2 \theta}_\infty \leq M $, $ M > 0 $, where $ \eta > 0 $ may be chosen as small as we need. Define 
\begin{equation*}
	\Stheta := \left\{ x \in \bbr^3: x_3 = \theta(x'), x_2 > 0 \right\}, \quad 
	\Gtheta := \left\{ x \in \bbr^3: x_3 > \theta(x'), x_2 = 0 \right\}.
\end{equation*}
Moreover, denote the unit outer normal vector on $ \Stheta $ by
\begin{equation*}
	\nuStheta := \frac{1}{\sqrt{1 + \vert \nabla_{x'} \theta (x') \vert^2}} \tran{(\nabla_{x'} \theta(x'), -1)}.
\end{equation*}
Now we consider the problem
\begin{equation}\label{StokesEqs: Bent-quarter}
	\begin{alignedat}{3}
		\rho \pt u - \Div \Smu(u, \pi) & = f, && \tin \OM_\theta \times J, \\
		\Div u & = f_d, && \tin \OM_\theta \times J, \\
		\tran{(u_1, u_3)} & = g_1, && \ton \Gtheta \times J, \\
		- \pi + 2 \mu \ptial{2} u_2 & = g_2, && \ton \Gtheta \times J, \\
		u & = g_3, && \ton \Stheta \times J, \\
		u(0) & = u_0, && \tin \OM_\theta,
	\end{alignedat}
\end{equation}
where $ \Smu(u, \pi) = - \pi \bbi + \mu ( \nabla u + \tran{\nabla} u ) $ and $ \rho, \mu > 0 $ are given constants. For given data $ (f, f_d, g_1, g_2, g_3, u_0) $, the following compatibility conditions hold:
\begin{gather}
	\Div u_0 = \rv{f_d}_{t = 0}, \quad
	\rv{u_0}_{\Stheta} = \rv{g_3}_{t = 0}, \quad
	\tran{((u_0)_1, (u_0)_3)} = \rv{g_1}_{t = 0}, \label{initialComp: Bent Stokes} \\
	g_1 = \tran{((g_3)_1, (g_3)_3)}, \text{ at the contact line } \Bar{\Gtheta} \cap \Bar{\Stheta}. \label{ContactComp: Bent Stokes} 
\end{gather}

\subsubsection{Reduction}
\label{subsubsection:reduction-bent-quarter}
To solve \eqref{StokesEqs: Bent-quarter}, we shall reduce it to the case $ (u_0, f, g_1, g_2) = 0 $. For this purpose, it is possible to extend $ u_0 $ and $ f $ to some $ \tu_0 \in \W{2 - 2/q}(\bbr^3)^3 $ and $ \tf \in \Lq{q}(J; \Lq{q}(\bbr^3)^3) $ respectively (for example, extend them to $ \bbr \times \bbr_+ \times \bbr $ by the extension \eqref{Extension:x_3theta} below and then extend them to $ \bbr^3 $ by a standard extension operator). Then solving the full-space heat equation
\begin{equation*}
	\begin{alignedat}{3}
		\rho \pt u - \mu \Delta u & = f, && \tin \bbr^3 \times J, \\
		u(0) & = \tu_0, && \tin \bbr^3,
	\end{alignedat}
\end{equation*}
yields a solution
\begin{equation*}
	\tu \in \W{1}(J; \Lq{q}(\bbr^3)^3) \cap \Lq{q}(J; \W{1}(\bbr^3)^3).
\end{equation*}
For $ \tg_1 := g_1 - \tran{((\tu)_1, (\tu)_3)}|_{\Gtheta} $, $ \tg_2 := g_2 - \rv{2 \mu \ptial{2} \tu_2}_{\Gtheta} $ and $ \tf_d := f_d - \Div \tu $, it holds that $ \rv{\tg_1}_{t = 0} = 0 $ and $ \tf_d |_{t = 0} = 0 $ from \eqref{initialComp: Bent Stokes}. Analogously, one can suitably extend them to some function
\begin{gather*}
	\hg_1 \in \WO{1 - \frac{1}{2q}}(J; \Lq{q}(\bbr \times \{0\} \times \bbr)^2) \cap \Lq{q}(J; \W{2 - \frac{1}{q}}(\bbr \times \{0\} \times \bbr)^2), \\
	\hg_2 \in \W{\onehalf - \frac{1}{2q}}(J; \Lq{q}(\bbr \times \{0\} \times \bbr)^2) \cap \Lq{q}(J; \W{1 - \frac{1}{q}}(\bbr \times \{0\} \times \bbr)^2), \\
	\hatf_d \in \WO{1}(J; \dW{-1}(\bbr \times \bbr_+ \times \bbr)) \cap \Lq{q}(J; \W{1}(\bbr \times \bbr_+ \times \bbr)).
\end{gather*} 
Then we solve the half-space Stokes equation with outflow boundary condition
	\begin{alignat*}{3}
		\rho \pt u - \Smu (u, \pi) & = 0, && \tin \bbr \times \bbr_+ \times \bbr \times J, \\
		\Div u & = \hatf_d, && \tin \bbr \times \bbr_+ \times \bbr \times J, \\
		\tran{(u_1, u_3)} & = \hg_1, && \ton \bbr \times \{0\} \times \bbr \times J, \\
		- \pi + 2 \mu \ptial{2} u_2 & = \hg_2, && \ton \bbr \times \{0\} \times \bbr \times J, \\
		u(0) & = 0, && \tin \bbr \times \bbr_+ \times \bbr,
	\end{alignat*}
and obtain a unique solution $ (\hatu, \hpi) $ satisfying
\begin{gather*}
	\hatu \in \WO{1}(J; \Lq{q}(\bbr \times \bbr_+ \times \bbr)^3) \cap \Lq{q}(J; \W{1}(\bbr \times \bbr_+ \times \bbr)^3),\\
	\hpi \in \Lq{q}(J; \dW{1}(\bbr \times \bbr_+ \times \bbr)^3), \quad
	\rv{\hpi}_{x_2 = 0} \in \W{\onehalf - \frac{1}{2q}}(J; \Lq{q}(\bbr^2)) \cap \Lq{q}(J; \W{1 - \frac{1}{q}}(\bbr^2)),
\end{gather*}
by Bothe--Pr\"uss \cite[Theorem 6.1]{BP2007} or Pr\"uss--Simonett \cite[Theorem 7.2.1]{PS2016}. If $ (u, \pi) $ is a solution of \eqref{StokesEqs: Bent-quarter}, then $ (\cu, \cpi) := (u - \tu - \hatu, \pi - \hpi) $ solves \eqref{StokesEqs: Bent-quarter} with reduced data $ (u_0, f, f_d, g_1, g_2) = 0 $ and a modified function $ g_3 $ (not to be relabeled) in the right regularity classes satisfying the compatibility conditions $ \rv{g_3}_{t = 0} = 0 $ and $ \tran{((g_3)_1, (g_3)_3)} = 0 $ at the contact line $ \Bar{\Gtheta} \cap \Bar{\Stheta} $ by \eqref{ContactComp: Bent Stokes}.

\subsubsection{Transform to a quarter space} \label{subsubsection:transform-quarter}
In this part, we transform the boundaries $ \Stheta $ and $ \Gtheta $ to $ S := \bbr \times \bbr_+ \times \{0\} $ and $ G := \bbr \times \{0\} \times \bbr_+ $, respectively, and hence, $ \OM_\theta $ to $ \OM := \bbr \times \bbr_+^2 $. To this end, we introduce the new variables $ \barx := (\barx_1, \barx_2, \barx_3) $ where $ \barx_1 = x_1 $, $ \barx_2 = x_2 $ and $ \barx_3 = x_3 - \theta(x') $ for $ x \in \OM_\theta $. For $ (u, \pi) $ a solution to \eqref{StokesEqs: Bent-quarter}, set
\begin{equation*}
	\baru(t, \barx) := u(t, \barx_1, \barx_2, \barx_3 + \theta(\barx')), \quad
	\barpi(t, \barx) := \pi(t, \barx_1, \barx_2, \barx_3 + \theta(\barx')).
\end{equation*}
for $ t \in J, \barx \in \OM $. In the same way we transform the data $ (f_d, g_2, g_3) $ to $ (\bar{f}_d, \bar{g}_2, \bar{g}_3) $. Note that
\begin{equation} \label{transform-relation-low}
	\begin{gathered}
		\nabla \theta  = \tran{(\nabla_{x'} \theta, 0)}, \quad
		\nabla u  = \nabla \baru - \ptial{\barx_3} \baru \otimes \nabla \theta, \\
		\Div u  = \Div \baru - \nabla \theta \cdot \ptial{\barx_3} \baru, \quad
		\Div T = \Div \Bar{T} - \ptial{\barx_3} \Bar{T} \nabla \theta,
	\end{gathered}
\end{equation}
where $ u $, $ \baru $ are vector-valued and $ T $, $ \Bar{T} $ are tensor-valued functions. Then
\begin{equation} \label{transform-relation-high}
	\begin{aligned}
		\Delta u & = \Div \nabla u = \Div(\nabla \baru - \ptial{\barx_3} \baru \otimes \nabla \theta) \\
		& = \Delta \baru - \ptial{\barx_3} \nabla \baru \nabla \theta - \Div(\ptial{\barx_3} \baru \otimes \nabla \theta)
		+ \ptial{\barx_3} (\ptial{\barx_3} \baru \otimes \nabla \theta) \nabla \theta \\
		& = \Delta \baru - \ptial{\barx_3} \nabla \baru \nabla \theta
		- (\nabla \ptial{\barx_3} \baru) \nabla \theta + \Delta \theta \ptial{\barx_3} \baru
		+ (\ptial{\barx_3}^2 \baru \otimes \nabla \theta) \nabla \theta.
	\end{aligned}
\end{equation}
In the following, we will write $ \ptial{\barx_3} $ as $ \ptial{3} $ for the sake of readability. Consequently, for $ (\baru, \barpi) $, we have the problem
\begin{equation}\label{StokesEqs: Bent-quarter_transformed}
	\begin{alignedat}{3}
		\rho \pt \baru - \mu \Delta \baru + \nabla \barpi & = M_1(\theta, \baru, \barpi), && \tin \OM \times J, \\
		\Div \baru & = M_2(\theta, \baru), && \tin \OM \times J, \\
		\tran{(\baru_1, \baru_3)} & = 0, && \ton G \times J, \\
		- \barpi + 2 \mu \ptial{2} \baru_2 & = M_3(\theta, \baru), && \ton G \times J, \\
		\baru & = \bar{g}_3, && \ton S \times J, \\
		\baru(0) & = 0, && \tin \OM,
	\end{alignedat}
\end{equation}
where $ M_1 $, $ M_2 $ and $ M_3 $ are given by
\begin{align*}
	M_1(\theta, \baru, \barpi) & := \mu \big( - 2 \ptial{3} \nabla \baru \nabla \theta
		+ \Delta \theta \ptial{3} \baru
		+ (\ptial{3}^2 \baru \otimes \nabla \theta) \nabla \theta \big)
		+ \nabla \theta(\barx') \ptial{3} \barpi, \\
	M_2(\theta, \baru) & := \nabla \theta(\barx') \cdot \ptial{3} \baru, \quad
	M_3(\theta, \baru) := 2 \mu \ptial{2} \theta(\barx') \ptial{3} \baru_2.
\end{align*}

\subsubsection{Existence and uniqueness} \label{Bent Stokes: Existence and uniqueness}
Notice that \eqref{StokesEqs: Bent-quarter_transformed} can be seen as a perturbation of \eqref{StokesEqs: quarter} if data are in the right regularity classes and $ \norm{\nabla \theta}_\infty \leq \eta $, where $ \eta > 0 $ is small enough. Since Theorem \ref{Theorem: MP-Stokes-quarterspace} holds, we are going to apply the Neumann series argument to \eqref{StokesEqs: Bent-quarter_transformed}. To this end, we employ the similar regularity classes for $ (\baru, \barpi) $ and given data as in Wilke \cite[Section 1.3.2]{Wilke2020}, namely, define
\begin{equation*}
	\bbeo(J) := \bbeo_u(J) \times \bbe_\pi(J),
\end{equation*}
where
\begin{align*}
	& \bbeo_u(J) := \left\{ u \in \WO{1}(J; \Lq{q}(\OM)^3) \cap \Lq{q}(J; \W{2}(\OM)^3): \rv{\tran{(u_1, u_3)}}_{G} = 0 \right\}, \\
	& \bbe_\pi(J) := \left\{ \pi \in \Lq{q}(J; \dW{1}(\OM)): \rv{\pi}_{G} \in \W{\onehalf - \frac{1}{2q}}(J; \Lq{q}(G)) \cap \Lq{q}(J; \W{1 - \frac{1}{q}}(G)) \right\}, 
\end{align*}
and 
\begin{equation*}
	\widetilde{\bbf}(J) := \bbf_1(J) \times \bbfo_2(J) \times \bbfo_3(J) \times \bbf_4(J),
\end{equation*}
where
\begin{gather*}
	\bbf_1(J) := \Lq{q}(J; \Lq{q}(\OM)^3), \quad
	\bbfo_2(J) := \WO{1}(J; \dW{-1}(\OM)) \cap \Lq{q}(J; \W{1}(\OM)), \\
	\bbfo_3(J) := \WO{\onehalf - \frac{1}{2q}}(J; \Lq{q}(G)) \cap \Lq{q}(J; \W{1 - \frac{1}{q}}(G)), \\
	\bbf_4(J) := \W{1 - \frac{1}{2q}}(J; \Lq{q}(S)^3) \cap \Lq{q}(J; \W{2 - \frac{1}{q}}(S)^3).
\end{gather*}
In addition, let
\begin{align*}
	\bbfo(J) := \left\{ (f, f_d, g_2, g_3) \in \widetilde{\bbf}(J): \tran{((g_3)_1, (g_3)_3)} = 0 \text{ at the contact line } \Bar{G} \cap \Bar{S} \right\}.
\end{align*}
Now define an operator $ L : \bbeo(J) \rightarrow \bbfo(J) $ by
\begin{equation*}
	L(\baru, \barpi) := \left[ 
		\begin{gathered}
			\rho \pt \baru - \mu \Delta \baru + \nabla \barpi \\
			\Div \baru \\
			\rv{- \barpi + 2 \mu \ptial{2} \baru_2}_G \\
			\rv{\baru}_S
		\end{gathered}
	\right],
\end{equation*}
and hence, $ L $ is an isomorphism by Theorem \ref{Theorem: MP-Stokes-quarterspace}. For the right-hand side of \eqref{StokesEqs: Bent-quarter_transformed}, we set
\begin{align*}
	M(\theta, \baru, \barpi)
	:= \tran{(M_1(\theta, \baru, \barpi), M_2(\theta, \baru), M_3(\theta, \baru), 0)}
\end{align*}
and 
\begin{align*}
	F := \tran{(0, 0, 0, f_3)}, \quad f_3 := \bar{g}_3.
\end{align*}
It is clear that $ F \in \widetilde{\bbf}(J) $, since $ \theta \in C^3(\bbr^2) $. Noticing that $ \tran{((\bar{g}_3)_1, (\bar{g}_3)_3)} = 0 $, one obtains $ F \in \bbfo(J) $.

To proceed with the Neumann series argument, we focus on the perturbation term $ M(\theta, \baru, \barpi) $. It can be verified easily that $ M(\theta, \baru, \barpi) \in \bbfo(J) $ for each $ (\baru, \barpi) \in \bbeo(J) $, combining the smoothness of $ \theta $. The only point that needs to be taken care of is $ M_2(\theta, \baru) \in \WO{1}(J; \dW{-1}(\OM)) $. By integration by parts with respect to the $ \barx_3 $, we have
\begin{equation} \label{M2-W^-1}
	\int_\OM M_2(\theta, \baru) \phi \d \barx
	= - \int_\OM \nabla \theta \cdot \baru \ptial{3} \phi \d \barx, \quad 
	\text{for all } \phi \in W_{q',0}^1(\OM),
\end{equation}
which yields the claim with $ \baru \in \bbeo_u(J) $.

Then we rewrite \eqref{StokesEqs: Bent-quarter_transformed} as
\begin{align*}
	(\baru, \barpi) = L^{-1} M(\theta, \baru, \barpi) + L^{-1} F.
\end{align*}
Our aim in the following is to show that for each $ \varepsilon > 0 $, there exist $ T_0 > 0 $ and $ \eta_0 > 0 $ such that
\begin{equation} \label{M-epsilon}
	\norm{M(\theta, \baru, \barpi)}_{\bbfo(J)} \leq \varepsilon \norm{(\baru, \barpi)}_{\bbeo(J)},
\end{equation}
for $ 0 < T < T_0 $ and $ 0 < \eta < \eta_0 $.

Note that
\begin{equation*}
	\WO{1}(J; \Lq{q}(\OM)) \cap \Lq{q}(J; \W{2}(\OM))
	\hookrightarrow \WO{\onehalf}(J; \W{1}(\OM))
	\hookrightarrow \Lq{2q}(J; \W{1}(\OM)),
\end{equation*}
valid for every $ q > 1 $ and the embedding constant does not depend on $ T > 0 $, since $ \rv{\baru}_{t = 0} = 0 $. Then direct calculation yields
\begin{equation}\label{M_1-epsilon}
	\begin{aligned}
		& \norm{M_1(\theta, \baru, \barpi)}_{\bbf_1(J)} \\
		& \ \leq C \norm{\nabla \theta}_\infty \norm{(\baru, \barpi)}_{\bbeo(J)} + \norm{\nabla \theta}_\infty^2 \norm{\baru}_{\bbeo_u(J)} + \norm{\nabla^2 \theta}_\infty \norm{\nabla \baru}_{\Lq{q}(J; \Lq{q}(\OM))} \\
		& \ \leq C \norm{\nabla \theta}_\infty \norm{(\baru, \barpi)}_{\bbeo(J)} + \norm{\nabla \theta}_\infty^2 \norm{\baru}_{\bbeo_u(J)} + T^{\frac{1}{2q}} \norm{\nabla^2 \theta}_\infty \norm{\nabla \baru}_{\Lq{2q}(J; \Lq{q}(\OM))} \\
		& \ \leq C ( \norm{\nabla \theta}_\infty (1 + \norm{\nabla \theta}_\infty) + T^{\frac{1}{2q}} \norm{\nabla^2 \theta}_\infty ) \norm{(\baru, \barpi)}_{\bbeo(J)},
	\end{aligned}
\end{equation}
where $ C > 0 $ does not depend on $ T > 0 $.
For $ M_2(\theta, \baru) = \nabla \theta(\barx') \cdot \ptial{3} \baru $, we have
\begin{align*}
	\norm{M_2(\theta, \baru)}_{\bbfo_2(J)}
	= \norm{M_2(\theta, \baru)}_{\Lq{q}(J; \W{1}(\OM))} 
	+ \norm{M_2(\theta, \baru)}_{\WO{1}(J; \dW{-1}(\OM))}.
\end{align*}
The first term in the right-hand side above can be obtained by
\begin{align*}
	& \norm{M_2(\theta, \baru, \barpi)}_{\Lq{q}(J; \W{1}(\OM))} \\
	& \quad \leq C \norm{\nabla \theta}_\infty \norm{\baru}_{\bbeo_u(J)} + \norm{\nabla^2 \theta}_\infty \norm{\baru}_{\Lq{q}(J; \W{1}(\OM))} \\
	& \quad \leq C \norm{\nabla \theta}_\infty \norm{\baru}_{\bbeo_u(J)} + T^{\frac{1}{2q}} \norm{\nabla^2 \theta}_\infty \norm{\baru}_{\Lq{2q}(J; \W{1}(\OM))} \\
	& \quad \leq C ( \norm{\nabla \theta}_\infty + T^{\frac{1}{2q}} \norm{\nabla^2 \theta}_\infty ) \norm{\baru}_{\bbeo_u(J)}
\end{align*}
similarly to $ M_1 $, while the second term follows from \eqref{M2-W^-1} that
\begin{align*}
	\norm{M_2(\theta, \baru, \barpi)}_{\WO{1}(J; \dW{-1}(\OM))}
	\leq C \norm{\nabla \theta}_\infty \norm{\baru}_{\bbeo_u(J)} .
\end{align*}
Then
\begin{equation} \label{M_2-epsilon}
	\norm{M_2(\theta, \baru)}_{\bbfo_2(J)}
	\leq C ( \norm{\nabla \theta}_\infty + T^{\frac{1}{2q}} \norm{\nabla^2 \theta}_\infty ) \norm{\baru}_{\bbeo_u(J)}.
\end{equation}
For $ M_3(\theta, \baru) = 2 \mu \ptial{2} \theta(\barx') \ptial{3} \baru_2 $, one may need to verify both $ \WO{1/2 - 1/2q}(J; \Lq{q}(G)) $ and $ \Lq{q}(J; \W{1 - 1/q}(G)) $. To this end, we proceed by the definition of Sobolev--Slobodeckij space that
\begin{equation*}
	\norm{M_3(\theta, \baru)}_{\WO{\onehalf - \frac{1}{2q}}(J; \Lq{q}(G))} 
	\leq C \norm{\nabla \theta}_\infty \norm{\nabla \baru}_{\WO{\onehalf - \frac{1}{2q}}(J; \Lq{q}(G))}
	\leq C \norm{\nabla \theta}_\infty \norm{\baru}_{\bbeo_u(J)}.
\end{equation*}
With the help of Sobolev trace theory, one derives that
\begin{align*}
	\norm{M_3(\theta, \baru)}_{\Lq{q}(J; \W{1 - \frac{1}{q}}(G))} 
	\leq C \norm{\nabla \theta}_\infty \norm{\nabla \baru}_{\Lq{q}(J; \W{1}(\OM))}
	\leq C \norm{\nabla \theta}_\infty \norm{\baru}_{\bbeo_u(J)}.
\end{align*}
Then we arrive at
\begin{equation}
	\label{M_3-epsilon}
	\norm{M_3(\theta, \baru)}_{\bbfo_3(J)}
	\leq C \norm{\nabla \theta}_\infty \norm{\baru}_{\bbeo_u}.
\end{equation}
Consequently, collecting \eqref{M_1-epsilon} and $ \eqref{M_2-epsilon} $ together yields
\begin{equation*}
	\norm{M(\theta, \baru, \barpi)}_{\bbfo(J)}
	\leq C ( \norm{\nabla \theta}_\infty (1 + \norm{\nabla \theta}_\infty) + T^{\frac{1}{2q}} \norm{\nabla^2 \theta}_\infty ) \norm{(\baru, \barpi)}_{\bbeo(J)}.
\end{equation*}
Since $ \norm{\nabla \theta}_\infty \leq \eta $, $ \norm{\nabla^2 \theta}_\infty \leq M $, by choosing $ \eta > 0 $ and $ T > 0 $ sufficiently small, one obtains \eqref{M-epsilon}. A Neumann series argument in $ \bbeo(J) $ finally implies that there exists a unique solution $ (\baru, \barpi) \in \bbeo(J) $ of $ L(\baru, \barpi) = M(\theta, \baru, \barpi) + F $ or equivalently a solution $ (u, \pi) $ of \eqref{StokesEqs: Bent-quarter}, provided that the data satisfy all relevant compatibility conditions at the contact line $ \Bar{G} \cap \Bar{S} $.

Conversely, \eqref{StokesEqs: Bent-quarter} admits a solution operator $ S_{QS} : \bbf_{QS} \rightarrow \bbe_{QS} $, where $ \bbf_{QS} $ and $ \bbe_{QS} $ are the solution space and data space, respectively, for the bent quarter-space and the data in $ \bbf_{QS} $ satisfy the compatibility conditions at the contact line $ \{ (x_1, 0, \theta(x_1)): x_1 \in \bbr \} $.

\subsection{The two-phase Stokes equations in half-spaces} \label{Section: Twophase-halfspace}
Let $ \OM := \bbr \times \bbr_+ \times \bbr $, $ G := \bbr \times \{ 0 \} \times \bbr $, $ \Sigma := \bbr \times \bbr_+ \times \{ 0 \} $ and $ \partial \Sigma := \bbr \times \{ 0 \} \times \{ 0 \} $. 
Consider the problem 
\begin{equation}\label{TwoPhaseStokesEqs: half-space}
	\begin{alignedat}{3}
		\rho \pt u - \Div \Smu(u, \pi) & = f, && \tin \Omega \backslash \Sigma \times J, \\
		\Div u & = f_d, && \tin \Omega \backslash \Sigma \times J, \\
		\jump{u} & = g_1, && \ton \Sigma \times J, \\
		\jump{- \pi e_3 + \mu \left( \ptial{3} u + \nabla u_3 \right)} & = g_2, && \ton \Sigma \times J, \\
		\tran{(u_1, u_3)} & = 0, && \ton G \backslash \partial \Sigma \times J, \\
		- \pi + 2 \mu \ptial{2} u_2 & = g_4, && \ton G \backslash \partial \Sigma \times J, \\
		u(0) & = 0, && \tin \Omega \backslash \Sigma.
	\end{alignedat}
\end{equation}
Then we have the following well-posedness result.
\begin{theorem} \label{Theorem: MP-Twophase-halfspace}
	Let $ q > 5 $, $ T > 0 $, $ \rho_j, \mu_j > 0 $, $ j = 1,2 $, and $ J = (0,T) $. Assume that $ \Sigma, G \in C^3 $ and $ \partial \Sigma $ is of class $ C^4 $. Then \eqref{TwoPhaseStokesEqs: half-space} admits a unique solution $ (u, \pi) $ with regularity
	\begin{gather*}
		u \in \WO{1}(J; \Lq{q}(\OM)^3) \cap \Lq{q}(J; \W{2}(\OM \backslash \Sigma)^3), \\
		\pi \in \Lq{q}(J; \dW{1}(\OM \backslash \Sigma)), \quad 
		\jump{\pi} \in \W{\onehalf - \frac{1}{2q}}(J; \Lq{q}(\Sigma)) \cap \Lq{q}(J; \W{1 - \frac{1}{q}}(\Sigma)), \\
		\rv{\pi}_{x_2 = 0} \in \W{\onehalf - \frac{1}{2q}}(J; \Lq{q}(G)) \cap \Lq{q}(J; \W{1 - \frac{1}{q}}(G)),
	\end{gather*}
	if and only if the data satisfy the following regularity and compatibility conditions:
	\begin{enumerate}
		\item $ f \in \Lq{q}(J; \Lq{q}(\OM)^3) $,
		\item $ f_d \in \WO{1}(J; \dW{-1}(\OM)) \cap \Lq{q}(J; \W{1}(\OM \backslash \Sigma)) $,
		\item $ g_1 \in \WO{1 - \frac{1}{2q}}(J; \Lq{q}(\Sigma)^3) \cap \Lq{q}(J; \W{2 - \frac{1}{q}}(\Sigma)^3) $,
		\item $ \tran{((g_2)_1, (g_2)_2)} \in \WO{\onehalf - \frac{1}{2q}}(J; \Lq{q}(\Sigma)^2) \cap \Lq{q}(J; \W{1 - \frac{1}{q}}(\Sigma)^2) $,
		\item $ (g_2)_3 \in \W{\onehalf - \frac{1}{2q}}(J; \Lq{q}(\Sigma)) \cap \Lq{q}(J; \W{1 - \frac{1}{q}}(\Sigma)) $,
		\item $ g_4 \in \W{\onehalf - \frac{1}{2q}}(J; \Lq{q}(G)) \cap \Lq{q}(J; \W{1 - \frac{1}{q}}(G \backslash \partial \Sigma)) $,
		\item $ \tran{((g_1)_1, (g_1)_3)} = 0 $,
		$ (g_2)_1 = 0 $, $ (g_2)_3 = \jump{g_4} - \jump{2 \mu f_d} $ at $ \partial \Sigma $.
	\end{enumerate}
\end{theorem}
\begin{proof}
	The idea of the proof is based on the procedure in Wilke \cite{Wilke2020}, while a different type of boundary conditions (\textit{outflow} conditions) is considered in our case.
	
	\textbf{Reduction of $ f $ and $ f_d $.} Let us firstly reduce $ (f, f_d) $ to zero. To this end, we extend $ f \in \Lq{q}(J; \Lq{q}(\OM)^3) $ to some function $ \tf \in \Lq{q}(J; \Lq{q}(\bbr^3)^3) $ and
	\begin{equation*}
		f_d \in \Lq{q}(J; \W{1}(\OM \backslash \Sigma)) \cap \WO{1}(J; \dW{-1}(\OM))
	\end{equation*} 
	to some function 
	\begin{equation*}
		\tf_d \in \Lq{q}(J; \W{1}(\bbr^2 \times \dR)) \cap \WO{1}(J; \dW{-1}(\bbr^2 \times \dR))
	\end{equation*}
	by the extension \eqref{Extension:Sobolev} above with respect to $ - x_2 $ direction. Define 
	\begin{gather*}
		\tf^\pm := \rv{\tf}_{x_3 \gtrless 0} \in \Lq{q}(J; \Lq{q}(\bbr^2 \times \bbr_\pm)), \\
		\tf_d^\pm := \rv{\tf_d}_{x_3 \gtrless 0} \in \Lq{q}(J; \W{1}(\bbr^2 \times \bbr_\pm)) \cap \WO{1}(J; \dW{-1}(\bbr^2 \times \bbr_\pm)).
	\end{gather*}
	Now extend $ \tf^+ $ with respect to $ x_3 $ to some function $ \hatf^+ \in \Lq{q}(J; \Lq{q}(\bbr^3)^3) $ and $ \tf_d^+ $ to some function $ \hatf_d^+ \in \Lq{q}(J; \W{1}(\bbr^3)) \cap \W{1}(J; \dW{-1}(\bbr^3)) $. Let $ \mu^\pm := \rv{\mu}_{x_3 \gtrless 0} $, $ \rho^\pm := \rv{\rho}_{x_3 \gtrless 0} $ and extend $ \mu^+, \rho^+ \in \bbr \times \bbr_+^2 $ to $ \hat{\mu}^+ \equiv \mu^+, \hat{\rho}^+ \equiv \rho^+ \in \bbr^3 $ by constants. Then solving the full space Stokes equation
	\begin{equation} \label{StokesEqs: full-space-TwoPhaseStokes}
		\begin{alignedat}{3}
			\hat{\rho}^+ \pt \hatu^+ - \Div S_{\hat{\mu}^+} (\hatu^+, \hpi^+) & = \hatf^+, && \tin \bbr^3 \times J, \\
			\Div \hatu^+ & = \hatf_d^+, && \tin \bbr^3 \times J, \\
			\hatu^+(0) & = 0, && \tin \bbr^3,
		\end{alignedat}
	\end{equation}
	yields a unique solution $ (\hatu^+, \hpi^+) $ satisfying
	\begin{gather*}
		\hatu^+ \in \WO{1}(J; \Lq{q}(\bbr^3)^3) \cap \Lq{q}(J; \W{2}(\bbr^3)^3), \quad
		\hpi^+ \in \Lq{q}(J; \dW{1}(\bbr^3)),
	\end{gather*}
	by Pr\"uss--Simonett \cite[Theorem 7.1.1]{PS2016}. Analogously, we extend $ \tf^- $ and $ \tf_d^- $ to some function $ \hatf^- \in \Lq{q}(J; \Lq{q}(\bbr^3)^3) $ and $ \hatf_d^- \in \Lq{q}(J; \W{1}(\bbr^3)) \cap \W{1}(J; \dW{-1}(\bbr^3)) $  with respect to $ x_3 $ respectively. Let $ \hat{\mu}^- \equiv \mu^-, \hat{\rho}^- \equiv \rho^- $. Then we solve \eqref{StokesEqs: full-space-TwoPhaseStokes} with all superscripts $ ^+ $ replaced by $ ^- $ to obtain a unique solution
	\begin{gather*}
		\hatu^- \in \WO{1}(J; \Lq{q}(\bbr^3)^3) \cap \Lq{q}(J; \W{2}(\bbr^3)^3), \quad
		\hpi^- \in \Lq{q}(J; \dW{1}(\bbr^3)).
	\end{gather*}
	Based on these functions, we define
	\begin{equation*}
		(\hatu, \hpi) := \left\{
			\begin{aligned}
				& \rv{(\hatu^+, \hpi^+)}_{\OM}, && \text{ if } x_3 > 0, \\
				& \rv{(\hatu^-, \hpi^-)}_{\OM}, && \text{ if } x_3 < 0.
			\end{aligned}
		\right.
	\end{equation*}
	Then
	\begin{gather*}
		\hatu \in \WO{1}(J; \Lq{q}(\OM)^3) \cap \Lq{q}(J; \W{2}(\OM \backslash \Sigma)^3), \quad
		\hpi \in \Lq{q}(J; \dW{1}(\OM \backslash \Sigma)).
	\end{gather*}
	If $ (u, \pi) $ is a solution of \eqref{TwoPhaseStokesEqs: half-space}, then $ (u - \hatu, \pi - \hpi) $ solves \eqref{TwoPhaseStokesEqs: half-space} with $ (f, f_d) = 0 $ and some modified data $ g_j $, $ j \in \{1, 2, 3, 4\} $ (not to be relabeled), in the right regularity classes, having vanishing traces at $ t = 0 $ whenever it exists and satisfying the compatibility conditions at $ \partial \Sigma $ stated in Theorem \ref{Theorem: MP-Twophase-halfspace}.

	\textbf{Reduction of $ g_1 $ and $ g_4 $.}
	Now we extend
	\begin{equation*}
		g_4^+ := \rv{g_4}_{x_3 > 0} \in \W{\onehalf - \frac{1}{2q}}(J; \Lq{q}(\bbr \times \{0\} \times \bbr_+)) \cap \Lq{q}(J; \W{1 - \frac{1}{q}}(\bbr \times \{0\} \times \bbr_+))
	\end{equation*}
	by means of even reflection to functions
	\begin{equation*}
		\tg_4^+ \in \W{\onehalf - \frac{1}{2q}}(J; \Lq{q}(\bbr \times \{0\} \times \bbr)) \cap \Lq{q}(J; \W{1 - \frac{1}{q}}(\bbr \times \{0\} \times \bbr)).
	\end{equation*}
	Let $ \mu^+ := \rv{\mu}_{x_3 > 0} $, $ \rho^+ := \rv{\rho}_{x_3 > 0} $. Then the half-space Stokes problem with the outflow boundary condition
	\begin{equation}\label{StokesEqs: half-space_outflow}
		\begin{alignedat}{3}
			\rho^+ \pt \cu^+ - \mu^+ \Delta \cu^+ + \nabla \cpi^+ & = 0, && \tin \OM \times J, \\
			\Div \cu^+ & = 0, && \tin \OM \times J, \\
			\tran{(\cu_1^+, \cu_3^+)} & = 0, && \ton G \times J, \\
			- \cpi^+ + 2 \mu^+ \ptial{2} \cu_2^+ & = \tg_4^+, && \ton G \times J, \\
			\cu^+(0) & = 0, && \tin \OM,
		\end{alignedat}
	\end{equation}
	admits a unique solution $ (\cu^+, \cpi^+) $ with regularity
	\begin{gather*}
		\cu^+ \in \WO{1}(J; \Lq{q}(\bbr \times \bbr_+ \times \bbr)^3) \cap \Lq{q}(J; \W{2}(\bbr \times \bbr_+ \times \bbr)^3), \\
		\cpi^+ \in \Lq{q}(J; \dW{1}(\bbr \times \bbr_+ \times \bbr)), \quad
		\rv{\cpi^+}_{x_2 = 0} \in \W{\onehalf - \frac{1}{2q}}(J; \Lq{q}(G)) \cap \Lq{q}(J; \W{1 - \frac{1}{q}}(G)),
	\end{gather*}
	by Bothe--Pr\"uss \cite[Theorem 6.1]{BP2007} or Pr\"uss--Simonett \cite[Theorem 7.2.1]{PS2016}.
	Repeating the same procedure with all superscripts $ ^+ $ replaced by $ ^- $ in \eqref{StokesEqs: half-space_outflow}, where $ \tg_4^- $ is the extension of $ g_4^- := \rv{g_4}_{x_3 < 0} $ similar to $ \tg_4^+ $, we get the corresponding functions $ (\cu^-, \cpi^-) $ with regularity
	\begin{gather*}
		\cu^- \in \WO{1}(J; \Lq{q}(\bbr \times \bbr_+ \times \bbr)^3) \cap \Lq{q}(J; \W{2}(\bbr \times \bbr_+ \times \bbr)^3), \\
		\cpi^- \in \Lq{q}(J; \dW{1}(\bbr \times \bbr_+ \times \bbr)), \quad
		\rv{\cpi^-}_{x_2 = 0} \in \W{\onehalf - \frac{1}{2q}}(J; \Lq{q}(G)) \cap \Lq{q}(J; \W{1 - \frac{1}{q}}(G)).
	\end{gather*}
	
	Define
	\begin{equation*}
		(\cu, \cpi) := \left\{
			\begin{aligned}
				& (\cu^+, \cpi^+), && \text{ if } x_3 > 0, \\
				& (\cu^-, \cpi^-), && \text{ if } x_3 < 0.
			\end{aligned}
		\right.
	\end{equation*}
	Then $ (\baru, \barpi) := (u - \cu, \pi - \cpi) $ (restricted) satisfies that $ \rv{\baru}_{t = 0} = 0 $, $ \jump{\baru} = g_1 - \jump{\cu} =: k $ on $ \Sigma $ and
	\begin{equation*}
		\tran{(\baru_1, \baru_3)} = 0, \quad - \barpi + 2 \mu \ptial{2} \baru_2 = 0,
	\end{equation*}
	on $ G \backslash \partial \Sigma $, which also means $ \ptial{i} \bar{u}_j = 0 $ for $ i,j = 1, 3 $ on $ G \backslash \partial \Sigma $. From the compatibility conditions, we know $ k_1 = k_3 = 0 $ on $ \partial \Sigma $. Moreover, since $ \Div \baru = 0 $ up to the boundary, one obtains $ \ptial{2} k_2 = \jump{\ptial{2} \baru_2} = \Div \baru - \jump{\ptial{1} \baru_1 + \ptial{3} \baru_3} = 0 $ on $ \partial \Sigma $. Therefore, one may extend
	\begin{equation*}
		k \in \WO{1 - \frac{1}{2q}}(J; \Lq{q}(\Sigma)^3) \cap \Lq{q}(J; \W{2 - \frac{1}{q}}(\Sigma)^3)
	\end{equation*}
	to some function
	\begin{equation*}
		\tilde{k} \in \WO{1 - \frac{1}{2q}}(J; \Lq{q}(\bbr^2 \times \{0\})^3) \cap \Lq{q}(J; \W{2 - \frac{1}{q}}(\bbr^2 \times \{0\})^3)
	\end{equation*}
	by odd reflection of $ k_1, k_3 $ as before and even reflection of $ k_2 $. Then the half-space Dirichlet Stokes equation
	\begin{equation}\label{StokesEqs: half-space_Twophase-Dirichlet}
		\begin{alignedat}{3}
			\rho \pt w - \mu \Delta w + \nabla p & = 0, && \tin \bbr \times \bbr \times \bbr_+ \times J, \\
			\Div w & = 0, && \tin \bbr \times \bbr \times \bbr_+ \times J, \\
			w & = \tilde{k}, && \ton \bbr \times \bbr \times \{0\} \times J, \\
			w(0) & = 0, && \tin \bbr \times \bbr \times \bbr_+,
		\end{alignedat}
	\end{equation}
	admits a unique solution
	\begin{gather*}
		w \in \WO{1}(J; \Lq{q}(\bbr^2 \times \bbr_+)^3) \cap \Lq{q}(J; \W{2}(\bbr^2 \times \bbr_+)^3), \
		p \in \Lq{q}(J; \dW{1}(\bbr^2 \times \bbr_+)).
	\end{gather*}
	By symmetry, it is easy to verify that the function \begin{gather*}
		\bar{w}(t,x) = \tran{(- w_1, w_2, - w_3)}(t, x_1, - x_2, x_3), \
		\bar{p}(t,x) = - p(t, x_1, - x_2, x_3)
	\end{gather*}
	is also a solution of \eqref{StokesEqs: half-space_Twophase-Dirichlet}. The uniqueness of $ (w, p) $ implies that $ (w, p) = (\bar{w}, \bar{p}) $ and therefore $ w_1 = w_3 = 0 $ as well as $ \ptial{2} w_2 = 0 $, $ p = 0 $ on $ G \backslash \partial \Sigma $. Let $ (\baru_\pm, \barpi_\pm) := \rv{(\baru, \barpi)}_{x_3 \gtrless 0} $ and define
	\begin{equation*}
		(u^*, \pi^*) := \left\{
			\begin{aligned}
				& (\baru_+, \barpi_+) - (w, p), && \text{ if } x_3 > 0, \\
				& (\baru_-, \barpi_-), && \text{ if } x_3 < 0.
			\end{aligned}
		\right.
	\end{equation*}
	Then $ \jump{u^*} = 0 $ on $ \Sigma $ and
	\begin{equation*}
		\tran{(u^*_1, u^*_3)} = 0, \quad - \pi^* + 2 \mu \ptial{2} u^*_2 = 0,
	\end{equation*}
	on $ G \backslash \partial \Sigma $. Hence, $ (u^*, \pi^*) $ (restricted) solves \eqref{TwoPhaseStokesEqs: half-space} with $ (f, f_d, g_j) = 0 $, $ j \in \{1,4\} $ and the modified data $ g_2 $ (not to be relabeled) in proper regularity classes having vanishing trace at $ t = 0 $ whenever it exists and satisfying the compatibility conditions.
	
	\textbf{Proof of Theorem \ref{Theorem: MP-Twophase-halfspace}.} From the compatibility conditions and boundary conditions on $ G \backslash \partial \Sigma $, it is possible to extend $ ((g_2)_1, (g_2)_3) $ by odd reflection and $ (g_2)_2 $ by even reflection with respect to $ - x_2 $. In the following, we consider the reflected problem
	\begin{equation}\label{TwoPhaseStokesEqs: full-space}
		\begin{alignedat}{3}
			\rho \pt \tu - \mu \Delta \tu + \nabla \tpi & = 0, && \tin \bbr \times \bbr \times \dR \times J, \\
			\Div \tu & = 0, && \tin \bbr \times \bbr \times \dR \times J, \\
			\jump{u} & = 0, && \ton \bbr \times \bbr \times \{0\} \times J, \\
			\jump{- \tpi e_3 + \mu \left( \ptial{3} \tu + \nabla \tu_3 \right)} & = \tg_2, && \ton \bbr \times \bbr \times \{0\} \times J, \\
			\tu(0) & = 0, && \tin \bbr \times \bbr \times \dR.
		\end{alignedat}
	\end{equation}
	with given reflected data
	\begin{gather*}
		((\tg_2)_1, (\tg_2)_2) \in \WO{\onehalf - \frac{1}{2q}}(J; \Lq{q}(\bbr^2 \times \{0\})^2) \cap \Lq{q}(J; \W{1 - \frac{1}{q}}(\bbr^2 \times \{0\})^2), \\
		(\tg_2)_3 \in \W{\onehalf - \frac{1}{2q}}(J; \Lq{q}(\bbr^2 \times \{0\})) \cap \Lq{q}(J; \dW{1 - \frac{1}{q}}(\bbr^2 \times \{0\})).
	\end{gather*}
	By Pr\"uss--Simonett \cite[Theorem 3.1]{PS2010}, \eqref{TwoPhaseStokesEqs: full-space} admits a unique solution $ (\tu, \tpi) $ with
	\begin{gather*}
		\tu \in \W{1}(J; \Lq{q}(\bbr^3)^3) \cap \Lq{q}(J; \W{2}(\bbr^2 \times \dR)^3), \
		\tpi \in \Lq{q}(J; \dW{1}(\bbr^2 \times \dR)), \\
		\jump{\tpi} \in \W{\onehalf - \frac{1}{2q}}(J; \Lq{q}(\bbr^2 \times \{0\})) \cap \Lq{q}(J; \W{1 - \frac{1}{q}}(\bbr^2 \times \{0\})).
	\end{gather*}
	Define $ (\baru, \barpi) $ as
	\begin{gather*}
		\baru(t,x) := \tran{(- u_1, u_2, - u_3)}(t, x_1, - x_2, x_3), \
		\barpi(t,x) := - \pi(t, x_1, - x_2, x_3).
	\end{gather*}
	It is clear that $ (\baru, \barpi) $ is also a solution of \eqref{TwoPhaseStokesEqs: full-space}. The uniqueness implies that $ (\tu, \tpi) = (\hatu, \hpi) $ and therefore, 
	\begin{equation*}
		\tu_1 = \tu_3 = 0, \quad \ptial{2} \tu_2 = 0,
	\end{equation*}
	as well as $ \tpi = 0 $ on $ G \backslash \partial \Sigma $. Consequently, the restriction $ \rv{(\tu, \tpi)}_{\OM} $ is the strong solution of \eqref{TwoPhaseStokesEqs: half-space} with $ (f, f_d, g_j) = 0 $, $ j \in \{1,4\} $, while the uniqueness can be verified by a similar argument as in Section \ref{Section: Stokes-quarter}.
\end{proof}

\subsection{The two-phase Stokes equations in half-spaces with a bent interface} \label{Section: TwoPhase-Bent-halfspace}
In this section, we consider the case of a bent interface for the two-phase Stokes equations in half-spaces. So let the interface $ \Sigma $ be given as a graph of a function $ \theta : \bbr \times \bbr_+ \rightarrow \bbr $ of class $ C^3 $. Moreover, let $ x = (x', x_3) $, $ x' = (x_1, x_2) $ and $ \nabla_{x'} = \tran{(\ptial{1}, \ptial{2})} $, we assume that $ \norm{\nabla_{x'} \theta(x')}_\infty \leq \eta $ and $ \norm{\nabla_{x'}^i \theta(x')}_\infty \leq M $, $ i \in \{2,3\} $ for $ M > 0 $, where $ \eta > 0 $ will be chosen sufficiently small later. Then the interface is defined as
\begin{equation*}
	\Sigtheta := \left\{ (x', x_3) \in \bbr^3 : x_3 = \theta(x'), x' \in \bbr \times \bbr_+ \right\},
\end{equation*}
and the associated unit normal vector is given by
\begin{equation*}
	\nuSigtheta(x') := \beta(x') \tran{(- \nabla_{x'} \theta(x'), 1)}, \quad 
	\beta(x') := \frac{1}{\sqrt{1 + |\nabla_{x'} \theta(x')|^2}}.
\end{equation*}
Let $ \OM := \bbr \times \bbr_+ \times \bbr $, $ G := \bbr \times \{ 0 \} \times \bbr $ and 
\begin{equation*}
	\partial \Sigtheta := \left\{ (x', x_3): x_3 = \theta(x'), x' \in \bbr \times \{0\} \right\}.
\end{equation*}
We consider the problem
\begin{equation} \label{TwoPhaseStokesEqs: half-space_bent interface}
	\begin{alignedat}{3}
		\rho \pt u - \Div \Smu(u, \pi) & = f, && \tin \OM \backslash \Sigtheta \times J, \\
		\Div u & = f_d, && \tin \OM \backslash \Sigtheta \times J, \\
		\jump{u} & = g_1, && \ton \Sigtheta \times J, \\
		\jump{- \pi \bbi + \mu ( \nabla u + \tran{\nabla} u )} \nuSigtheta & = g_2, && \ton \Sigtheta \times J, \\
		\tran{(u_1, u_3)} & = g_3, && \ton G \backslash \partial \Sigtheta \times J, \\
		- \pi + 2 \mu \ptial{2} u_2 & = g_4, && \ton G \backslash \partial \Sigtheta \times J, \\
		u(0) & = u_0, && \tin \OM \backslash \Sigtheta,
	\end{alignedat}
\end{equation}
where $ \Smu(u, \pi) = - \pi \bbi + \mu ( \nabla u + \tran{\nabla} u ) $ and $ \rho, \mu > 0 $ are given constants. 
For given data $ (f, f_d, g_1, g_2, g_3, g_4, u_0) $, we give the following compatibility conditions at time $ t = 0 $
\begin{gather*}
	\Div u_0 = \rv{f_d}_{t = 0}, \quad
	\jump{u_0} = \rv{g_1}_{t = 0}, \quad \tran{((u_0)_1, (u_0)_3)}|_{G} = \rv{g_3}_{t = 0}, \\
	\PSigtheta \jump{\mu ( \nabla u_0) + \tran{\nabla} u_0 ) \nuSigtheta} = \PSigtheta g_2 |_{t = 0},
\end{gather*}
where $ \PSigtheta := I - \nuSigtheta \otimes \nuSigtheta $ denotes the tangential projection of $ \nuSigtheta $.
Moreover, at the contact line $ \partial \Sigtheta $, we have
\begin{equation}
	\label{Compatibility:two-phase-bent}
	\begin{gathered}
		\jump{g_3} = \tran{((g_1)_1, (g_1)_3)}, \\
		\begin{aligned}
			(g_2)_1 & = \jump{2 \mu \ptial{1} (g_3)_1 \nuSigtheta \cdot e_1 + \mu (\ptial{1} (g_3)_2 + \ptial{3} (g_3)_1) \nuSigtheta \cdot e_3}, \\
				& \quad + \jump{g_4 - 2 \mu f_d} \nuSigtheta \cdot e_1 + \jump{2 \mu (\ptial{1} (g_3)_1 + \ptial{3} (g_3)_2)} \nuSigtheta \cdot e_1,
		\end{aligned} \\
		\begin{aligned}
			(g_2)_3 & = \jump{2 \mu \ptial{3} (g_3)_2 \nuSigtheta \cdot e_3 + \mu (\ptial{1} (g_3)_2 + \ptial{3} (g_3)_1) \nuSigtheta \cdot e_1} \\
				& \quad + \jump{g_4 - 2 \mu f_d} \nuSigtheta \cdot e_3 + \jump{2 \mu (\ptial{1} (g_3)_1 + \ptial{3} (g_3)_2)} \nuSigtheta \cdot e_3.
		\end{aligned}
	\end{gathered}
\end{equation}

\subsubsection{Reduction}
As above, we want to reduce \eqref{TwoPhaseStokesEqs: half-space_bent interface} to the case $ (u_0, f, f_d, g_3, g_4) = 0 $ for convenience. In a same way as in Section \ref{subsubsection:reduction-bent-quarter}, one can reduce to the case $ (u_0, f, f_d) = 0 $, we do not give the details here. Then we have modified data $ g_j $, $ j \in \{1,2,3,4\} $, in the right regularity classes and having vanishing traces at $ t = 0 $ whenever it exists. 

Next, set $ g_j^\pm := \rv{g_j}_{x_3 \gtrless \theta(x_1, 0)} $, $ j \in \{3,4\} $. By the extension
\begin{equation} \label{Extension:x_3theta}
	\tg_3^+(x_1, 0, x_3) = 
	\left\{
		\begin{aligned}
			& g_3^+(x_1, x_3 - \theta(x_1, 0)), && \text{ if } x_1 \in \bbr, x_3 > \theta(x_1, 0), \\
			& - g_3^+(x_1, -2 (x_3 - \theta(x_1, 0))) \\
			& \qquad + 2 g_3^+(x_1, - (x_3 - \theta(x_1, 0)) / 2), && \text{ if } x_1 \in \bbr, x_3 < \theta(x_1, 0),
		\end{aligned}
	\right.
\end{equation}
we have
\begin{equation*}
	\tg_3^+ \in \WO{1 - \frac{1}{2q}}(J; \Lq{q}(\bbr \times \{0\} \times \bbr)^2) \cap \Lq{q}(J; \W{2 - \frac{1}{q}}(\bbr \times \{0\} \times \bbr)^2).
\end{equation*}
Employing an even reflection to $ g_4^+ $, one obtains the extended function
\begin{equation*}
	\tg_4^+ \in \W{\onehalf - \frac{1}{2q}}(J; \Lq{q}(\bbr \times \{0\} \times \bbr)^2) \cap \Lq{q}(J; \W{1 - \frac{1}{q}}(\bbr \times \{0\} \times \bbr)^2).
\end{equation*}
Let $ \mu^+ := \rv{\mu}_{x_3 > 0} $, $ \rho^+ := \rv{\rho}_{x_3 > 0} $. Now consider the half-space Stokes equation with the outflow boundary condition
\begin{equation}\label{StokesEqs: bent half-space_outflow}
	\begin{alignedat}{3}
		\rho^+ \pt \cu^+ - \mu^+ \Delta \cu^+ + \nabla \cpi^+ & = 0, && \tin \OM \times J, \\
		\Div \cu^+ & = 0, && \tin \OM \times J, \\
		\tran{(\cu_1^+, \cu_3^+)} & = \tg_3^+, && \ton G \times J, \\
		- \cpi^+ + 2 \mu^+ \ptial{2} \cu_2^+ & = \tg_4^+, && \ton G \times J, \\
		\cu^+(0) & = 0, && \tin \OM,
	\end{alignedat}
\end{equation}
which admits a unique solution $ (\cu^+, \cpi^+) $ with regularity
\begin{gather*}
	\cu^+ \in \WO{1}(J; \Lq{q}(\OM)^3) \cap \Lq{q}(J; \W{2}(\OM)^3), \\
	\cpi^+ \in \Lq{q}(J; \dW{1}(\OM)), \quad \rv{\cpi^+}_{x_2 = 0} \in \W{\onehalf - \frac{1}{2q}}(J; \Lq{q}(G)) \cap \Lq{q}(J; \W{1 - \frac{1}{q}}(G)),
\end{gather*}
by Bothe--Pr\"uss \cite[Theorem 6.1]{BP2007} or Pr\"uss--Simonett \cite[Theorem 7.2.1]{PS2016}. Repeating the same procedure with all superscripts $ ^+ $ replaced by $ ^- $ in \eqref{StokesEqs: bent half-space_outflow}, where $ \tg_j^- $ are the extensions of $ g_j^- := \rv{g_j}_{x_3 < \theta(x_1, 0)} $ similar to $ \tg_j^+ $, $ j = 3,4 $, we get the corresponding functions $ (\cu^-, \cpi^-) $ with regularity
\begin{gather*}
	\cu^- \in \WO{1}(J; \Lq{q}(\OM)^3) \cap \Lq{q}(J; \W{2}(\OM)^3), \\
	\cpi^- \in \Lq{q}(J; \dW{1}(\OM)), \quad
	\rv{\cpi^-}_{x_2 = 0} \in \W{\onehalf - \frac{1}{2q}}(J; \Lq{q}(G)) \cap \Lq{q}(J; \W{1 - \frac{1}{q}}(G)).
\end{gather*}

We define
\begin{equation*}
	(\cu, \cpi) := \left\{
		\begin{aligned}
			& (\cu^+, \cpi^+), && \text{ if } x_3 > \theta(x'), \\
			& (\cu^-, \cpi^-), && \text{ if } x_3 < \theta(x').
		\end{aligned}
	\right.
\end{equation*}
Then $ (\baru, \barpi) := (u - \cu, \pi - \cpi) $ (restricted) solves \eqref{TwoPhaseStokesEqs: half-space_bent interface} with $ (u_0, f, f_d, g_3, g_4) = 0 $ and modified data $ (g_1, g_2) $ (not to be relabeled) belonging to proper regularity classes with vanishing traces at $ t = 0 $ whenever they exist and satisfying
\begin{equation*}
	\tran{((g_1)_1, (g_1)_3)} = 0, \quad
	\tran{((g_2)_1, (g_2)_3)} = 0, \quad 
	\ton G.
\end{equation*}


\subsubsection{Transform to a half space}
By the reduction argument explained above we may assume
\begin{equation*}
	(u_0, f, f_d, g_3, g_4) = 0.
\end{equation*}
Our aim now is to transform $ \Sigtheta $ to $ \Sigma := \bbr \times \bbr_+ \times \{0\} $ and $ \partial \Sigtheta $ to $ \partial \Sigma := \bbr \times \{0\} \times \{0\} $. To this end, we introduce the new variables $ \barx' = x' $ and $ \barx_3 = x_3 - \theta(x') $, where $ \barx' := (\barx_1, \barx_2) $, for $ x \in \OM $, as in Section \ref{subsubsection:transform-quarter}. Then we define
\begin{equation*}
	\baru(t, \barx) := u(t, \barx', \barx_3 + \theta(\barx')), \quad
	\barpi(t, \barx) := \pi(t, \barx', \barx_3 + \theta(\barx')),
\end{equation*}
for $ t \in J $, $ \barx' \in \bbr \times \bbr_+ $. In the same way, we transform the data $ (g_1, g_2) $ to $ (\bar{g}_1, \bar{g}_2) $. In the following, we will write $ \ptial{\barx_3} $ as $ \ptial{3} $ for the sake of readability. Then with the help of \eqref{transform-relation-low} and \eqref{transform-relation-high} in Section \ref{subsubsection:transform-quarter}, one arrives at the transformed system
\begin{equation} \label{TwoPhaseStokesEqs: half-space_bent interface_transformed}
	\begin{alignedat}{3}
		\rho \pt \baru - \mu \Delta \baru + \nabla \barpi & = M_1(\theta, \baru, \barpi), && \tin \OM \backslash \Sigma \times J, \\
		\Div \baru & = M_2(\theta, \baru), && \tin \OM \backslash \Sigma \times J, \\
		\jump{\baru} & = \bar{g}_1, && \ton \Sigma \times J, \\
		\jump{- \barpi e_3 + \mu ( \ptial{3} \baru + \nabla \baru_3 )} & = M_4(\theta, \baru) + (\bar{g}_2 + \nabla \theta (\bar{g}_2)_3)/\beta, && \ton \Sigma \times J, \\
		\tran{(\baru_1, \baru_3)} & = 0, && \ton G \backslash \partial \Sigma \times J, \\
		- \barpi + 2 \mu \ptial{2} \baru_2 & = M_5(\theta, \baru), && \ton G \backslash \partial \Sigma \times J, \\
		\baru(0) & = 0, && \tin \OM \backslash \Sigma,
	\end{alignedat}
\end{equation}
where $ M_j $, $ j = \{1, 2, 4, 5\} $, are given by
\begin{align*}
	M_1(\theta, \baru, \barpi) & := \mu \big( - 2 \ptial{3} \nabla \baru \nabla \theta
		+ \nabla^2 \theta \ptial{3} \baru
		+ (\ptial{3}^2 \baru \otimes \nabla \theta) \nabla \theta \big)
		+ \nabla \theta(\barx') \ptial{3} \barpi, \\
	M_2(\theta, \baru) & := \nabla \theta \cdot \ptial{3} \baru, \\
	M_4(\theta, \baru) & := \jump{\mu (\nabla \baru + \tran{\nabla} \baru)} \nabla \theta
	- \jump{\mu (\nabla \theta \otimes \ptial{3} \baru + \ptial{3} \baru \otimes \nabla \theta)} \nabla \theta \\
	 & \qquad \qquad \qquad \qquad + (\jump{\mu \ptial{3} \baru_3}/ \beta^2 - \jump{\mu (\ptial{3} \baru + \nabla \baru_3)} \cdot \nabla \theta) \nabla \theta, \\
	M_5(\theta, \baru) & := 2 \mu \ptial{2} \theta(x') \ptial{3} \baru_2.
\end{align*}

\subsubsection{Existence and uniqueness} \label{Bent TwoPhaseStokes: Existence and uniqueness}
As in Section \ref{Bent Stokes: Existence and uniqueness} and \cite[Section 1.3.4]{Wilke2020}, if the data are in the right regularity classes and $ \norm{\nabla \theta}_\infty \leq \eta $, $ \eta > 0 $ is sufficiently small, one may consider \eqref{TwoPhaseStokesEqs: half-space_bent interface_transformed} as a perturbation of \eqref{TwoPhaseStokesEqs: half-space_bent interface} and apply the Neumann series argument to \eqref{TwoPhaseStokesEqs: half-space_bent interface_transformed} via Theorem \ref{Theorem: MP-Twophase-halfspace}. To this end, we define the solution spaces as
\begin{equation*}
	\bbeo(J) := \bbeo_u(J) \times \bbe_\pi(J),
\end{equation*}
where
\begin{align*}
	& \bbeo_u(J) := \left\{ u \in \WO{1}(J; \Lq{q}(\OM)^3) \cap \Lq{q}(J; \W{2}(\OM \backslash \Sigma)^3): \rv{\tran{(u_1, u_3)}}_G = 0 \right\}, \\
	& \bbe_\pi(J) := \left\{ 
		\begin{aligned}
			& \pi \in \Lq{q}(J; \dW{1}(\OM)): 
			\jump{\pi} \in \W{\onehalf - \frac{1}{2q}}(J; \Lq{q}(\Sigma)) \cap \Lq{q}(J; \W{1 - \frac{1}{q}}(\Sigma)), \\
			& \qquad \qquad \qquad \quad \rv{\pi}_G \in \W{\onehalf - \frac{1}{2q}}(J; \Lq{q}(G)) \cap \Lq{q}(J; \W{1 - \frac{1}{q}}(G \backslash \partial \Sigma))
		\end{aligned}
	\right\}, 
\end{align*}
as well as the data function spaces
\begin{equation*}
	\widetilde{\bbf}(J) := \bbf_1(J) \times \bbfo_2(J) \times \bbfo_3(J) \times \bbfo_4(J) \times \bbfo_5(J),
\end{equation*}
where
\begin{align*}
	\bbf_1(J) & := \Lq{q}(J; \Lq{q}(\OM)^3), \\
	\bbfo_2(J) & := \WO{1}(J; \dW{-1}(\OM \backslash \Sigma)) \cap \Lq{q}(J; \W{1}(\OM \backslash \Sigma)), \\
	\bbfo_3(J) & := \WO{1 - \frac{1}{2q}}(J; \Lq{q}(\Sigma)^3) \cap \Lq{q}(J; \W{2 - \frac{1}{q}}(\Sigma)^3), \\
	\bbfo_4(J) & := \WO{\onehalf - \frac{1}{2q}}(J; \Lq{q}(\Sigma)^3) \cap \Lq{q}(J; \W{1 - \frac{1}{q}}(\Sigma)^3), \\
	\bbfo_5(J) & := \WO{\onehalf - \frac{1}{2q}}(J; \Lq{q}(G)) \cap \Lq{q}(J; \W{1 - \frac{1}{q}}(G \backslash \partial \Sigma)).
\end{align*}
In addition, let
\begin{align*}
	\bbfo(J) := \left\{ 
		(f_1, f_2, g_1, g_2, g_3) \in \widetilde{\bbf}(J): 
		\begin{aligned}
			\tran{((g_1)_1, (g_1)_3)} = 0, \  (g_2)_1 = 0, \text{ at } \partial \Sigma \\
			(g_2)_3 = \jump{g_3 - 2 \mu f_2}, \text{ at } \partial \Sigma
		\end{aligned}
	\right\}.
\end{align*}
Now we define an operator $ L : \bbeo(J) \rightarrow \bbfo(J) $ by
\begin{equation*}
	L(\baru, \barpi) := \left[ 
	\begin{gathered}
		\rho \pt \baru - \mu \Delta \baru + \nabla \barpi \\
		\Div \baru \\
		\jump{\baru} \\
		\jump{- \barpi e_3 + \mu ( \ptial{3} \baru + \nabla \baru_3 )} \\
		\rv{(- \barpi + 2 \mu \ptial{2} \baru_2)}_G
	\end{gathered}
	\right].
\end{equation*}
Then $ L $ is an isomorphism by Theorem \ref{Theorem: MP-Twophase-halfspace}. For the right-hand side of \eqref{StokesEqs: Bent-quarter_transformed}, we set
\begin{align*}
	M(\theta, \baru, \barpi)
	:= \tran{(M_1(\theta, \baru, \barpi), M_2(\theta, \baru), 0, M_4(\theta, \baru), M_5(\theta, \baru))}
\end{align*}
and 
\begin{align*}
	F := \tran{(0, 0, f_3, f_4, 0)}, \quad f_3 := \bar{g}_1, \quad f_4 := (\bar{g}_2 + \nabla \theta (\bar{g}_2)_3)/\beta.
\end{align*}
It is clear that $ F \in \widetilde{\bbf}(J) $, since $ \theta \in C^3(\bbr \times \bbr_+) $. We note that $ \tran{((\bar{g}_1)_1, (\bar{g}_1)_3)} = 0 $, $ ((\bar{g}_2 + \nabla \theta (\bar{g}_2)_3)/\beta)_1 = 0 $ as well as $ (\bar{g}_2)_3 = 0 $ at the contact line $ \partial \Sigma $, one obtains $ F \in \bbfo(J) $.

To apply the Neumann series argument, we investigate the perturbation term $ M(\theta, \baru, \barpi) $. Note that from the compatibility conditions, $ (M_4(\theta, \baru))_1 = 0 $ as well as $ (M_4(\theta, \baru))_3 = 0 $ must be true at the contact line $ \partial \Sigma $, since $ 2 \mu M_2(\theta, \baru) = M_5(\theta, \baru) $ at the contact line $ \partial \Sigma $. However, these do not hold in general, namely,
\begin{align*}
	(M_4(\theta, \baru))_1 & = \ptial{2} \theta \mu ( \ptial{1} \baru_2 + \ptial{2} \baru_1 + (1 + \ptial{1} \theta)\ptial{3} \baru_2 + \ptial{2} \baru_3 ), \text{ at } \partial \Sigma, \\
	(M_4(\theta, \baru))_3 & = \ptial{2} \theta \mu ( \ptial{2} \baru_3 + \ptial{3} \baru_2 ), \text{ at } \partial \Sigma.
\end{align*}
To overcome this trouble, we follow the argument in \cite{Wilke2020} by introducing a modified $ \Tilde{M}_4(\theta, \baru) $ as
\begin{equation*}
	\Tilde{M}_4(\theta, \baru) = (M_4(\theta, \baru))_1 - \ptial{2} \theta \tran{( K_1, 0, K_2 )},
\end{equation*}
where
\begin{align*}
	K_1(\theta, \baru) & = \ext_\Sigma \left( \jump{\rv{\mu ( \ptial{1} \baru_2 + \ptial{2} \baru_1 + (1 + \ptial{1} \theta)\ptial{3} \baru_2 + \ptial{2} \baru_3 )}_{G \backslash \partial \Sigma}} \right), \\
	K_2(\theta, \baru) & = \ext_\Sigma \left( \jump{\rv{\mu ( \ptial{2} \baru_3 + \ptial{3} \baru_2 )}_{G \backslash \partial \Sigma}} \right).
\end{align*}
Here, by Wilke \cite[Proposition 5.1]{Wilke2020}, there exists a linear and bounded extension operator $ \ext_\Sigma $ from
\begin{equation*}
	\WO{\onehalf - \frac{1}{q}}(J; \Lq{q}(\partial \Sigma)) \cap \Lq{q}(J; \W{1 - \frac{2}{q}}(\partial \Sigma))
\end{equation*}
to
\begin{equation*}
	\WO{\onehalf - \frac{1}{2q}}(J; \Lq{q}(\Sigma)) \cap \Lq{q}(J; \W{1 - \frac{1}{q}}(\Sigma)),
\end{equation*}
such that $ \rv{(\ext_\Sigma z)}_{\partial \Sigma} = z $ for all $ z \in \WO{1/2 - 1/q}(J; \Lq{q}(\partial \Sigma)) \cap \Lq{q}(J; \W{1 - 2/q}(\partial \Sigma)) $.
Define
\begin{align*}
	\Tilde{M}(\theta, \baru, \barpi)
	:= \tran{(M_1(\theta, \baru, \barpi), M_2(\theta, \baru), 0, \Tilde{M}_4(\theta, \baru), M_5(\theta, \baru))}.
\end{align*}
Then it can be verified easily that the modified perturbation $ \Tilde{M}(\theta, \baru, \barpi) \in \bbfo(J) $ for each $ (\baru, \barpi) \in \bbeo(J) $, combining the smoothness of $ \theta $. The only point that needs to be taken care of is $ M_2(\theta, \baru) \in \WO{1}(J; \dW{-1}(\OM \backslash \Sigma)) $. By integration by parts with respect to the $ \barx_3 $, we have
\begin{equation} \label{M2-W^-1_TwoPhase}
	\int_\OM M_2(\theta, \baru) \phi \d \barx
	= \int_\OM \nabla \theta \cdot \ptial{3} \baru \phi \d \barx 
	= - \int_\OM \nabla \theta \cdot  \baru \ptial{3} \phi \d \barx, 
\end{equation}
for all $ \phi \in W_{q',0}^1(\OM \backslash \Sigma) $, which implies the claim with $ \baru \in \bbeo_u(J) $.

Then we rewrite \eqref{TwoPhaseStokesEqs: half-space_bent interface_transformed}, with $ M_4 $ replaced by $ \Tilde{M}_4 $, as
\begin{align*}
	(\baru, \barpi) = L^{-1} \Tilde{M}(\theta, \baru, \barpi) + L^{-1} F.
\end{align*}
Our goal in the following is to show that for each $ \varepsilon > 0 $, there exist $ T_0 > 0 $ and $ \eta_0 > 0 $ such that
\begin{equation} \label{M-epsilon-TwoPhase}
	\norm{\Tilde{M}(\theta, \baru, \barpi)}_{\bbfo(J)} \leq \varepsilon \norm{(\baru, \barpi)}_{\bbeo(J)}
\end{equation}
holds for $ 0 < T < T_0 $ and $ 0 < \eta < \eta_0 $.

We estimate $ \Tilde{M}(\theta, \baru, \barpi) $ term by term. Mimicking the estimates in Section \ref{Bent Stokes: Existence and uniqueness}, one obtains
\begin{align}
	\label{M_1-epsilon-TwoPhase}
	\norm{M_1(\theta, \baru, \barpi)}_{\bbf_1(J)} 
	& \leq C ( \norm{\nabla \theta}_\infty + \norm{\nabla \theta}_\infty^2 + T^{\frac{1}{2q}} \norm{\nabla^2 \theta}_\infty ) \norm{(\baru, \barpi)}_{\bbeo(J)}, \\
	\label{M_2-epsilon_TwoPhase}
	\norm{M_2(\theta, \baru)}_{\bbfo_2(J)}
	& \leq C ( \norm{\nabla \theta}_\infty + T^{\frac{1}{2q}} \norm{\nabla^2 \theta}_\infty ) \norm{\baru}_{\bbeo_u(J)}.
\end{align}
where $ C > 0 $ does not depend on $ T > 0 $. For $ \Tilde{M}_4(\theta, \baru) \in \bbfo_4(J) $, we see that
\begin{align*}
	\norm{\Tilde{M}_4(\theta, \baru)}_{\bbfo_4(J)}
	= \norm{\Tilde{M}_4(\theta, \baru)}_{\WO{\onehalf - \frac{1}{2q}}(J; \Lq{q}(\Sigma)^3)} 
	+ \norm{\Tilde{M}_4(\theta, \baru)}_{\Lq{q}(J; \W{1 - \frac{1}{q}}(\Sigma)^3)}.
\end{align*}
The first term in the right-hand side above is estimated by
\begin{align*}
	\norm{\Tilde{M}_4(\theta, \baru)}_{\WO{\onehalf - \frac{1}{2q}}(J; \Lq{q}(\Sigma)^3)} 
	& \leq C (\norm{\nabla \theta}_\infty + \norm{\nabla \theta}_\infty^2) \norm{\nabla \baru}_{\WO{\onehalf - \frac{1}{2q}}(J; \Lq{q}(\Sigma)^{3 \times 3})} \\
	& \leq C (\norm{\nabla \theta}_\infty + \norm{\nabla \theta}_\infty^2) \norm{\baru}_{\bbeo_u},
\end{align*}
while the second term satisfies
\begin{align*}
	& \norm{\Tilde{M}_4(\theta, \baru)}_{\Lq{q}(J; \W{1 - \frac{1}{q}}(\Sigma)^3)}
	\leq C (\norm{\nabla \theta}_\infty + \norm{\nabla \theta}_\infty^2) \norm{\nabla \baru}_{\Lq{q}(J; \W{1 - \frac{1}{q}}(\Sigma)^{3 \times 3})} \\
	& \qquad \qquad \qquad + C (1 + \norm{\nabla \theta}_\infty) \norm{\nabla^2 \theta}_\infty \norm{\nabla \baru}_{\Lq{q}(J; \Lq{q}(\OM \backslash \Sigma)^{3 \times 3})} \\
	& \qquad \leq C (\norm{\nabla \theta}_\infty + \norm{\nabla \theta}_\infty^2) \norm{\baru}_{\bbeo_u} 
	+ C T^{\frac{1}{2q}} (1 + \norm{\nabla \theta}_\infty) \norm{\nabla^2 \theta}_\infty \norm{\baru}_{\bbeo_u},
\end{align*}
by the trace theory of Sobolev spaces. Hence,
\begin{align}
	\label{M_4-epsilon-TwoPhase}
	& \norm{\Tilde{M}_4(\theta, \baru)}_{\bbfo_4(J)} \\
	& \qquad \leq C (\norm{\nabla \theta}_\infty + \norm{\nabla \theta}_\infty^2) \norm{\baru}_{\bbeo_u} 
	+ C T^{\frac{1}{2q}} (1 + \norm{\nabla \theta}_\infty) \norm{\nabla^2 \theta}_\infty \norm{\baru}_{\bbeo_u}, \nonumber
\end{align}
The same estimate as in \eqref{M_3-epsilon} yields that
\begin{equation}
	\label{M_5-epsilon-TwoPhase}
	\norm{M_5(\theta, \baru)}_{\bbfo_5(J)} 
	\leq C \norm{\nabla \theta}_\infty \norm{\baru}_{\bbeo_u},
\end{equation}
Collecting \eqref{M_1-epsilon-TwoPhase}--\eqref{M_5-epsilon-TwoPhase}, we have
\begin{align*}
	& \norm{\Tilde{M}(\theta, \baru, \barpi)}_{\bbfo(J)} \\
	& \qquad
	\leq C \left( \norm{\nabla \theta}_\infty + \norm{\nabla \theta}_\infty^2
	+ C T^{\frac{1}{2q}} (1 + \norm{\nabla \theta}_\infty) \norm{\nabla^2 \theta}_\infty \right) \norm{(\baru, \barpi)}_{\bbeo(J)}.
\end{align*}
Since $ \norm{\nabla_{x'} \theta(x')}_\infty \leq \eta $ and $ \norm{\nabla_{x'}^2 \theta(x')}_\infty \leq M $, by choosing $ \eta > 0 $ and $ T > 0 $ sufficiently small, one obtains \eqref{M-epsilon-TwoPhase}. A Neumann series argument in $ \bbeo(J) $ finally implies that there exists a unique solution $ (\baru, \barpi) \in \bbeo(J) $ of $ L(\baru, \barpi) = \Tilde{M}(\theta, \baru, \barpi) + F $ or equivalently a solution $ (u, \pi) $ of \eqref{TwoPhaseStokesEqs: half-space_bent interface}.

Consequently, \eqref{TwoPhaseStokesEqs: half-space_bent interface} admits a solution operator $ S_{HS} : \bbf_{HS} \rightarrow \bbe_{HS} $, where $ \bbf_{HS} $ and $ \bbe_{HS} $ are the solution space and data space respectively, for a half-space with a bent interface and the data in $ \bbf_{HS} $ satisfy the compatibility conditions \eqref{Compatibility:two-phase-bent}.

%

\subsection{The heat equations in quarter-spaces} \label{Section: Heat-quarter}
As in Section \ref{Section: Stokes-quarter}, define $ \OM := \bbr \times \bbr_+ \times \bbr_+ $, $ G := \bbr \times \{ 0 \} \times \bbr_+ $ and $ S := \bbr \times \bbr_+ \times \{ 0 \} $. Consider the problem
\begin{equation}\label{HeatEqs: quarter}
	\begin{alignedat}{3}
		\pt u - D \Delta u & = f, && \tin \OM \times J, \\
		D \ptial{2} u & = g_1, && \ton G \times J, \\
		D \ptial{3} u & = g_2, && \ton S \times J, \\
		u(0) & = u_0, && \tin \OM,
	\end{alignedat}
\end{equation}
where $ u : \bbr^3 \times \bbr_+ \rightarrow \bbr $ is the quantity of the system, e.g., the temperature, the concentration, etc. $ D > 0 $ denotes the constant diffusivity. 

\begin{theorem} \label{Theorem: MP-Heat-quarterspace}
	Let $ q > 5 $, $ T > 0 $, $ D > 0 $ and $ J = (0,T) $. Assume that $ G, S \in C^3 $ and $ \partial G $ is of class $ C^4 $. Then there exists a unique solution 
	\begin{gather*}
		u \in \W{1}(J; \Lq{q}(\OM)) \cap \Lq{q}(J; \W{2}(\OM)), 
	\end{gather*}
	of \eqref{HeatEqs: quarter} if and only if the data satisfy the following regularity and compatibility conditions:
	\begin{enumerate}
		\item $ f \in \Lq{q}(J; \Lq{q}(\OM)) $,
		\item $ g_1 \in \W{\onehalf - \frac{1}{2q}}(J; \Lq{q}(G)) \cap \Lq{q}(J; \W{1 - \frac{1}{q}}(G)) $,
		\item $ g_2 \in \W{\onehalf - \frac{1}{2q}}(J; \Lq{q}(S)) \cap \Lq{q}(J; \W{1 - \frac{1}{q}}(S)) $,
		\item $ u_0 \in \W{2 - \frac{2}{q}}(\OM) $,
		\item $ \rv{D \ptial{2} u_0}_G = \rv{g_1}_{t = 0} $, $ \rv{D \ptial{3} u_0}_S = \rv{g_2}_{t = 0} $,
		\item $ \rv{\ptial{3} g_1}_{\Bar{G} \cap \Bar{S}} = \rv{\ptial{2} g_2}_{\Bar{G} \cap \Bar{S}} $.
	\end{enumerate}
\end{theorem}
\begin{proof}
	To prove Theorem \ref{Theorem: MP-Heat-quarterspace}, we extend the data suitably to the half-space and solve the heat equations in half space. To this end, one extends $ (f, g_1, u_0) $ with respect to $ x_3 $ by general extension to some functions 
	\begin{gather*}
		\tf \in \Lq{q}(J; \Lq{q}(\bbr \times \bbr_+ \times \bbr)), \quad
		\tu_0 \in \W{2 - \frac{2}{q}}(\bbr \times \bbr_+ \times \bbr), \\
		\tg_1 \in \W{\onehalf - \frac{1}{2q}}(J; \Lq{q}((\bbr \times \{0\} \times \bbr))) \cap \Lq{q}(J; \W{1 - \frac{1}{q}}((\bbr \times \{0\} \times \bbr))).
	\end{gather*}
	Solving
	\begin{equation*}
		\begin{alignedat}{3}
			\pt u - D \Delta u & = \tf, && \tin \bbr \times \bbr_+ \times \bbr \times J, \\
			D \ptial{2} u & = \tg_1, && \tin \bbr \times \{0\} \times \bbr \times J, \\
			u(0) & = \tu_0, && \tin \bbr \times \bbr_+ \times \bbr,
		\end{alignedat}
	\end{equation*}
	yields a unique solution
	\begin{equation*}
		\tu \in \W{1}(J; \Lq{q}(\bbr \times \bbr_+ \times \bbr)) \cap \Lq{q}(J; \W{2}(\bbr \times \bbr_+ \times \bbr)),
	\end{equation*}
	by \cite[Theorem 6.2.5]{PS2016}. Then $ (u - \tu) $ solves \eqref{HeatEqs: quarter} with $ (f, g_1, u_0) = 0 $ and modified $ g_2 $ (not to be relabeled) satisfying $ \rv{\ptial{2} g_2}_S = 0 $ with vanishing time trace at $ t = 0 $. Now extend $ g_2 $ by even reflection with respect to $ x_2 $ to functions
	\begin{gather*}
		\tg_2 \in \W{\onehalf - \frac{1}{2q}}(J; \Lq{q}(\bbr^2 \times \{0\})) \cap \Lq{q}(J; \W{1 - \frac{1}{q}}(\bbr^2 \times \{0\})).
	\end{gather*}
	We solve
	\begin{equation*}
		\begin{alignedat}{3}
			\pt u - D \Delta u & = 0, && \tin \bbr \times \bbr \times \bbr_+ \times J, \\
			D \ptial{3} u & = \tg_2, && \tin \bbr \times \bbr \times \{0\} \times J, \\
			u(0) & = 0, && \tin \bbr \times \bbr \times \bbr_+,
		\end{alignedat}
	\end{equation*}
	again by \cite[Theorem 6.2.5]{PS2016} to obtain a unique solution
	\begin{equation*}
		\hatu \in \W{1}(J; \Lq{q}(\bbr^2 \times \bbr_+)) \cap \Lq{q}(J; \W{2}(\bbr^2 \times \bbr_+)).
	\end{equation*}
	Note that $ \hatu(t, x_1, - x_2, x_3) $ is also a solution of above system. Then uniqueness of solutions implies that
	\begin{equation*}
		\hatu(t, x_1, x_2, x_3) = \hatu(t, x_1, - x_2, x_3) 
	\end{equation*}
	and hence $ \ptial{2} \hatu(t, x_1, 0, x_3) = 0 $.
	Therefore, the restricted $ u := \tu + \hatu $ is the desired solution to $ \eqref{HeatEqs: quarter} $. The uniqueness can be derived by the argument in Section \ref{Section: Stokes-quarter}, which completes the proof.
\end{proof}

\subsection{The heat equations in bent quarter-spaces} \label{Section: Heat-bent-quarter}
By the same argument in Section \ref{Section: Stokes-Bent-Quarter}, we can find a unique solution to \eqref{HeatEqs: quarter} with a perturbation of $ x_3 = 0 $. Hence in turn, one obtains a solution operator $ S_{HQS} : \bbf_{HQS} \rightarrow \bbe_{HQS} $, where $ \bbf_{HQS} $ and $ \bbe_{HQS} $ are the solution space and data space, respectively, for the bent quarter-space and the data in $ \bbf_{QS} $ satisfy the compatibility conditions at the contact line $ \{ (x_1, 0, \theta(x_1)): x_1 \in \bbr \} $.

\section{General bounded cylindrical domains}
\label{Section: cylindrical domain}
In this section, we will show that \eqref{Linearized TwoPhase}, \eqref{Linearized Heat_f} and \eqref{Linearized Heat_s} respectively admits a unique solution. For system \eqref{Linearized Heat_f} and \eqref{Linearized Heat_s}, a localization procedure do work since they are both standard parabolic equations, while some modifications are needed for \eqref{Linearized TwoPhase}. The point is that for \eqref{Linearized TwoPhase}, the presence of pressure and the divergence equation will bring additional difficulties, for which we proceed as in \cite{PS2016} and \cite{Wilke2020} to obtain extra regularity for the pressure under suitable conditions and reduce the system to a simpler form.

\subsection{Regularity of the pressure}
\label{Subsection: Regularity pressure}
In general, the pressure $ \pi $ has no more regularity than that stated in Theorem \ref{Theorem: Cylinder-TwoPhase}.
However, in a special situation, one may obtain extra time-regularity of $ \pi $.
\begin{proposition}\label{Proposition: pressure regularity}
	Let $ \OM $ be the bounded domain defined in Theorem \ref{Theorem: Cylinder-TwoPhase} and $ (u, \pi) \in \bbe(J) $ be a solution of the two-phase problem \eqref{Linearized TwoPhase} with 
	\begin{equation*}
		(u_0, f_d, g_1 \cdot \nuSigma, g_5 \cdot \nu_S) = 0
	\end{equation*}
	and $ f_u \in \WO{\alpha}(J; \Lq{q}(\OM)^3) $, for some $ \alpha \in (0, 1/2 - 1/2q) $.
	Then $ \pi \in \WO{\alpha}(J; \Lq{q}(\OM)) $ and there is a constant $ C > 0 $ independent of $ J $, $ u $, $ \pi $, $ f_u $, such that
	\begin{equation*}
		\norm{\pi}_{\W{\alpha}(J; \Lq{q}(\OM))}
		\leq C \left( \norm{u}_{\bbe_u}
		+ \norm{\pi}_{\bbe_\pi}
		+ \norm{f_u}_{\W{\alpha}(J; \Lq{q}(\OM)^3)}
		\right).
	\end{equation*} 
\end{proposition}
\begin{proof}
	Fix any $ \psi \in \Lq{q'}(\OM) $. By Theorem \ref{AppendixTheorem: Laplace}, the transmission problem
	\begin{equation} \label{EllipticEqs: transmission-Dirichlet}
		\begin{alignedat}{3}
			\Delta \phi & = \psi, && \tin \OM \backslash \Sigma, \\
			\jump{\rho \phi} & = 0, && \ton \Sigma, \\
			\jump{\ptial{\nuSigma} \phi} & = 0, && \ton \Sigma, \\
			\rho \phi & = 0, && \ton G \backslash \partial \Sigma, \\
			\ptial{\nu_S}\phi & = 0, && \ton S
		\end{alignedat}
	\end{equation}
	possesses a unique solution $ \phi \in W_{q'}^1 (\OM) \cap W_{q'}^2 (\OM \backslash \Sigma) $, which satisfies the estimate
	\begin{equation*}
		\norm{\phi}_{W_{q'}^1 (\OM) \cap W_{q'}^2 (\OM \backslash \Sigma)}
		\leq C \norm{\psi}_{\Lqa{q'}(\OM)}.
	\end{equation*}
	By integration by parts, one obtains
	\begin{align*}
		\int_\OM \pi \psi \d x
		& = \int_\OM \pi \Delta \phi \d x
		= - \int_\Sigma \jump{\pi} \ptial{\nuSigma} \phi \d \sigma
		+ \int_G \pi \ptial{\nu_G} \phi \d \sigma 
		- \int_\OM \nabla \pi \cdot \nabla \phi \d x \\
		& 
		= - \int_\Sigma \jump{\pi} \ptial{\nuSigma} \phi \d \Sigma	
		+ \int_G \pi \ptial{\nu_G} \phi \d \sigma 
		- \int_\OM (f_u - \rho \pt u + \nu \Delta u) \cdot \nabla \phi \d x \\
		& 
		= - \int_\Sigma \left( \jump{\pi} \ptial{\nuSigma} \phi + \jump{\mu \ptial{\nuSigma} u \cdot \nabla \phi} \right) \d \sigma	
		+ \int_G \pi \ptial{\nu_G} \phi \d \sigma 
		- \int_\OM f_u \cdot \nabla \phi \d x \\
		& \quad  
		- \int_S \mu \ptial{\nu_S} u \cdot \nabla \phi \d \sigma 
		- \int_G \mu \ptial{\nu_G} u \cdot \nabla \phi \d \sigma 
		+ \int_\OM \mu \nabla u : \nabla^2 \phi \d x.
	\end{align*}
	Taking the supremum of the left-hand side over all functions $ \psi \in \Lq{q'}(\OM) $, we have
	\begin{align*}
		\norm{\pi(t)}_{\Lq{q}(\OM)}
		& \leq C \left(
		\norm{\nabla u(t)}_{\Lq{q}(\OM)}
		+ \norm{\ptial{\nu_S} u(t)}_{\Lq{q}(S)}
		+ \norm{\ptial{\nu_G} u(t)}_{\Lq{q}(G)} \right. \\
		& \left. \qquad 
		+ \norm{\ptial{\nuSigma} u(t)}_{\Lq{q}(\Sigma)} 
		+ \norm{\jump{\pi(t)}}_{\Lq{q}(\Sigma)} 
		+ \norm{\pi(t)}_{\Lq{q}(G)}
		+ \norm{f_u(t)}_{\Lq{q}(\OM)}
		\right),
	\end{align*}
	for almost all $ t \in J $. Furthermore,
	\begin{align*}
		& \norm{\pi(t) - \pi(s)}_{\Lq{q}(\OM)} \\
		& \leq C \left(
		\norm{\nabla (u(t) - u(s))}_{\Lq{q}(\OM)}
		+ \norm{\ptial{\nu_S} (u(t) - u(s))}_{\Lq{q}(S)}
		+ \norm{\ptial{\nu_G} (u(t) - u(s))}_{\Lq{q}(G)} \right. \\
		& \left. \qquad
		+ \norm{\ptial{\nuSigma} (u(t) - u(s))}_{\Lq{q}(\Sigma)}
		+ \norm{\jump{\pi(t) - \pi(s)}}_{\Lq{q}(\Sigma)} 
		+ \norm{(\pi(t) - \pi(s))}_{\Lq{q}(G)} \right. \\
		& \left. \qquad
		+ \norm{f_u(t) - f_u(s)}_{\Lq{q}(\OM)}
		\right),
	\end{align*}
	for almost all $ t \in J $.
	Since $ \bbeo_u \hookrightarrow \WO{\onehalf}(J; \W{1}(\OM \backslash \Sigma)) $, trace theory implies that
	\begin{align*}
		& \rv{(\ptial{k} u_l)}_K \in \WO{\onehalf - \frac{1}{2q}}(J; \Lq{q}(K)), 
	\end{align*}
	for $ k,l \in \{1,2,3\} $ and $ K \in \{\Sigma, G, S\} $. From the embedding $ \W{s} \hookrightarrow \W{s - \epsilon} $, $ \epsilon > 0 $, one gets $ \pi \in \WO{\alpha}(J; \Lq{q}(\OM)) $ and the desired estimate.
%
\end{proof}
\begin{remark}
	Note that for unbounded domains, e.g., the (bent) quarter space and the half space with a (bent) interface, one can always reduce the unbounded domain to a bounded domain with the same part of boundary by suitable extension. Therefore, in the following localization procedure as in Section \ref{Subsection: Localization procedure}, we will apply Proposition \ref{Subsection: Regularity pressure} directly without any further explanations.
\end{remark}
\subsection{Reductions} \label{Subsection: Reduction-general}
In this section, we reduce the problem \eqref{Linearized TwoPhase} to the case 
\begin{equation*}
	(u_0, f_u, f_d, g_1 \cdot \nuSigma, g_5 \cdot \nu_S) = 0.
\end{equation*}
To this end, consider the parabolic transmission problem 
\begin{equation} \label{Reduction: Parabolic transmission}
	\begin{alignedat}{3}
		\rho \pt \baru - \mu \Delta \baru & = f_u, && \tin \OM \backslash \Sigma \times J, \\
		\jump{\baru} & = g_1, && \tin \Sigma \times J, \\
		\jump{\mu (\nabla \baru + \tran{\nabla} \baru) \nuSigma} & = g_2, && \ton \Sigma \times J, \\
		\tran{(\baru_1, \baru_3)} & = g_3, && \ton G \backslash \partial \Sigma \times J, \\
		2 \mu \ptial{2} \baru_2 & = g_4, && \ton G \backslash \partial \Sigma \times J, \\
		\baru & = g_5, && \ton S \times J, \\
		\baru(0) & = u_0, && \tin \OM \backslash \Sigma,
	\end{alignedat} 
\end{equation}
with $ (u_0, f_u, g_j) $, $ j \in \{1,...,5\} $, in suitable function spaces. Then by Theorem \eqref{AppendixTheorem: Parabolic}, \eqref{Reduction: Parabolic transmission} admits a unique solution
\begin{gather*}
	\baru \in \W{1}(J; \Lq{q}(\OM)^3) \cap \Lq{q}(J; \W{2}(\OM \backslash \Sigma)^3),
\end{gather*}
which implies that $ (u - \baru, \pi) $ sovles \eqref{Linearized TwoPhase} with $ (u_0, f_u, g_1, g_3, g_5) = 0 $. Next, solving the elliptic transmission problem
\begin{equation} \label{Reduction: Elliptic transmission}
	\begin{alignedat}{3}
		\Delta \phi & = f_d - \Div \baru, && \tin \OM \backslash \Sigma, \\
		\jump{\rho \phi} & = 0, && \ton \Sigma, \\
		\jump{\ptial{\nuSigma} \phi} & = 0, && \ton \Sigma, \\
		\rho \phi & = 0, && \ton G \backslash \partial \Sigma, \\
		\ptial{\nu_S}\phi & = 0, && \ton S
	\end{alignedat}
\end{equation}
yields a unique solution $ \phi $ with regularity
\begin{gather*}
	\nabla \phi \in \WO{1}(J; \Lq{q}(\OM)^3) \cap \Lq{q}(J; \W{2}(\OM \backslash \Sigma)^3),
\end{gather*}
thanks to Theorem \ref{AppendixTheorem: Laplace}.
Now for $ (u, \pi) $ the solution of \eqref{Linearized TwoPhase}, define
\begin{align*}
	(\tu, \tpi) :=
	(u - \baru - \nabla \phi, \pi + \rho \pt \phi - \mu \Delta \phi).
\end{align*}
It is clear that $ (\tu, \tpi) $ solves \eqref{Linearized TwoPhase} with $ (u_0, f_u, f_d, g_1 \cdot \nuSigma, g_5 \cdot \nu_S) = 0 $ and other modified data in right regularity classes having vanishing time traces at $ t = 0 $.

\subsection{Localization procedure} \label{Subsection: Localization procedure}
Firstly, we are devoted to prove Theorem \ref{Theorem: Cylinder-TwoPhase}, which is the most crucial part of this work. The proof is based on the localization procedure with the model problems we considered in Section \ref{Section: Model problems}.

\begin{proof}[\textbf{Proof of Theorem \ref{Theorem: Cylinder-TwoPhase}}]
	Following the localization method in e.g. \cite{DHP2003}, \cite{PS2016} and \cite{Wilke2020}, we split the proof into two parts.
	
	\textbf{(1). Existence of a left inverse.} Let $ (u, \pi) $ be a solution of \eqref{Linearized TwoPhase}. By the reduction argument in Section \ref{Subsection: Reduction-general}, there exists $ (\baru, \barpi) $ such that $ (\tu, \tpi) := (u, \pi) - (\baru, \barpi) $ solves the problem
	\begin{equation} \label{TwoPhase: reduced}
		\begin{alignedat}{3}
			\rho \pt \tu - \mu \Delta \tu + \nabla \tpi & = 0, && \tin \OM \backslash \Sigma \times J, \\
			\Div \tu & = 0, && \tin \OM \backslash \Sigma \times J, \\
			\jump{\tu} & = \tg_1, && \ton \Sigma \times J, \\
			\jump{- \tpi \bbi + \mu (\nabla \tu + \tran{\nabla} \tu)} \nuSigma & = \tg_2, && \ton \Sigma \times J, \\
			\tran{(\tu_1, \tu_3)} & = \tg_3, && \ton G \backslash \partial \Sigma \times J, \\
			- \tpi + 2 \mu \ptial{2} \tu_2 & = \tg_4, && \ton G \backslash \partial \Sigma \times J, \\
			\tu & = \tg_5, && \ton S, \times J \\
			\tu(0) & = 0, && \tin \OM \backslash \Sigma.
		\end{alignedat} 
	\end{equation}
	with $ (\tg_1 \cdot \nuSigma, \tg_5 \cdot \nu_S) = 0 $. Choose open sets $ U_k = B_r(x_k) $ with radius $ r > 0 $ and centers $ x_k $ such that 
	$ \partial \Sigma \subset \bigcup_{k = 5}^{N_1} U_k $, $ \partial S \subset \bigcup_{k = N_1 + 1}^{N} U_k $ and choose $ r > 0 $ small enough such that corresponding solution operators from Section \ref{Section: Stokes-Bent-Quarter} and \ref{Section: TwoPhase-Bent-halfspace} are well-defined. By Proposition \ref{AppendixPropositon: POU-Omega}, there exist open and connected sets such that 
	\begin{itemize}
		\item $ U_0 \cap \Bar{\OM} \neq \emptyset $, $ U_0 \cap \Bar{G} = \emptyset $,
		\item $ U_i \cap G_1 \neq \emptyset $, $ U_j \cap G_2 \neq \emptyset $, $ U_k \cap (\Sigma \cup S) = \emptyset $, $ i = 1, 2 $, $ j = 3, 4 $, $ k = i,j $, 
	\end{itemize}
	and a family of functions $ \{ \vp_k \}_{k = 0}^N \subset C_c^3(\bbr^3; [0,1]) $ such that $ \Bar{\OM} \subset \bigcup_{k = 0}^N U_k $, $ \supp \vp_k \subset U_k $, $ \sum_{k = 0}^{N} \vp_k = 1 $ and $ \ptial{\nuSigma} \vp_k(x) = 0 $, $ \ptial{\nu_S} \vp_k(x) = 0 $, for $ x \in U_k \cap (\partial \Sigma \cup \partial S) $, $ k \geq 5 $. 
	
	Now define $ (\tu_k, \tpi_k) := (\tu, \tpi) \vp_k $, the data $ \tg_{jk} := \tg_j \vp_k $, as well as the domain $ \OM^k := \OM \cap U_k $, where $ U_k $ is an open set if $ k = 0 $, a half space if $ k = 1,...,4 $, a bent quarter space if $ k = N_1 + 1,...,N $, a half space with a bent interface if $ k = 5,...,N_1 $. Moreover, let $ \Sigma^k := \Sigma \cap U_k $, $ G^k := G \cap U_k $ and $ S^k := S \cap U_k $. Then $ (\tu_k, \tpi_k) $ satisfies the problem
	\begin{equation} \label{TwoPhase: Localization}
		\begin{alignedat}{3}
			\rho \pt \tu_k - \mu \Delta \tu_k + \nabla \tpi_k & = F_k(\tu, \tpi), && \tin \OM^k \backslash \Sigma^k \times J, \\
			\Div \tu_k & = F_{dk}(\tu), && \tin \OM^k \backslash \Sigma^k \times J, \\
			\jump{\tu_k} & = \tg_{1k}, && \ton \Sigma^k \times J, \\
			\jump{- \tpi_k \bbi + \mu (\nabla \tu_k + \tran{\nabla} \tu_k)} \nu_{\Sigma^k} & = \tg_{2k} + G_{2k}(\tu), && \ton \Sigma^k \times J, \\
			\tran{((\tu_k)_1, (\tu_k)_3)} & = \tg_{3k}, && \ton G^k \backslash \partial \Sigma^k \times J, \\
			- \tpi_k + 2 \mu \ptial{2} (\tu_k)_2 & = \tg_{4k} + G_{4k}(\tu), && \ton G^k \backslash \partial \Sigma^k \times J, \\
			\tu_k & = \tg_{5k}, && \ton S^k \times J, \\
			\tu_k(0) & = 0, && \tin \OM^k \backslash \Sigma^k,
		\end{alignedat} 
	\end{equation}
	where
	\begin{gather*}
		F_k(\tu, \tpi) := [\nabla, \vp_k] \tpi - \mu [\Delta, \vp_k] \tu, \quad
		F_{dk}(\tu) := \tu \cdot \nabla \vp_k, \\
		G_{2k}(\tu) := \jump{\mu(\nabla \vp_k \otimes \tu + \tu \otimes \nabla \vp_k)} \nu_{\Sigma^k}, \quad
		G_{4k}(\tu) := 2 \mu \ptial{2} \vp_k \tu_2.
	\end{gather*}
	For $ k = 0 $, we extend $ \Sigma $ to $ \Tilde{\Sigma} $ and respectively $ S $ to $ \Tilde{S} $ smoothly, such that 
	\begin{equation*}
		 \Tilde{\Sigma} \cap U_0 = \Sigma \cap U_0, \quad 
		 \Tilde{\OM} \cap U_0 = \OM \cap U_0, \quad
		 \Tilde{\Sigma} \subset \Tilde{\OM}.
	\end{equation*}
	Then solving a two-phase Stokes problem in a bounded domain with smooth boundary and interface yields the unique solution of this local problem by \cite[Theorem 8.1.4]{PS2016}. If $ k = 1,...,4 $, we have the half-space Stokes problem with outflow boundary conditions. This problem was investigated in \cite[Section 6]{BP2007} and \cite[Section 7.2]{PS2016}. For $ k = 5, ..., N_1 $ and $ k = N_1 + 1, ..., N  $, we obtain a bent quarter-space Stokes problem and respectively, a half-space two-phase Stokes problem with a bent interface, which are solvable according to Section \ref{Section: Stokes-Bent-Quarter} and \ref{Section: TwoPhase-Bent-halfspace}. Hence, the solution operators for the charts $ U_k $, $ k \geq 5 $, are well-defined by the results in Section \ref{Section: Stokes-Bent-Quarter} and \ref{Section: TwoPhase-Bent-halfspace}. We denote the corresponding solution operators for each chart by $ \cS_k $.

	Next, we want to reduce $ F_{dk}(\tu) $, since we do not have enough time regularity for it, while $ F_k $, $ G_{2k} $ and $ G_{4k} $ are endowed with extra time regularity due to the regularity of $ \tu $ and Proposition \ref{Proposition: pressure regularity}. More specifically, if $ \tu \in \bbeo_u(J) \hookrightarrow \WO{1/2}(J; \W{1}(\OM^k)^3) $, then 
	\begin{equation*}
		[\Delta, \vp_k] \tu \in \WO{1/2}(J; \Lq{q}(\OM^k)^3) \cap \Lq{q}(J; \W{1}(\OM^k \backslash \Sigma^k)^3).
	\end{equation*}
	From Proposition \ref{Proposition: pressure regularity}, we know that 
	\begin{equation*} 
		\tpi \in  \WO{\alpha}(J; \Lq{q}(\OM^k)) \cap \Lq{q}(J; \W{1}(\OM^k \backslash \Sigma^k)),
	\end{equation*}
	and hence 
	\begin{equation*}
		[\nabla, \vp_k] \tpi \in  \WO{\alpha}(J; \Lq{q}(\OM^k)^3) \cap \Lq{q}(J; \W{1}(\OM^k \backslash \Sigma^k)^3),
	\end{equation*}
	for $ 0 < \alpha < 1/2 - 1/2q $. Then
	\begin{equation*}
		F_k(\tu, \tpi) \in  \WO{\alpha}(J; \Lq{q}(\OM^k)) \cap \Lq{q}(J; \W{1}(\OM^k \backslash \Sigma^k)),
	\end{equation*}
	as well as
	\begin{gather*}
		G_{2k}(\tu) \in  \WO{1 - 1/2q}(J; \Lq{q}(\Sigma^k)^3) \cap \WO{1/2}(J; \W{1 - 1/q}(\Sigma^k)^3), \\
		G_{4k}(\tu) \in  \WO{1 - 1/2q}(J; \Lq{q}(G^k)) \cap \WO{1/2}(J; \W{1 - 1/q}(G^k \backslash \partial \Sigma^k)), 
	\end{gather*}
	by similar argument and the trace theory. However, $ F_{dk} \in \bbeo_u(J) $ is an exception without enough regularity. Thus, we reduce it by solving an auxiliary elliptic transmission problem
	\begin{equation} \label{LocalizationReduction: Elliptic transmission}
		\begin{alignedat}{3}
			\Delta \phi_k & = F_{dk}(\tu), && \tin \OM^k \backslash \Sigma^k, \\
			\jump{\rho \phi_k} & = 0, && \ton \Sigma^k, \\
			\jump{\ptial{\nuSigma} \phi_k} & = 0, && \ton \Sigma^k, \\
			\rho \phi_k & = 0, && \ton G^k \backslash \partial \Sigma^k, \\
			\ptial{\nu_S} \phi_k & = 0, && \ton S^k,
		\end{alignedat}
	\end{equation}
	with Theorem \ref{Appendix: Elliptic transmission} to obtain a unique solution $ \phi_k $ with regularity
	\begin{align*}
		& \nabla \phi_k \in \bbeo_\phi(J) := \WO{1}(J; \W{1}(\OM^k \backslash \Sigma^k)^3) \cap \Lq{q}(J; \W{3}(\OM^k \backslash \Sigma^k)^3) \\
		& \qquad \qquad \qquad \hookrightarrow \WO{1}(J; \W{1}(\OM^k \backslash \Sigma^k)^3) \cap \WO{1/2}(J; \W{2}(\OM^k \backslash \Sigma^k)^3),
	\end{align*}
	and the estimate
	\begin{equation*}
		\norm{\nabla \phi_k}_{\bbeo_\phi(J)}
		\leq C_N \norm{\tu}_{\bbeo_u(J)},
	\end{equation*}
	where $ C_N $ depends on $ N $, $ \OM $ and $ \Sigma $ but does not depend on $ J $. Moreover, trace theory yields that 
	\begin{gather*}
		\jump{\nabla \phi_k} \in \WO{1 - 1/2q}(J; \W{1}(\Sigma^k)^3) \cap \WO{1/2}(J; \W{2 - 1/q}(\Sigma^k)^3), \\
		\rv{\nabla \phi_k}_{G^k} \in \WO{1 - 1/2q}(J; \W{1}(G^k \backslash \partial \Sigma^k)^3) \cap \WO{1/2}(J; \W{2 - 1/q}(G^k \backslash \partial \Sigma^k)^3), \\
		\rv{\nabla \phi_k}_{S^k} \in \WO{1 - 1/2q}(J; \W{1}(S^k)^3) \cap \WO{1/2}(J; \W{2 - 1/q}(S^k)^3),
	\end{gather*}
	with vanishing normal part on $ \Sigma^k $ and $ S^k $ and 
	\begin{gather*}
		\jump{\mu \nabla^2 \phi_k} \in \WO{1 - 1/2q}(J; \Lq{q}(\Sigma^k)^{3 \times 3}) \cap \WO{1/2}(J; \W{1 - 1/q}(\Sigma^k)^{3 \times 3}), \\
		\rv{\mu \nabla^2 \phi_k}_{S^k} \in \WO{1 - 1/2q}(J; \Lq{q}(S^k)^{3 \times 3}) \cap \WO{1/2}(J; \W{1 - 1/q}(S^k)^{3 \times 3}),
	\end{gather*}
	whenever they exist. Define
	\begin{equation*}
		(\hatu_k, \hpi_k) := (\tu_k - \nabla \phi_k, \tpi_k + \rho \pt \phi_k - \mu \Delta \phi_k).
	\end{equation*}
	Then we arrive at the system
	\begin{equation} \label{TwoPhase: Localization-reduced}
		\begin{alignedat}{3}
			\rho \pt \hatu_k - \mu \Delta \hatu_k + \nabla \hpi_k & = F_k(\tu, \tpi), && \tin \OM^k \backslash \Sigma^k \times J, \\
			\Div \hatu_k & = 0, && \tin \OM^k \backslash \Sigma^k \times J, \\
			\jump{\hatu_k} & = \tg_{1k} - \jump{\nabla \phi_k}, && \ton \Sigma^k \times J, \\
			\jump{- \hpi_k \bbi + \mu (\nabla \hatu_k + \tran{\nabla} \hatu_k)} \nu_{\Sigma^k} & = \tg_{2k} + \Hat{G}_{2k}(\tu), && \ton \Sigma^k \times J, \\
			\tran{((\hatu_k)_1, (\hatu_k)_3)} & = \tg_{3k} - \tran{(\ptial{1} \phi_k, \ptial{3} \phi_k)}, && \ton G^k \backslash \partial \Sigma^k \times J, \\
			- \hpi_k + 2 \mu \ptial{2} (\hatu_k)_2 & = \tg_{4k} + \Hat{G}_{4k}(\tu), && \ton G^k \backslash \partial \Sigma^k \times J, \\
			\hatu_k & = \tg_{5k} - \nabla \phi_k, && \ton S^k \times J, \\
			\hatu_k(0) & = 0, && \tin \OM^k \backslash \Sigma^k,
		\end{alignedat} 
	\end{equation}
	where
	\begin{align*}
		\Hat{G}_{2k}(\tu) & := G_{2k}(\tu) - \jump{2 \mu \nabla^2 \phi_k} \nu_{\Sigma^k} + \jump{\mu \Delta \phi_k} \nu_{\Sigma^k}, \\
		\Hat{G}_{4k}(\tu) & := G_{4k}(\tu) - 2 \mu \ptial{2}^2 \phi_k + \mu \Delta \phi_k. 
	\end{align*}
	By the solution operators $ \cS_k $, one may rewrite \eqref{TwoPhase: Localization-reduced} as
	\begin{equation} \label{Abstract: twophase}
		(\hatu_k, \hpi_k) = \cS_k (\cD_k + \cR_k(\tu, \tpi)),
	\end{equation}
	where 
	\begin{equation*}
		\cD_k 
		:= \tran{(0, 0, \tg_{1k}, \tg_{2k}, \tg_{3k}, \tg_{4k}, \tg_{5k})}
	\end{equation*}denotes the given data and 
	\begin{equation*}
		\cR_k(\tu, \tpi)
		:= \tran{(F_k(\tu, \tpi), 0, - \jump{\nabla \phi_k}, \Hat{G}_{2k}(\tu), - (\ptial{1} \phi_k, \ptial{3} \phi_k), \Hat{G}_{4k}(\tu), - \nabla \phi_k)}
	\end{equation*}
	represents the remaining part on the right-hand side of \eqref{TwoPhase: Localization-reduced}. 
	
	Define a cutoff function $ \{ \eta_k \}_{k = 0}^N \subset C_c^\infty(U_k) $ such that $ \rv{\eta_k}_{\supp \vp_k} = 0 $. Then $ \tu = \sum_{k = 0}^N \tu_k \eta_k $, $ \tpi = \sum_{k = 0}^N \tpi_k \eta_k $. Multiplying \eqref{Abstract: twophase} by $ \eta_k $, replacing $ (\hatu_k, \hpi_k) \eta_k $ by $ (\tu_k - \nabla \phi_k, \tpi_k + \rho \pt \phi_k - \mu \Delta \phi_k) \eta_k $ and rearranging the equation, one obtains
	\begin{equation*} 
		(\tu_k, \tpi_k) \eta_k = \cS_k (\cD_k \eta_k + \cR_k(\tu, \tpi) \eta_k) + (\nabla \phi_k, - \rho \pt \phi_k + \mu \Delta \phi_k) \eta_k.
	\end{equation*}
	Now recalling the function spaces in Section \ref{Subsection: Linearized system-Lagrangian}, we claim that there exist a $ \delta > 0 $ and a constant $ C $ independent of $ T > 0 $, such that for $ \alpha > 1/2q $,
	\begin{align}
		\norm{\cR_k(\tu, \tpi) \eta_k}_{\bbfo(J)}
		& \leq C T^\delta \norm{(\tu, \tpi)}_{\bbeo(J)}, \label{Eqs: R_k}\\
		\norm{\nabla \phi_k \eta_k}_{\bbeo_u(J)}
		& \leq C T^\delta \left( \norm{(\tu, \tpi)}_{\bbeo(J)} + \norm{\cD_k}_{\bbfo(J)} \right), \label{Eqs: nabla phi_k} \\
		\norm{(- \rho \pt \phi_k + \mu \Delta \phi_k) \eta_k}_{\bbeo_\pi(J)}
		& \leq C T^\delta \left( \norm{(\tu, \tpi)}_{\bbeo(J)} + \norm{\cD_k}_{\bbfo(J)} \right). \label{Eqs: partial_t phi_k}
	\end{align}
	It is easy to verify \eqref{Eqs: R_k}, since we have given associated additional time regularity for the remaining terms. We present the estimate for $ F_k(\tu, \tpi) $ as an example. For $ \alpha > 1/2q $, we have
	\begin{align*}
		\norm{\eta_k F_k(\tu, \tpi)}_{\bbf_1(J)}
		\leq T^{\frac{1}{2q}} \norm{F_k(\tu, \tpi)}_{\Lq{2q}(J; \Lq{q}(\OM^k)^3)} 
		\leq C T^{\frac{1}{2q}} \norm{F_k(\tu, \tpi)}_{\WO{\alpha}(J; \Lq{q}(\OM^k)^3)}.
	\end{align*}
	By Proposition \ref{Proposition: pressure regularity}, we get
	\begin{align*}
		& \norm{(- \rho \pt \phi_k + \mu \Delta \phi_k) \eta_k}_{\WO{\alpha}(J; \Lq{q}(\OM^k))}
		\leq \norm{\tpi \eta_k - \hpi_k}_{\WO{\alpha}(J; \Lq{q}(\OM^k))} \\
		& \qquad \qquad \qquad \qquad \qquad
		\leq C \left( \norm{(\tu, \tpi)}_{\bbeo(J)} + \norm{\cD_k}_{\bbfo(J)} + \norm{\cR_k(\tu, \tpi) \eta_k}_{\bbfo(J)} \right).
	\end{align*}
	Since $ \nabla \phi_k \in \WO{1}(J; \W{1}(\OM^k \backslash \Sigma^k)^3) $, one has
	\begin{align*}
		& \pt \phi_k \in \WO{\alpha}(J; \Lq{q}(\OM^k)) \cap \Lq{q}(J; \W{2}(\OM^k \backslash \Sigma^k)) 
		\hookrightarrow \WO{\alpha - \epsilon}(J; \W{1}(\OM^k \backslash \Sigma^k)),
	\end{align*}
	for $ \epsilon > 0 $ small. Then
	\begin{align*}
		& \norm{\nabla \phi_k \eta_k}_{\bbeo_u(J)}
		\leq C \norm{\nabla \phi_k}_{\Lq{q}(J; \W{2}(\OM^k \backslash \Sigma^k)^3)}
		+ \norm{\phi_k}_{\WO{1}(J; \W{1}(\OM^k \backslash \Sigma^k)^3)} \\
		& \leq C T^\frac{1}{2q} \norm{\nabla \phi_k}_{\WO{\onehalf}(J; \W{2}(\OM^k \backslash \Sigma^k)^3)}
		+ \norm{\phi_k}_{\Lq{q}(J; \W{1}(\OM^k \backslash \Sigma^k))}
		+ \norm{\pt \phi_k}_{\Lq{q}(J; \W{1}(\OM^k \backslash \Sigma^k))} \\
		& \qquad \leq C T^\frac{1}{2q} \norm{\nabla \phi_k}_{\WO{\onehalf}(J; \W{2}(\OM^k \backslash \Sigma^k)^3)}
		+ C T^\frac{1}{2q} \norm{\pt \phi_k}_{\WO{\alpha - \epsilon}(J; \W{1}(\OM^k \backslash \Sigma^k))} \\
		& \qquad
		\leq C T^\frac{1}{2q} \left( \norm{(\tu, \tpi)}_{\bbeo(J)} + \norm{\cD_k}_{\bbfo(J)} \right),
	\end{align*}
	and 
	\begin{align*}
		& \norm{(- \rho \pt \phi_k + \mu \Delta \phi_k) \eta_k}_{\bbeo_\pi(J)} \\
		& \qquad 
		\leq C T^\frac{1}{2q} \norm{\pt \phi_k}_{\WO{\alpha - \epsilon}(J; \dW{1}(\OM^k \backslash \Sigma^k))}
		+ C T^\frac{1}{2q} \norm{\nabla \phi_k}_{\WO{\onehalf}(J; \W{2}(\OM^k \backslash \Sigma^k)^3)} \\
		&  \qquad \qquad 
		+ C T^\frac{1}{2q} \norm{\jump{\Delta \phi_k}}_{\WO{\onehalf}(J; \W{1 - \frac{1}{q}}(\Sigma^k))}
		+ C T^\onehalf \norm{\jump{\Delta \phi_k}}_{\WO{1 - \frac{1}{2q}}(J; \Lq{q}(\Sigma^k))} \\
		&  \qquad \qquad 
		+ C T^\frac{1}{2q} \norm{\rv{\Delta \phi_k}_{S^k}}_{\WO{\onehalf}(J; \W{1 - \frac{1}{q}}(\Sigma^k))}
		+ C T^\onehalf \norm{\rv{\Delta \phi_k}_{S^k}}_{\WO{1 - \frac{1}{2q}}(J; \Lq{q}(\Sigma^k))} \\
		& \qquad \leq C T^\frac{1}{2q} \left( \norm{(\tu, \tpi)}_{\bbeo(J)} + \norm{\cD_k}_{\bbfo(J)} \right),
	\end{align*}
	where we used
	\begin{align*}
		& \norm{\pt \phi_k}_{\Lq{q}(J; \dW{1}(\OM^k \backslash \Sigma^k))}
		\leq T^{\frac{1}{2q}} \norm{\pt \phi_k}_{\Lq{2q}(J; \dW{1}(\OM^k \backslash \Sigma^k))} \\
		& \qquad \qquad
		\leq C T^{\frac{1}{2q}} \norm{\pt \phi_k}_{\W{\frac{\alpha}{2} - \frac{\epsilon}{2}}(J; \dW{1}(\OM^k \backslash \Sigma^k))} 
		\leq C T^{\frac{1}{2q}} \norm{\pt \phi_k}_{\W{\alpha - \epsilon}(J; \dW{1}(\OM^k \backslash \Sigma^k))}, 
	\end{align*}
	for some $ 1/2q < \alpha < 1/2 - 1/2q $ and $ \epsilon $ small enough. Then we get \eqref{Eqs: nabla phi_k} and \eqref{Eqs: partial_t phi_k}. Consequently, there is a constant $ \delta > 0 $, such that
	\begin{equation*}
		\norm{(\tu_k, \tpi_k) \eta_k}_{\bbeo(J)}
		\leq C \left( \norm{\cD_k}_{\bbfo(J)} 
		+ T^\delta \norm{(\tu, \tpi)}_{\bbeo(J)} \right),
	\end{equation*}
	where $ C > 0 $ does not depend on $ T > 0 $. Taking the sum over all charts $ k = 0,...,N $, one obtains 
	\begin{equation*}
		\norm{(\tu, \tpi)}_{\bbeo(J)}
		\leq C \left( \norm{\cD}_{\bbfo(J)} 
		+ T^\delta \norm{(\tu, \tpi)}_{\bbeo(J)} \right),
	\end{equation*}
	where $ \cD $ denotes the data in \eqref{TwoPhase: reduced}. Then choosing $ T > 0 $ sufficiently small yields the \textit{a priori} estimate
	\begin{equation*}
		\norm{(\tu, \tpi)}_{\bbeo(J)}
		\leq C \norm{\cD}_{\bbfo(J)}.
	\end{equation*}
	Hence we may conclude that the operator $ \cL : \bbeo \rightarrow \bbfo $ defined by the left-hand side of \eqref{TwoPhase: reduced} is injective and has closed range and there is a left inverse $ \cS $ for $ \cL $ such that $ \cS \cL z = z $ for all $ z \in \bbeo(J) $.
	
	\textbf{(2). Existence of a right inverse.}
	Now we are in the position to prove the existence of a right inverse. Given data 
	\begin{equation*}
		F := \tran{(f_u, f_d, g_1, g_2, g_3, g_4, g_5)} \in \bbf(J), \quad u_0 \in X_{\gamma, u},
	\end{equation*}
	satisfying the compatibility conditions \eqref{Compatibility: TwoPhase-initial-cylinder} and \eqref{Compatibility: TwoPhase-contact-cylinder}. As stated in Section \ref{Subsection: Reduction-general}, without loss of generality one may assume that $ u_0 = 0 $, which means that the time traces of all the data at $ t = 0 $ vanish whenever they exist. 
	
	Let $ \baru, \nabla \phi $ be the unique solution of \eqref{Reduction: Parabolic transmission} and \eqref{Reduction: Elliptic transmission}, respectively. Define
	\begin{equation*}
		(\tu, \tpi) := (\baru - \nabla \phi, - \pt \phi + \mu \Delta \phi), \quad 
		\Tilde{\cS} F := (\tu, \tpi).
	\end{equation*}
	Then it follows that
	\begin{equation*}
		\cL \Tilde{\cS} F = 
		\tran{\left(f_u, f_d, g_1 + G_1(\phi), g_2 + G_2(\phi), g_3 + G_3(\phi), g_4 + G_4(\phi), g_5 + G_5(\phi), 0\right)},
	\end{equation*}
	where 
	\begin{gather*}
		G_1(\phi) := - \jump{\nabla \phi}, \quad 
		G_2(\phi) := - \jump{2 \mu \nabla^2 \phi} \nuSigma + \jump{\mu \Delta \phi} \nuSigma, \\
		G_3(\phi) := \tran{(- \ptial{1} \phi, - \ptial{3} \phi)}, \quad 
		G_4(\phi) := - 2 \mu \ptial{2}^2 \phi + \mu \Delta \phi, \quad 
		G_5(\phi) := - \nabla \phi.
	\end{gather*}
	
	In the following, by localizaiton we consider the problem
	\begin{equation} \label{TwoPhase: Localization-right}
		\begin{alignedat}{3}
			\rho \pt \tu_k - \mu \Delta \tu_k + \nabla \tpi_k & = 0, && \tin \OM^k \backslash \Sigma^k \times J, \\
			\Div \tu_k & = 0, && \tin \OM^k \backslash \Sigma^k \times J, \\
			\jump{\tu_k} & = G_{1k}(\phi), && \ton \Sigma^k \times J, \\
			\jump{- \tpi_k \bbi + \mu (\nabla \tu_k + \tran{\nabla} \tu_k)} \nu_{\Sigma^k} & = G_{2k}(\phi), && \ton \Sigma^k \times J, \\
			\tran{((\tu_k)_1, (\tu_k)_3)} & = G_{3k}(\phi), && \ton G^k \backslash \partial \Sigma^k \times J, \\
			- \tpi_k + 2 \mu \ptial{2} (\tu_k)_2 & = G_{4k}(\phi), && \ton G^k \backslash \partial \Sigma^k \times J, \\
			\tu_k & = G_{5k}(\phi), && \ton S^k \times J, \\
			\tu_k(0) & = 0, && \tin \OM^k \backslash \Sigma^k,
		\end{alignedat} 
	\end{equation}
	where $ G_{jk}(\phi) := G_{j}(\phi) \vp_k $, $ j \in \{1,...,5\} $. Now, we are going to check if $ G_{jk}(\phi) $ satisfy all the relevant compatibility condition at $ \partial \Sigma^k $ and $ \partial S^k $, whenever they exist. For $ k = 0, ...,4 $, one does not need compatibility condition for $ G_{jk} $, $ j = 1,2,5 $. Since $ \phi $ has vanishing trace at $ t = 0 $, one obtains $ \rv{(G_{3k}, G_{4k})}_{t = 0} = 0 $. For $ k = 5, ...,N_1 $, i.e., the bent quarter-space Stokes problem, it is obviously that $ G_{3k} = \tran{((G_{5k})_1, (G_{5k})_3)} $, which fulfills the compatibility condition on the contact line $ \partial S^k $. For $ k = N_1, ..., N $, namely, the half-space two-phase Stokes problem with a bent interface, we have $ \jump{G_{3k}} = \tran{((G_{1k})_1, (G_{1k})_3)} $ and 
	\begin{gather*}
		\begin{aligned}
			(G_{2k})_1 & = \jump{2 \mu \ptial{1} (G_{3k})_1 \nu_{\Sigma^k} \cdot e_1 + \mu (\ptial{1} (G_{3k})_2 + \ptial{3} (G_{3k})_1) \nu_{\Sigma^k} \cdot e_3} \\
				& \quad + \jump{G_{4k}} \nu_{\Sigma^k} \cdot e_1 + \jump{2 \mu (\ptial{1} (G_{3k})_1 + \ptial{3} (G_{3k})_2)} \nu_{\Sigma^k} \cdot e_1,
		\end{aligned} \\
		\begin{aligned}
			(G_{2k})_3 & = \jump{2 \mu \ptial{3} (G_{3k})_2 \nu_{\Sigma^k} \cdot e_3 + \mu (\ptial{1} (G_{3k})_2 + \ptial{3} (G_{3k})_1) \nu_{\Sigma^k} \cdot e_1} \\
				& \quad + \jump{G_{4k}} \nu_{\Sigma^k} \cdot e_3 + \jump{2 \mu (\ptial{1} (G_{3k})_1 + \ptial{3} (G_{3k})_2)} \nu_{\Sigma^k} \cdot e_3,
		\end{aligned}
	\end{gather*}
	at the contact line $ \partial \Sigma^k $, which verifies the compatibility condition \eqref{Compatibility:two-phase-bent}.
	
	Therefore, according to the model problems we established in Section \ref{Section: Model problems}, for each $ k \in \{0,1,...,N\} $, there exists a unique solution $ (\tu_k, \tpi_k) $ of \eqref{TwoPhase: Localization-right} in right regularity class. Define a cutoff function $ \{ \eta_k \}_{k = 0}^N \subset C_c^\infty(U_k) $ such that $ \rv{\eta_k}_{\supp \vp_k} = 0 $. Solving the elliptic transmission problem
	\begin{equation} \label{Elliptic transmission: right inverse}
		\begin{alignedat}{3}
			\Delta \phi_k & = \rv{(\tu_k \cdot \nabla \eta_k)}_\OM, && \tin \OM \backslash \Sigma, \\
			\jump{\rho \phi_k} & = 0, && \ton \Sigma, \\
			\jump{\ptial{\nuSigma} \phi_k} & = 0, && \ton \Sigma, \\
			\rho \phi_k & = 0, && \ton G \backslash \partial \Sigma, \\
			\ptial{\nu_S} \phi_k & = 0, && \ton S,
		\end{alignedat}
	\end{equation}
	yields a unique solution $ \phi_k $ with regularity
	\begin{equation*}
		\nabla \phi_k \in \WO{1}(J; \W{1}(\OM \backslash \Sigma)^3) \cap \Lq{q}(J; \W{3}(\OM \backslash \Sigma)^3).
	\end{equation*}
	Finally, define
	\begin{equation*}
		\Bar{\cS} F := \sum_{k = 0}^N (\tu_k \eta_k - \nabla \phi_k, \tpi_k \eta_k - \rho \pt \phi_k + \mu \Delta \phi_k).
	\end{equation*}
	Then it holds that
	\begin{equation*}
		\cL \Bar{\cS} F = \sum_{k = 0}^N
		\left(
			\begin{gathered}
				- \mu [\Delta, \eta_k] \tu_k + [\nabla, \eta_k] \tu_k \\
				0 \\
				\eta_k G_{1k}(\phi) + G_{1k}(\phi_k) \\
				\eta_k G_{2k}(\phi) + \Hat{G}_{2k}(\phi_k) \\
				\eta_k G_{3k}(\phi) + G_{3k}(\phi_k) \\
				\eta_k G_{4k}(\phi) + \Hat{G}_{4k}(\phi_k) \\
				\eta_k G_{5k}(\phi) + G_{5k}(\phi_k) \\
				0
			\end{gathered}
		\right),
	\end{equation*}
	where
	\begin{align*}
		\Hat{G}_{2k}(\phi_k) & := G_{2k}(\phi_k) - \jump{\mu(\nabla \eta_k \otimes \tu_k + \tu_k \otimes \nabla \eta_k)} \nu_{\Sigma^k}, \\
		\Hat{G}_{4k}(\phi_k) & := G_{4k}(\phi_k) - 2 \mu \ptial{2} \phi_k (\tu_k)_2. 
	\end{align*}
	From the construction of $ \eta_k $, we know that 
	\begin{equation*}
		G_j(\phi) = \sum_{k = 0}^N \eta_k G_{jk} (\phi).
	\end{equation*}
	Let $ \Hat{\cS} F := \Tilde{\cS} F - \Bar{\cS} F $, it follows that
	\begin{equation*}
		\cL \Hat{\cS} F = \cL \Tilde{\cS} F - \cL \Bar{\cS} F = F - \cR F,
	\end{equation*}
	where
	\begin{equation*}
		\cR F := \sum_{k = 0}^N 
		\tran{\left(- \mu [\Delta, \eta_k] \tu_k + [\nabla, \eta_k] \tu_k, 0, G_{1k}, \Hat{G}_{2k},	G_{3k}, \Hat{G}_{4k}, G_{5k}, 0 \right)}.
	\end{equation*}
	As in the first part of the proof, since we have additional time-regularity for $ \tu_k, \tpi_k $ and $ \phi $, we could conclude that there exists constant $ \delta > 0 $ such that
	\begin{equation*}
		\norm{\cR F}_{\bbfo(J)} \leq C T^\delta \norm{F}_{\bbfo(J)},
	\end{equation*}
	where $ C > 0 $ does not depend $ T > 0 $. Taking $ T > 0 $ small enough, for example, $ C T^\delta < 1/2 $, the operator $ (\bbi - \cR) $ is invertible. Substitute $ F $ above by $ (\bbi - \cR)^{-1} F $, one obtains
	\begin{equation*}
	\cL \Hat{\cS} (\bbi - \cR)^{-1} F = F,
	\end{equation*}
	which defines the right inverse for $ \cL $ as $ \cS := \Hat{\cS}(\bbi - \cR)^{-1} $.
	
	This completes the proof.
\end{proof}

Given $ c_s \in \Lq{q}(J; \W{2}(\Os)) \cap \W{1}(J; \Lq{q}(\Os)) $, we obtain the well-posedness of 
\begin{equation} \label{Linearized TwoPhase-c_s}
	\begin{alignedat}{3}
		\rho \pt u - \Div \Smu(u, \pi) & = f_u, && \tin \OM \backslash \Sigma \times J, \\
		\Div u - \frac{\gamma \beta}{\rho_s} c_s & = f_d, && \tin \OM \backslash \Sigma \times J, \\
		\jump{u} & = g_1, && \ton \Sigma \times J, \\
		\jump{(- \pi \bbi + \mu (\nabla u + \tran{\nabla} u)) \nuSigma} & = g_2, && \ton \Sigma \times J, \\
		\tran{(u_1, u_3)} & = g_3, && \ton G \backslash \partial \Sigma \times J, \\
		- \pi + 2 \mu \ptial{2} u_2 & = g_4, && \ton G \backslash \partial \Sigma \times J, \\
		u & = g_5, && \ton S \times J, \\
		u(0) & = u_0, && \tin \OM \backslash \Sigma,
	\end{alignedat} 
\end{equation}
with the aid of Theorem \ref{Theorem: Cylinder-TwoPhase}.
\begin{corollary} \label{Corollary: linear part}
	Let $ \gamma, \beta > 0 $. Given $ c_s \in \Lq{q}(J; \W{2}(\Os)) \cap \W{1}(J; \Lq{q}(\Os)) $. Then under the assumptions of Theorem \ref{Theorem: Cylinder-TwoPhase}, \eqref{Linearized TwoPhase-c_s} admits a unique solution
	\begin{gather*}
		(u, \pi) \in \bbe(J),
	\end{gather*}
	if and only if the data are subject to the regularity and compatibility conditions in Theorem \ref{Theorem: Cylinder-TwoPhase}.
\end{corollary}
\begin{proof}
	Since $ c_s \in \bbe_{c,s}(J) = \Lq{q}(J; \W{2}(\Os)) \cap \W{1}(J; \Lq{q}(\Os)) $, it follows from the embeddings $ \W{2}(\Os) \hookrightarrow \W{1}(\Os), 
	\Lq{q}(\Os) \hookrightarrow \W{-1}(\Os) $
	that $ c_s \in \bbf_2(J) $.
	For $ c_{0,s} \in \W{2 - 2/q}(\Os) $, trace method of real interpolation implies that there exists a function $ \tilde{c}_s \in \bbe_{c,s}(J) $ such that $ \rv{\tilde{c}_s}_{t = 0} = c_{0,s} $. Then one may decompose $ c_s $ as $ c_s := \tilde{c}_s + \bar{c}_s $, where $ \bar{c}_s \in \bbeo_{c,s}(J) $. Extend $ \bar{c}_s $ suitably from $ \bbeo_{c,s}(J) $ to $ \bar{c} \in \Lq{q}(J; \W{2}(\OM \backslash \Sigma)) \cap \WO{1}(J; \Lq{q}(\OM)) $. By compatibility conditions \eqref{Compatibility: TwoPhase-initial-cylinder} and \eqref{Compatibility: TwoPhase-contact-cylinder}, we define a function $ \hc $ such that 
	\begin{equation*}
		 \hc = \left\{
		 	\begin{aligned}
			 	& \rv{\bar{c}}_{\Os} = \bar{c}_s, && \text{ in } \Os, \\
			 	& \frac{\mu_s}{\mu_f} \rv{\bar{c}}_{\Of}, && \text{ in } \Of.
		 	\end{aligned}
		 \right.
	\end{equation*}
	Then $ \hc \in \bbfo_2(J) $ and $ \jump{\mu \hc} = 0 $ on $ \Sigma $. Consequently, the problem
	\begin{alignat*}{3}
		\rho \pt \baru - \Div \Smu(\baru, \barpi) & = 0, && \tin \OM \backslash \Sigma \times J, \\
		\Div \baru & = \frac{\gamma \beta}{\rho_s} \hc, && \tin \OM \backslash \Sigma \times J, \\
		\jump{\baru} & = 0, && \ton \Sigma \times J, \\
		\jump{- \barpi \bbi + \mu (\nabla \baru + \tran{\nabla} \baru)} \nuSigma & = 0, && \ton \Sigma \times J, \\
		\tran{(\baru_1, \baru_3)} & = 0, && \ton G \backslash \partial \Sigma \times J, \\
		- \barpi + 2 \mu \ptial{2} \baru_2 & = 0, && \ton G \backslash \partial \Sigma \times J, \\
		\baru & = 0, && \ton S \times J, \\
		\baru(0) & = 0, && \tin \OM \backslash \Sigma.
	\end{alignat*}
	admits a unique solution $ (\baru, \barpi) \in \bbeo(J) $ thanks to Theorem \ref{Theorem: Cylinder-TwoPhase}. Let $ (\tu, \tpi) $ be the unique solution of \eqref{Linearized TwoPhase} with $ f_d $ substituted by $ f_d + \frac{\gamma \beta}{\rho_s} \tilde{c}_s $, then $ (\baru + \tu, \barpi + \tpi) $ solves \eqref{Linearized TwoPhase-c_s}.
\end{proof}

By standard localization procedure, one can prove the local well-posedness of the heat equation in a cylindrical domain and in a cylindrical ring respectively. We omit the proof of Theorem \ref{Theorem: Cylinder-Heat_f} and \ref{Theorem: Cylinder-Heat_s} here and refer to Wilke \cite[Section 5.3.2]{Wilke2020} for the similar argument with model problems constructed in Section \ref{Section: Heat-bent-quarter} and Pr\"uss--Simonett \cite[Section 6.2]{PS2016}.

\section{Nonlinear well-posedness}
\label{Section: Nonlinear well-posedness}
In this section, we aim to prove the local well-posedness of \eqref{Nonlinear: Lagrangian}, namely, to prove Theorem \ref{Theorem: main}.
To this end, we firstly introduce the function spaces for the data
\begin{gather*}
	\bbf_1(J) := \Lq{q}(J; \Lq{q}(\OM)^3), \quad
	\bbf_2(J) := \W{1}(J; \dW{-1}(\OM)) \cap \Lq{q}(J; \W{1}(\OM \backslash \Sigma)), \\
	\bbf_3(J) := \W{\onehalf - \frac{1}{2q}}(J; \Lq{q}(\Sigma)^3) \cap \Lq{q}(J; \W{1 - \frac{1}{q}}(\Sigma)^3), \\
	\bbf_4(J) := \W{1 - \frac{1}{2q}}(J; \Lq{q}(G)^3) \cap \Lq{q}(J; \W{2 - \frac{1}{q}}(G \backslash \partial \Sigma)^3), \\
	\bbf_5(J) := \W{\onehalf - \frac{1}{2q}}(J; \Lq{q}(G)) \cap \Lq{q}(J; \W{1 - \frac{1}{q}}(G \backslash \partial \Sigma)), \\
	\bbf_6(J) := \Lq{q}(J; \Lq{q}(\Of)), \quad
	\bbf_7(J) := \W{\onehalf - \frac{1}{2q}}(J; \Lq{q}(\Sigma)) \cap \Lq{q}(J; \W{1 - \frac{1}{q}}(\Sigma)), \\
	\bbf_8(J) := \W{\onehalf - \frac{1}{2q}}(J; \Lq{q}(G)) \cap \Lq{q}(J; \W{1 - \frac{1}{q}}(G \backslash \partial \Sigma)),\\
	\bbf_9(J) := \W{\onehalf - \frac{1}{2q}}(J; \Lq{q}(S)) \cap \Lq{q}(J; \W{1 - \frac{1}{q}}(S)), \\
	\bbf_{10}(J) := \Lq{q}(J; \W{1}(\Os)), \quad
	\bbf_{11}(J) := \Lq{q}(J; \W{1}(\Os)), \quad
	\bbf(J) := \Pi_{j = 1}^{11} \bbf_j(J).
\end{gather*}
Let $ \bw = (\hv, \hpi, \hc, \hcss, \hg) $ and $ \bw_0 = (\vo, \co, 0, 1) $. Recalling the definition of solution and initial spaces in Section \ref{Subsection: equivalent Lagrangian}, we reformulate \eqref{Nonlinear: Lagrangian} in the abstract form
\begin{equation}
	\label{abstract}
	\sL (\bw) = \sN(\bw, \wo) \quad \textrm{for all}\ \bw \in \bbe(J),\ (\vo, \co) \in X_\gamma,
\end{equation}
where
\begin{gather*}
	\sL (\bw) := 
	\left( 
		\begin{gathered}
			\hr \pt \hv - \hdiv \bS( \hv, \hpi ) \\
			\hdiv \left( \hv \right) - \frac{\gamma \beta}{\hrs} \hcs \\
			\jump{\bS( \hv, \hpi )} \hn_\Sigma \\
			\rv{\cP_G(\hv)}_G \\
			\rv{\bS( \hv, \hpi ) \hn_G \cdot \hn_G}_G \\
			\pt \hc - \hD \hDelta \hc \\
			\hD \hnab \hc \cdot \hn_\Sigma \\
			\hD \hnab \hc \cdot \hn_G \\
			\hDs \hnab \hcs \cdot \hn_S \\
			\pt \hcss - \beta \hcs \\
			\pt \hg - \frac{\gamma \beta}{n \hrs} \hcs \\
			\rv{(\hv, \hc, \hcss, \hg)}_{t = 0}
		\end{gathered}
	\right), \quad 
	\sN (\bw, \wo) := 
	\left( 
		\begin{gathered}
			\bK(\bw) \\
			G(\bw) \\
			\bH^1(\bw) \\
			\bH^2(\bw) \\
			H^3(\bw) \\
			F^1(\bw) \\
			F^2(\bw) \\
			F^3(\bw) \\
			F^4(\bw) \\
			F^5(\bw) \\
			F^6(\bw) \\
			\wo
		\end{gathered} 
	\right).
\end{gather*}
Define
\begin{equation*}
	\sM(\bw) := \left( \bK, G, \bH^1, \bH^2, H^3, F^1, F^2, F^3, F^4, F^5, F^6 \right)^\top (\bw).
\end{equation*}
Then it follows from Corollary \ref{Corollary: linear part} and Theorem \ref{Theorem: Cylinder-Heat_f}, \ref{Theorem: Cylinder-Heat_s} that $ \sL : \bbe(J) \rightarrow \bbf(J) \times X_\gamma $  is an \textit{isomorphism} for $ J := [0,T] $, $ T > 0 $. Moreover, the following proposition holds for $ \sM(\bw) : \bbe(J) \rightarrow \bbf(J) $. 
\begin{proposition}
	\label{Propostion: contraction}
	Let $ q > 3 $ , $ J = [0, T] $ and $ R > 0 $. Assume $ w = (\hv, \hpi, \hc, \hcss, \hg) \in \bbe(J) $ with $ \rv{\hg}_{t = 0} = 1 $ and $ \norm{\bw}_{\bbe(J)} \leq R $, then there exist a constant $ C = C(R) > 0 $, a finite time $ T_R > 0 $ depending on $ R $ and $ \delta > 0 $ such that for $ 0 < T < T_R $, $ \sM(\bw) : \bbe(J) \rightarrow \bbf(J) $ is well-defined and bounded along with the estimates:
	\begin{equation}
		\label{Mw}
		\norm{\sM(\bw)}_{\bbf(J)}
		\leq C(R) \TD \left( \norm{\bw}_{\bbe(J)} + 1 \right).
	\end{equation}
	Moreover, for $ \bw^1 = (\hv^1, \hpi^1, \hc^1, {{}\hcss}^1, \hg^1), \bw^2 = (\hv^2, \hpi^2, \hc^2, {{}\hcss}^2, \hg^2) \in \YT $, $ \rv{\hc^i}_{t = 0} = \co $, $ \rv{\hcss}_{t = 0} = 0 $, $ \rv{\hg^i}_{t = 0} = 1 $ and $ \norm{\bw^i}_{\bbe(J)} \leq R \ (i = 1, 2) $, there exist a constant $ C = C(R) > 0 $, a finite time $ T_R > 0 $ depending on $ R $ and $ \delta > 0 $ such that for $ 0 < T < T_R $,
	\begin{equation}
		\label{MMw}
		\norm{\sM(\bw^1) - \sM(\bw^2)}_{\bbf(J)} \leq C(R) \TD \norm{\bw^1 - \bw^2}_{\bbe(J)}.
	\end{equation}
\end{proposition}
\begin{proof}
	This proposition is same as in \cite[Proposition 4.2]{AL2021} and the proof is very similar. The different points we need to pay attention are the estimates of $ \bH^2 $ and $ H^3 $, where 
	\begin{align*}
		& \bH^2 = - \left( \bbi - \big((\invtr{\hF} \hn_G) \otimes (\invtr{\hF} \hn_G) - \hn_G \otimes \hn_G \big) \right) \hv,  \\
		& H^3 = - \hsigma \invtr{\hF} \hn_G \cdot (\invtr{\hF} \hn_G) + \bS(\hv, \hpi) \hn_G \cdot \hn_G.
	\end{align*}
	Note that for tensors $ \bS, \bT \in \bbr^{3 \times 3} $ and vectors $ \bu, \bv \in \bbr^3 $, we have the tensor algebra property
	\begin{gather*}
		\bS(\bu \otimes \bv) = (\bS \bu) \otimes \bv, \quad 
		(\bu \otimes \bv)\bS= \bu \otimes (\tran{\bS} \bv), \\
		\bT \bu \cdot (\tran{\bS} \bv) = (\bS \bT) \bu \cdot \bv = \bS (\bT \bu) \cdot \bv.
	\end{gather*}
	Then one can derive
	\begin{align*}
		& \bH^2(\bw^1) - \bH^2(\bw^2) \\
		& \qquad = \Big(
			\big(\invtr{\hF}(\hnab \hv^1) - \invtr{\hF}(\hnab \hv^2)\big) (\hn_G \otimes \hn_G) \inv{\hF}(\hnab \hv^1) \\
			& \qquad \qquad \qquad + \invtr{\hF}(\hnab \hv^2) (\hn_G \otimes \hn_G) \big(\inv{\hF}(\hnab \hv^1) - \inv{\hF}(\hnab \hv^2)\big) 
			\Big) \hv^1 \\
		& \qquad \qquad + \Big(
			\big(\invtr{\hF}(\hnab \hv^2) - \bbi\big) (\hn_G \otimes \hn_G) \inv{\hF}(\hnab \hv^1) \\
			& \qquad \qquad \qquad + (\hn_G \otimes \hn_G) \big(\inv{\hF}(\hnab \hv^2) - \bbi\big) 
		\Big) (\hv^1 - \hv^2) - (\hv^1 - \hv^2) \bbi.
	\end{align*}
	and
	\begin{align*}
		& H^3(\bw^1) - H^3(\bw^2) \\
		& \qquad = - (\inv{\hF}(\hnab \hv^1) - \bbi) \big(\hsigma(\hw^1) \invtr{\hF}(\hnab \hv^1) - \hsigma(\hw^2) \invtr{\hF}(\hnab \hv^2) \big) \hn_G \cdot \hn_G \\
		& \qquad \quad - \big(\inv{\hF}(\hnab \hv^1) - \inv{\hF}(\hnab \hv^2)\big) (\hsigma(\hw^2) \invtr{\hF}(\hnab \hv^2)) \hn_G \cdot \hn_G 
		- \bK \hn_G \cdot \hn_G.
	\end{align*}
	With the general trace theorem and multiplication property, see e.g. \cite[Lemma 2.1]{AL2021}, one can derive the estimates for $ \bH^2 $ and $ H^3 $ following the procedure in \cite[Proposition 4.2]{AL2021}.
\end{proof}

Now we are in the position to show Theorem \ref{Theorem: main}.
\begin{proof}[\textbf{Proof of Theorem \ref{Theorem: main}}]
	For $ (\vo, \co) \in X_\gamma $, by the trace method of real interpolation, see e.g. Lunardi \cite[Proposition 1.13]{Lunardi2018}, there exists a function $ \tilde{\bw} = (\tv, \tilde{c}) \in \bbe_{\hv}(\infty) \times \bbe_{\hc}(\infty) $ such that $ \rv{\tilde{\bw}}_{t = 0} = (\vo, \co) $. Then one can reduce \eqref{Nonlinear: Lagrangian} to the case of trivial initial data by eliminating $ \tilde{\bw} $. Now we set a well-defined constant for 
	\begin{equation*}
		C_{\sL} := \sup_{0 \leq T \leq 1} \norm{\sL^{-1}}_{\cL(\bbfo(J), \bbeo(J))},
	\end{equation*}
	which can be verified to be bounded as $ t \rightarrow 0 $ as in \cite{AL2021}.
	Choose $ R > 0 $ large such that $ R \geq 2 C_\sL \norm{(\vo, \co)}_{X_\gamma} $. Then 
	\begin{equation}
		\label{L0}
		\norm{\sL^{-1} \sN(0, \wo)}_{\bbe(J)} 
		\leq C_\sL \norm{(\vo, \co)}_{X_\gamma}
		\leq \frac{R}{2}.
	\end{equation}
	For $ \norm{\bw^i}_{\bbe(J)} \leq R $, $ i = 1, 2 $, we take $ T_R > 0 $ small enough such that $ C_\sL C(R) \TD_R \leq 1/2 $,	where $ C(R) $ is the constant in \eqref{MMw}. Then for $ 0 < T < T_R $, we infer from Proposition \ref{Propostion: contraction} that
	\begin{equation}
		\begin{aligned}
			\label{L12}
			& \norm{\sL^{-1}\sN(\bw^1, \wo) - \sL^{-1}\sN(\bw^2, \wo)}_{\bbe(J)}  \\
			& \qquad \qquad \qquad \qquad \leq C_\sL C(R) \TD \norm{\bw^1 - \bw^2}_{\bbe(J)} 
			\leq \onehalf \norm{\bw^1 - \bw^2}_{\bbe(J)},
		\end{aligned}
	\end{equation}
	which implies the contraction property.
	From \eqref{L0} and \eqref{L12}, we have
	\begin{align*}
		& \norm{\sL^{-1}\sN(\bw, \wo)}_{\bbe(J)} \\
		& \qquad \leq \norm{\sL^{-1}\sN(0, \wo)}_{\bbe(J)} + \norm{\sL^{-1}\sN(\bw, \wo) - \sL^{-1}\sN(0, \wo)}_{\bbe(J)}
		\leq R.
	\end{align*}
	Define $ \cM_{R,T} $ by
	\begin{equation*}
		\cM_{R,T} := \left\{ \bw \in \overline{B_{\bbeo(J)}(0,R)}: w = (\hv, \hpi, \hc, \hcs, \hg), \right\},
	\end{equation*}
	a closed subset of $ \bbeo(J) $. Hence, $ \sL^{-1}\sN : \cM_{R,T} \rightarrow \cM_{R,T} $ is well-defined for all $ 0 < T < T_R $ and a strict contraction. Since $ \bbeo(J) $ is a Banach space, the Banach fixed-point Theorem implies the existence of a unique fixed-point of $ \sL^{-1}\sN $ in $ \cM_{R,T} $, i.e., \eqref{Nonlinear: Lagrangian} admits a unique strong solution in $ \cM_{R,T} $ for small time $ 0 < T < T_R $.
	
	The uniqueness in $ \bbeo(J) $ follows easily by mimicking the continuity argument in \cite[Proof of Theorem 2.1]{AL2021}, we omit it here and complete the proof.
\end{proof}

\appendix
\section{Partition of unity with vanishing Neumann trace}
\label{AppSection: Partition}
\begin{proposition} \label{AppendixPropositon: POU-G0}
	Let $ G_0 := G_0^+ \cup G_0^- \cup \Gamma \subset \bbr^2 $ with three disjoint components and $ \Bar{G_0^-} \subset G_0 $, be a bounded domain with the boundary $ \partial G_0 \in C^{m+1} $ and an interface $ \Gamma = \partial G_0^- \in C^{m+1} $, $ m \in \bbn_+ $. $ \Bar{G_0^-} \subset G_0 $. Then for each finite open covering $ \{ U_k \}_{k = 1}^{N} $ of $ \Gamma \cup \partial G_0 $ with $ \{ U_k \}_{k = 1}^{N_1} \supset \Gamma $ and $ \{ U_k \}_{k = N_1 + 1 }^{N} \supset \partial G_0 $, there exist open sets $ U_0 \subset G_0^- $ and $ U_{N + 1} \subset G_0^+ $ such that $ U_0 \cap \Gamma = \emptyset $, $ U_{N + 1} \cap (\partial G_0 \cup \Gamma) = \emptyset $, $ \bigcup_{k = 0}^{N + 1} U_k \supset \Bar{G_0} $. Moreover, there is a subordinated partition of unity $ \{ \psi_k \}_{k = 0}^{N + 1} \subset C_c^m(\bbr^2) $, such that $ \supp \psi_k \subset U_k $ and $ \ptial{\nu_{G_0}} \psi_k = 0 $ on $ \partial G_0 $, $ \ptial{\nu_{\Gamma}} \psi_k = 0 $, on $ \Gamma $, where $ \nu_{\Gamma} $ denotes the unit normal vector pointing from $ G_0^- $ to $ G_0^+ $, $ \nu_{G_0} $ is the outer unit normal vector on $ \partial G_0 $.
\end{proposition}
\begin{proof}
	The proof is based on \cite[Proposition 5.3]{Wilke2020}. For any finite open covering $ \{ U_k \}_{k = 1}^{N} $ of $ \Gamma \cup \partial G_0 $, where $ \bigcup_{k = 1}^{N_1} U_k \supset \Gamma $ and $ \bigcup_{k = N_1 + 1}^{N} U_k \supset \partial G_0 $, there exists $ U_{N_1 + 1} \subset G_0^+ $, such that $ U_{N_1 + 1} \cap (\Gamma \cup \partial G_0) = \emptyset $ and $ \bigcup_{k = 1}^{N + 1} U_k \supset \Bar{G_0^+} $ by \cite[Proposition 5.3]{Wilke2020}, while at the same time, there exists $ U_0 \subset G_0^- $, such that $ U_0 \cap \Gamma = \emptyset $ and $ \bigcup_{k = 0}^{N_1} U_k \supset \Bar{G_0^-} $. 
	
	Moreover, there is two subordinated partitions of unity $ \{ \phi_k \}_{k = 0}^{N_1} $, $ \{ \varphi_k \}_{k = 1}^{N + 1} \subset C_c^m(\bbr^2) $ such that $ \supp \phi_k \subset U_k $ and $ \supp \varphi_k \subset U_k $, $ \ptial{\nu_{\Gamma}} \phi_k = 0 $, $ - \ptial{\nu_{\Gamma}} \varphi_k = 0 $, on $ \Gamma $ and $ \ptial{\nu_{G_0}} \varphi_k = 0 $, on $ \partial G_0 $. Now we argue by the truncation. Let $ V := \bigcup_{k = 0}^{N + 1} U_k  $ and $ V^- := \bigcup_{k = 0}^{N_1} U_k $. For $ x \in V $, define $ d(x) $ as the distance from $ x $ to $ V^- $. Let $ \eta(x) \in C_0^\infty (\bbr; [0,1]) $ be a cutoff function over $ V $ such that $ \eta(s) = 1 $ if $ \abs{s} < \epsilon $ and $ \eta(s) = 0 $ if $ \abs{s} > 2 \epsilon $ for $ \epsilon > 0 $ small. Define $ \psi_k := \eta(d(x)) \phi_k + (1 - \eta(d(x))) \varphi_k $. Then
	\begin{equation*}
		\sum_{k = 0}^{N + 1} \psi_k = \sum_{k = 0}^{N_1} \eta \phi_k + \sum_{k = 1}^{N + 1} (1 - \eta)\varphi_k = 1.
	\end{equation*}
	Then $ \{ \psi_k \}_{k = 0}^{N + 1} \in C_c^m(\bbr^2) $ is a partition of unity such that $ \supp \psi_k \subset U_k $. Since $ \eta(s) $ is constant in both a neighborhood of and far away from $ s = 0 $, one obtains $ \ptial{\nu_{\Gamma}} \psi_k = 0 $, on $ \Gamma $ and $ \ptial{\nu_{G_0}} \psi_k = 0 $, on $ \partial G_0 $.	
\end{proof}
Analogous to \cite[Section 5.2]{Wilke2020}, the results can be extended to the general cylindrical domain $ \OM := \OM^+ \cup \OM^- \cup \Sigma $ with three disjoint components satisfying $ \Bar{\OM^-} \subset \OM $, $ \partial \OM = G \cup \Bar{S} $, $ \partial \OM^- = \Sigma \cup G_0^- $, $ \partial \OM^+ = \Sigma \cup S \cup G_0^+ $,
where $ L_1 < L_2 < \infty $ are two constant, $ G := \cup_{j = 1,2} G_0 \times \{L_j\} $ and $ S, \Sigma \subset \bbr^3 $ are general hypersurfaces which will be assigned with certain regularity. In particular, $ \partial \Sigma = \cup_{j = 1,2} \Gamma \times \{L_j\} $, $ \partial S = \cup_{j = 1,2} \partial G_0 \times \{L_j\} $ are sub-manifolds of dimension one in $ \bbr^3 $.
Moreover, we assume that $ S \perp G $ at $ \partial G $, $ \Sigma \perp G $ at $ \partial \Sigma $ and $ \Sigma \cap S = \emptyset $.

\begin{proposition} \label{AppendixPropositon: POU-Omega}
	Let $ \OM $ be a bounded domain defined above with $ G, S, \Sigma \in C^{m} $ and $ \partial \Sigma, \partial S \in C^{m + 1} $, $ m \in \bbn_+ $. 
	Then for each finite open covering  $ \{ U_k \}_{k = 1}^{N} $ of $ \partial S \cup \partial \Sigma $ in $ \bbr^3 $, there exist open sets $ U_j \subset \bbr^3 $, $ j \in \{ N + 1, ..., N + 5 \} $, such that
	\begin{itemize}
		\item $ U_{N+1} \cap (\Bar{\OM} \backslash \Bar{G}) \neq \emptyset $, $ U_{N+1} \cap \Bar{G} = \emptyset $,
		\item $ U_{N+1+j} \cap U_{N+1} \cap G_j^+ \neq \emptyset $, $ U_{N+1+j} \cap (\Sigma \cup S) = \emptyset $, $ j = 1,2 $,
		\item $ U_{N+3+j} \cap U_{N+1} \cap G_j^- \neq \emptyset $, $ U_{N+3+j} \cap \Sigma = \emptyset $, $ j = 1,2 $,
		\item $ \bigcup_{j = 1}^{N+5} U_j \supset \Bar{\OM} $.
	\end{itemize}
	Moreover, there is a subordinated partition of unity $ \{ \psi_k \}_{k = 1}^{N+5} \subset C_c^m(\bbr^3) $, such that $ \supp \psi_k \subset U_k $ and $ \ptial{\nu_S} \psi_k = 0 $ on $ \partial S $, $ \ptial{\nuSigma} \psi_k = 0 $, on $ \partial \Sigma $.
\end{proposition}

\section{Auxiliary transmission problems}
\label{AppSection: Auxiliary transmission problems}
In this section, we give the existence and uniqueness of auxiliary elliptic and parabolic transmission problems, respectively. 
Given $ G_0 := G_0^+ \cup G_0^- \cup \Gamma \subset \bbr^2 $ with $ \Bar{G_0^-} \subset G_0 $, we define the cylindrical domain for $ 0 < L_1 < L_2 < \infty $ as
$ \OM := \OM^+ \cup \OM^- \cup \Sigma $, $ \Bar{\OM^-} \subset \OM $ with three disjoint parts, satisfying $ \partial \OM = G \cup \Bar{S} $, $ \partial \OM^- = \Bar{\Sigma} \cup G_0^- $, $ \partial \OM^+ = \Bar{\Sigma} \cup \Bar{S} \cup G_0^+ $,
where $ G := \cup_{j = 1,2} G_0 \times \{L_j\} $ and $ S, \Sigma \subset \bbr^3 $ are general hypersurfaces which will be assigned with certain regularity. Particularly, $ \partial \Sigma = \cup_{j = 1,2} \Gamma \times \{L_j\} $, $ \partial S = \cup_{j = 1,2} \partial G_0 \times \{L_j\} $ are two one dimensional sub-manifolds of $ \bbr^3 $.
Moreover, we assume that $ S \perp G $ at $ \partial G $, $ \Sigma \perp G $ at $ \partial \Sigma $ and $ \Sigma \cap S = \emptyset $, which means $ \partial S $, $ \partial \Sigma $ are the contact lines involving with ninety-degree contact angles.
In addition, $ \nu $ denotes the unit outer normal vector on the interface $ \Sigma $ (point from $ \OM^- $ to $ \OM^+ $) and the boundary $ \partial \OM $.


\subsection{Elliptic transmission problems}
Firstly, for $ \lambda > 0 $, we consider the elliptic system
\begin{equation} \label{Appendix: Elliptic transmission}
	\begin{alignedat}{3}
		\lambda \phi - \Delta \phi & = f, && \tin \OM \backslash \Sigma, \\
		\jump{\rho \phi} & = g_1, && \ton \Sigma, \\
		\jump{\ptial{\nuSigma} \phi} & = g_2, && \ton \Sigma, \\
		\rho \phi & = g_3, && \ton G \backslash \partial \Sigma, \\
		\ptial{\nu_S}\phi & = g_4, && \ton S.
	\end{alignedat}
\end{equation}
Then we have the following theorem.
\begin{theorem} \label{AppendixTheorem: Elliptic}
	Let $ q \geq 2 $, $ \rho > 0 $, $ T > 0 $, $ J = (0,T) $ and $ s \in \{0, 1\} $. Assume that $ \OM $, $ \Sigma $, $ G $, $ S $ are defined as above and $ \Sigma $, $ G $, $ S $ are of class $ C^3 $, as well as $ \partial G \in C^4 $. Then there is a constant $ \lambda_0 \geq 0 $ such that for all $ \lambda > \lambda_0 $, \eqref{Appendix: Elliptic transmission} admits a unique solution $ \phi \in \W{2 + s }(\OM \backslash \Sigma) $ if and only if the data satisfy the following regularity and compatibility conditions:
	\begin{enumerate}
		\item $ f \in \W{s}(\OM \backslash \Sigma) $,
		\item $ g_1 \in \W{2 + s - \frac{1}{q}}(\Sigma) $, 
		\item $ g_2 \in \W{1 + s - \frac{1}{q}}(\Sigma) $, 
		\item $ g_3 \in \W{2 + s - \frac{1}{q}}(G \backslash \partial \Sigma) $, 
		\item $ g_4 \in \W{1 + s - \frac{1}{q}}(S) $, 
		\item $ \jump{g_3} = g_1 $, $ \jump{\ptial{\nuSigma}(g_3 / \rho)} = g_2 $, at $ \partial \Sigma $,
		\item $ \ptial{\nu_S}(g_3 / \rho) = g_4 $, on $ \partial S $.
	\end{enumerate}
\end{theorem}
\begin{proof}
	We first consider the case $ s = 0 $. The general tool is the localization procedure, see e.g. \cite{DHP2003}, \cite{PS2016}, \cite{Wilke2020}. To this end, one gives four model problems with respect to \eqref{Appendix: Elliptic transmission}. Namely,
	\begin{itemize}
		\item The half-space elliptic equation with Dirichlet boundary condition
		\item The quarter-space elliptic equation with Neumann and Dirichlet boundary conditions
		\item The half-space elliptic transmission problem with a flat interface and Dirichlet boundary condition
		\item The elliptic transmission problem in smooth domain with a flat interface and Neumann boundary condition.
	\end{itemize}
	The first and the last one have been solved well, readers are referred to e.g. Pr\"uss--Simonett \cite{PS2016}. 
	
	\textbf{Elliptic equation in a quarter space.} For a quarter space $ \{(x_1, x_2, x_3): x_1 \in \bbr, x_2 \in \bbr_+, x_3 \in \bbr_+\} $, we consider the problem for $ \lambda > 0 $ large,
	\begin{equation}\label{AppendixEqs: Laplace-quarter}
		\begin{alignedat}{3}
			\lambda \phi - \Delta \phi & = f, && \quad x_1 \in \bbr,\ x_2 > 0,\ x_3 > 0, \\
			\rho \phi & = g_1, && \quad x_1 \in \bbr,\ x_2 = 0,\ x_3 > 0, \\
			\ptial{3} \phi & = g_2, && \quad x_1 \in \bbr,\ x_2 > 0,\ x_3 = 0.
		\end{alignedat}
	\end{equation}
	Naturally, the compatibility condition at the contact line
	\begin{equation*}
		\ptial{3} \left( \frac{g_1}{\rho }\right) = g_2, \quad x_1 \in \bbr, \ x_2 = 0,\ x_3 = 0,
	\end{equation*}
	should be satisfied. 
	To solve \eqref{AppendixEqs: Laplace-quarter}, we reduce it to the case $ g_1 = 0 $. Extend $ g_1 \in \W{2 - \frac{1}{q}}(\bbr \times \{0\} \times \bbr_+) $ with respect to $ x_3 $ by the extension
	\begin{equation*}
		\tg_1 (x_1, x_2, x_3) = \left\{
			\begin{aligned}
				& g_1 (x_1, x_2, x_3), && \text{ if } x_3 > 0, \\
				& - g_1 (x_1, x_2, -2 x_3) + 2 g_1 (x_1, x_2, - x_3 / 2), && \text{ if } x_3 < 0.
			\end{aligned}
		\right.
	\end{equation*}
	Then $ \tg_1 \in \W{2 - \frac{1}{q}}(\bbr \times \{0\} \times \bbr) $. 
	Let $ \bar{\rho} \equiv \rho $ in $ \bbr \times \bbr_+ \times \bbr $. Solving
	\begin{equation*}
		\begin{alignedat}{3}
		\lambda \phi - \Delta \phi & = 0, && \quad x_1 \in \bbr,\ x_2 > 0,\ x_3 \in \bbr, \\
		\bar{\rho} \phi & = \tg_1, && \quad x_1 \in \bbr,\ x_2 = 0,\ x_3 \in \bbr, 
		\end{alignedat}
	\end{equation*}
	yields a unique solution $ \bar{\phi} \in \W{2}(\bbr \times \bbr_+ \times \bbr) $ by \cite[Section 6.2]{PS2016}. Then $ (\phi - \bar{\phi}) $ (restricted) solves \eqref{AppendixEqs: Laplace-quarter} with $ g_1 = 0 $ and modified data (not to be relabeled) $ g_2 $ satisfying the compatibility condition $ \rv{g_2}_{x_2 = 0} = 0 $. Extend $ f \in \Lq{q}(\bbr \times \bbr_+^2) $ and $ g_2 \in \W{1 - \frac{1}{q}}(\bbr \times \bbr_+ \times \{0\}) $ to some function $ \tf \in \Lq{q}(\bbr^2 \times \bbr_+) $ and $ \tg_2 \in \W{1 - \frac{1}{q}}(\bbr^2 \times \{0\}) $ by odd reflection. Then we solve the half-space elliptic equation
	\begin{equation*}
		\begin{alignedat}{3}
			\lambda \phi - \Delta \phi & = \tf, && \quad  x_1 \in \bbr,\ x_2 \in \bbr,\ x_3 > 0, \\
			\ptial{3} \phi & = \tg_2, && \quad  x_1 \in \bbr,\ x_2 \in \bbr,\ x_3 = 0, 
		\end{alignedat}
	\end{equation*}
	to obtain a unique solution $ \tilde{\phi} \in \W{2}(\bbr^2 \times \bbr_+) $ by \cite[Section 6.2]{PS2016}. By the symmetry, the function $ - \tilde{\phi}(x_1, - x_2, x_3) $ is also a solution to the above system and the uniqueness implies $ \tilde{\phi}\vert_{x_2 = 0} = 0 $.	
	Then the restricted function $ (\bar{\phi} + \tilde{\phi}) \in \W{2}(\bbr \times \bbr_+^2) $ solves \eqref{AppendixEqs: Laplace-quarter}. The uniqueness follows easily by carrying out the energy estimate. 
	
	Analogously, there is a solution operator with regard to the data and solutions for the elliptic equation in a bent quarter space as in Section \ref{Section: Stokes-Bent-Quarter}.
	
	\textbf{Elliptic transmission problem in a half space.} For a half space
	\begin{equation*}
		\{(x_1, x_2, x_3): x_1 
		\in \bbr, x_2 \in \bbr_+, x_3 \in \bbr\},
	\end{equation*} 
	we consider the problem for $ \lambda > 0 $ large,
	\begin{equation}\label{AppendixEqs: Laplace-transmission-half}
		\begin{alignedat}{3}
			\lambda \phi - \Delta \phi & = f, && \quad x_1 \in \bbr,\ x_2 > 0,\ x_3 \in \dR, \\
			\jump{\rho \phi} & = g_1, && \quad x_1 \in \bbr,\ x_2 > 0,\ x_3 = 0, \\
			\jump{\ptial{3} \phi} & = g_2, && \quad x_1 \in \bbr,\ x_2 > 0,\ x_3 = 0, \\
			\rho \phi & = g_3, && \quad x_1 \in \bbr,\ x_2 = 0,\ x_3 \in \dR.
		\end{alignedat}
	\end{equation}
	The compatibility conditions at the contact line are
	\begin{equation*}
		\jump{g_3} = g_1, \quad
		\jump{\ptial{3} \left( \frac{g_3}{\rho }\right)} = g_2, \quad x_1 \in \bbr, \ x_2 = 0,\ x_3 = 0.
	\end{equation*}
	Same as above, we reduce \eqref{AppendixEqs: Laplace-transmission-half} to the case $ (f, g_3) = 0 $ firstly. Define $ g_3^\pm := \rv{g_3}_{x_3 \gtrless 0} $, $ f^\pm := \rv{f}_{x_3 \gtrless 0} $ and $ \rho^\pm := \rv{\rho}_{x_3 \gtrless 0} $. Extend $ g_3^\pm \in \W{2 - \frac{1}{q}}(\bbr \times \{0\} \times \bbr_\pm) $ to some functions $ \tg_3^\pm \in \W{2 - \frac{1}{q}}(\bbr \times \{0\} \times \bbr) $ and respectively, $ f^\pm \in \Lq{q}(\bbr \times \bbr_+ \times \bbr_\pm) $ to $ \tf^\pm \in \Lq{q}(\bbr \times \bbr_+ \times \bbr) $ by the extension above, $ \rho^\pm $ to $ \tilde{\rho}^\pm $ by constant.
	Solving the half-space elliptic equation
	\begin{equation*}
		\begin{alignedat}{3}
			\lambda \phi - \Delta \phi & = \tf^+, && \quad x_1 \in \bbr,\ x_2 > 0,\ x_3 \in \bbr, \\
			\tilde{\rho}^+ \phi & = \tg_3^+, && \quad x_1 \in \bbr,\ x_2 = 0,\ x_3 \in \bbr, 
		\end{alignedat}
	\end{equation*}
	yields a unique solution $ \bar{\phi}^+ \in \W{2}(\bbr \times \bbr_+ \times \bbr) $ by \cite[Section 6.2]{PS2016}. Replacing the superscript ``$ + $'' above by ``$ - $'', one obtains a unique solution $ \bar{\phi}^- \in \W{2}(\bbr \times \bbr_+ \times \bbr) $. Define
	\begin{equation*}
		\bar{\phi} := \left\{
			\begin{aligned}
				& \bar{\phi}^+, && \text{ if } x_3 > 0, \\
				& \bar{\phi}^-, && \text{ if } x_3 < 0.
			\end{aligned}
		\right.
	\end{equation*}
	Then $ (\phi - \bar{\phi}) $ (restricted) solves \eqref{AppendixEqs: Laplace-transmission-half} with $ (f, g_3) = 0 $ and modified data (not to be relabeled) $ g_1, g_2 $ satisfying the compatibility condition $ \rv{g_1}_{x_2 = 0} = \rv{g_2}_{x_2 = 0} = 0 $. Hence one can extend $ g_1 \in \W{2 - \frac{1}{q}}(\bbr \times \bbr_+ \times \{0\}) $ to some function $ \tg_1 \in \W{2 - \frac{1}{q}}(\bbr^2 \times \{0\}) $ and $ g_2 \in \W{1 - \frac{1}{q}}(\bbr \times \bbr_+ \times \{0\}) $ to some function $ \tg_2 \in \W{1 - \frac{1}{q}}(\bbr^2 \times \{0\}) $ by odd reflection. Then we solve the full-space elliptic transmission problem with a flat interface
	\begin{equation*}
		\begin{alignedat}{3}
			\lambda \phi - \Delta \phi & = 0, && \quad x_1 \in \bbr,\ x_2 \in \bbr,\ x_3 \in \dR, \\
			\jump{\rho \phi} & = \tg_1, && \quad x_1 \in \bbr,\ x_2 \in \bbr,\ x_3 = 0, \\
			\jump{\ptial{3} \phi} & = \tg_2, && \quad x_1 \in \bbr,\ x_2 \in \bbr,\ x_3 = 0, 
		\end{alignedat}
	\end{equation*}
	to obtain a unique solution $ \tilde{\phi} \in \W{2}(\bbr^2 \times \dR) $ by \cite[Section 6.5]{PS2016}.
	By the symmetry, the function $ - \tilde{\phi}(x_1, - x_2, x_3) $ is also a solution to the above system. Then the uniqueness implies $ \tilde{\phi}\vert_{x_2 = 0} = 0 $.
	Consequently, the restricted function $ (\bar{\phi} + \tilde{\phi}) \in \W{2}(\bbr \times \bbr_+ \times \dR) $ solves \eqref{AppendixEqs: Laplace-transmission-half}.
	As in Section \ref{Section: TwoPhase-Bent-halfspace}, one can conclude that there is also a solution operator with regard to the data and solutions for the elliptic transmission problem in a half space with a bent interface.
	
	Now following the procedure in Section \ref{Subsection: Localization procedure}, one can complete the proof analogously, with the help of Proposition \ref{AppendixPropositon: POU-Omega}. 
	
	For the case $ s = 1 $, the higher-order regularity case, we refer to e.g. Section 6.3.5, Section 6.5.3 in \cite{PS2016} for the models problems in a higher-order regularity class. Then proceeding a similar localization argument gives us the desired higher regularity.
\end{proof}

Based on Theorem \ref{Appendix: Elliptic transmission}, we have the solvability and regularity for the Laplace transmission problem.
\begin{equation} \label{Appendix: Laplace transmission}
	\begin{alignedat}{3}
		- \Delta \phi & = f, && \tin \OM \backslash \Sigma, \\
		\jump{\rho \phi} & = 0, && \ton \Sigma, \\
		\jump{\ptial{\nuSigma} \phi} & = 0, && \ton \Sigma, \\
		\rho \phi & = 0, && \ton G \backslash \partial \Sigma, \\
		\ptial{\nu_S}\phi & = 0, && \ton S.
	\end{alignedat}
\end{equation}
\begin{theorem} \label{AppendixTheorem: Laplace}
	Under the assumption of Theorem \ref{Appendix: Elliptic transmission}, \eqref{Appendix: Laplace transmission} admits a unique solution $ \phi \in \W{2 + s }(\OM \backslash \Sigma) $ if and only if $ f \in \W{s}(\OM \backslash \Sigma) $.
\end{theorem}
\begin{proof}
	For $ s = 0 $, let $ H^k := W_2^k $ and $ H_K^k := W_{2,K}^k $ be the Sobolev spaces, where 
	\begin{equation*}
		W_{q,K}^k(\OM)^3 = \left\{ \bu \in W_q^k(\OM)^3: \rv{\bu}_{K} = 0, \quad K \in \{G, S\} \right\}.
	\end{equation*}
	Testing \eqref{Appendix: Laplace transmission} with a function $ \psi \in H_G^1(\OM) $, one obtains
	\begin{equation*}
		\int_{\OM} \nabla \phi \cdot \nabla \psi \d x
		= \int_{\OM} f \psi \d x.
	\end{equation*}
	Then by Lax--Milgram Theorem, we know that there is a unique weak solution $ \phi \in H_G^1(\OM) $ of \eqref{Appendix: Laplace transmission}. Equivalently, $ \phi $ solves
	\begin{equation} \label{Appendix: Laplace transmission-lambda}
		\begin{alignedat}{3}
			\lambda \phi - \Delta \phi & = f + \lambda \phi =: \tf, && \tin \OM \backslash \Sigma, \\
			\jump{\rho \phi} & = 0, && \ton \Sigma, \\
			\jump{\ptial{\nuSigma} \phi} & = 0, && \ton \Sigma, \\
			\rho \phi & = 0, && \ton G \backslash \partial \Sigma, \\
			\ptial{\nu_S}\phi & = 0, && \ton S.
		\end{alignedat}
	\end{equation}
	weakly. From Sobolev embeddings, we have $ H^1(\OM) \hookrightarrow \Lq{6}(\OM) $ in three dimensional case. Then $ \tf \in \Lq{p}(\OM) $ for $ p = \min(q, 6) $. If $ q \leq 6 $, then \eqref{Appendix: Laplace transmission-lambda} can be solved uniquely in $ \W{2}(\OM \backslash \Sigma) \cap H_G^1(\OM) $ by Theorem \ref{AppendixTheorem: Elliptic}. Otherwise if $ q > 6 $, one obtains $ \phi \in W_{6}^2(\OM \backslash \Sigma) \cap W_{6,G}^1(\OM) \hookrightarrow \Lq{p}(\OM) $ for any $ p \geq 1 $. Consequently, $ \tf \in \Lq{q}(\OM) $, which means $ \phi \in W_{q}^2(\OM \backslash \Sigma) \cap W_{q,G}^1(\OM) $ by Theorem \ref{AppendixTheorem: Elliptic}.
	
	For $ s = 1 $, one can proceed the bootstrap argument from $ s = 0 $, with the help of higher regularity in Theorem \ref{AppendixTheorem: Elliptic}. Namely, we can promote the regularity for $ \tf $ in \eqref{Appendix: Laplace transmission-lambda} up tp $ \W{1}(\OM \backslash \Sigma) $ and get the desired regularity by Theorem \ref{AppendixTheorem: Elliptic} with $ s = 1 $.
\end{proof}

\begin{remark}
	Note that in Wilke \cite[Lemma 5.6]{Wilke2020}, a mean value free condition was imposed for $ f $, while in the present paper, we do not have this one. This is due to the face that if we integrate \eqref{Appendix: Laplace transmission} over $ \OM $, we will obtain
	\begin{equation*}
		\int_\OM f \d x = - \int_G \ptial{\nu_G} \phi \d \sigma,
	\end{equation*}
	which is not vanished since on $ G \backslash \partial \Sigma $, $ \rho \phi = 0 $ instead of $ \ptial{\nu_G} \phi = 0 $ as in \cite{Wilke2020}.
\end{remark}

\subsection{Parabolic transmission problems}
Consider the parabolic transmission problem with mixed boundary conditions
\begin{equation} \label{Appendix: Parabolic transmission}
	\begin{alignedat}{3}
		\rho \pt u - \mu \Delta u & = f, && \tin \OM \backslash \Sigma \times J, \\
		\jump{u} & = g_1, && \ton \Sigma \times J, \\
		\jump{\mu (\nabla u + \tran{\nabla} u) \nuSigma} & = g_2, && \ton \Sigma \times J, \\
		\tran{(u_1, u_3)} & = g_3, && \ton G \backslash \partial \Sigma \times J, \\
		2 \mu \ptial{2} u_2 & = g_4, && \ton G \backslash \partial \Sigma \times J, \\
		u & = g_5, && \ton S \times J, \\
		u(0) & = u_0, && \tin \OM \backslash \Sigma.
	\end{alignedat}
\end{equation}
Then one obtains the following theorem.
\begin{theorem} \label{AppendixTheorem: Parabolic}
	Let $ q > 3 $, $ \rho, \mu > 0 $, $ T > 0 $ and $ J = (0,T) $. Assume that $ \OM $, $ \Sigma $, $ G $, $ S $ are defined as above and $ \Sigma $, $ G $, $ S $ are of class $ C^3 $, as well as $ \partial G $ of $ C^4 $. Then \eqref{Appendix: Parabolic transmission} admits a unique solution 
	\begin{gather*}
		u \in \W{1}(J; \Lq{q}(\OM)^3) \cap \Lq{q}(J; \W{2}(\OM \backslash \Sigma)^3),
	\end{gather*}
	if and only if the data satisfy the following regularity and compatibility conditions:
	\begin{enumerate}
		\item $ f \in \Lq{q}(J; \Lq{q}(\OM)^3) $,
		\item $ g_1 \in \W{1 - \frac{1}{2q}}(J; \Lq{q}(\Sigma)^3) \cap \Lq{q}(J; \W{2 - \frac{1}{q}}(\Sigma)^3) $,
		\item $ g_2 \in \W{\onehalf - \frac{1}{2q}}(J; \Lq{q}(\Sigma)^3) \cap \Lq{q}(J; \W{1 - \frac{1}{q}}(\Sigma)^3) $,
		\item $ g_3 \in \W{1 - \frac{1}{2q}}(J; \Lq{q}(G)^2) \cap \Lq{q}(J; \W{2 - \frac{1}{q}}(G \backslash \partial \Sigma)^2) $,
		\item $ g_4 \in \W{\onehalf - \frac{1}{2q}}(J; \Lq{q}(G)) \cap \Lq{q}(J; \W{1 - \frac{1}{q}}(G \backslash \partial \Sigma)) $,
		\item $ g_5 \in \W{1 - \frac{1}{2q}}(J; \Lq{q}(S)^3) \cap \Lq{q}(J; \W{2 - \frac{1}{q}}(S)^3) $,
		\item $ u_0 \in \W{2 - \frac{2}{q}}(\OM \backslash \Sigma)^3 $, $ \jump{u_0} = \rv{g_1}_{t = 0} $, $ \cP_\Sigma( \jump{\nabla u_0 + \tran{\nabla} u_0} \nuSigma ) = \rv{\cP_\Sigma(g_2)}_{t = 0} $,
		\item $ \tran{((u_0)_1, (u_0)_3)}|_G = \rv{g_3}_{t = 0} $, $ \rv{2 \mu \ptial{2}(u_0)_2}_{G} = \rv{g_4}_{t = 0} $, $ \rv{u_0}_S = \rv{g_5}_{t = 0} $,
		\item $ \jump{g_3} = \tran{((g_1)_1, (g_1)_3)} $, at $ \partial \Sigma $, 
		\item $ (g_2)_1 = \jump{2 \mu \ptial{1} (g_3)_1 \nuSigma \cdot e_1 + \mu (\ptial{1} (g_3)_2 + \ptial{3} (g_3)_1) \nuSigma \cdot e_3} $ at $ \partial \Sigma $,
		\item $ (g_2)_3 = \jump{2 \mu \ptial{3} (g_3)_2 \nuSigma \cdot e_3 + \mu (\ptial{1} (g_3)_2 + \ptial{3} (g_3)_1) \nuSigma \cdot e_1} $, at $ \partial \Sigma $,
		\item $ g_3 = \tran{((g_5)_1, (g_5)_3)} $, $ g_4 = 2 \mu \ptial{2}(g_5)_2 $, at $ \partial S $.
	\end{enumerate}
\end{theorem}
\begin{proof}
	Again, we are going to proof the theorem by localization procedure. To this end, one gives four model problems with respect to \eqref{Appendix: Parabolic transmission}. Namely,
	\begin{itemize}
		\item The half-space heat equation with outflow boundary condition
		\item The quarter-space heat equation with outflow and Dirichlet boundary conditions
		\item The half-space heat transmission problem with a flat interface and Dirichlet boundary condition
		\item The heat transmission problem in smooth domain with a flat interface and Dirichlet boundary condition.
	\end{itemize}
	The first one has been solved, readers are referred to \cite{PS2016}. We focus on the rest ones.
	
	\textbf{Heat equation in a quarter space.} For a quarter space $ \{(x_1, x_2, x_3): x_1 \in \bbr, x_2 \in \bbr_+, x_3 \in \bbr_+\} $, we consider the problem
	\begin{equation}\label{AppendixEqs: Heat-quarter}
		\begin{alignedat}{3}
			\rho \pt u - \mu \Delta u & = f, && \tin \bbr \times \bbr_+ \times \bbr_+ \times J, \\
			\tran{(u_1, u_3)} & = g_1, && \ton \bbr \times \{0\} \times \bbr_+ \times J, \\
			2 \mu \ptial{2} u_2 & = g_2, && \ton \bbr \times \{0\} \times \bbr_+ \times J, \\
			u & = g_3, && \ton \bbr \times \bbr_+ \times\{0\}  \times J, \\
			u(0) & = u_0, && \tin \bbr \times \bbr_+ \times \bbr_+,
		\end{alignedat}
	\end{equation}
	where $ \rho, \mu > 0 $, $ u : \bbr^3 \times [0, T] \rightarrow \bbr^3 $ is the unknown quantity. The data $ g_i $, $ i = 1,2,3 $ satisfy the initial compatibility conditions
	\begin{equation*}
		\rv{\tran{((u_0)_1, (u_0)_3)}}_{x_2 = 0} = \rv{g_1}_{t = 0}, \quad 
		\rv{2 \mu \ptial{2} (u_0)_2}_{x_2 = 0} = \rv{g_2}_{t = 0}, \quad 
		\rv{u_0}_{x_3 = 0} = \rv{g_3}_{t = 0},
	\end{equation*}
	and compatibility condition at the contact line $ \bbr \times \{0\} \times \{0\} $
	\begin{equation*}
		\tran{((g_3)_1, (g_3)_3)} = g_1, \quad 2 \mu \ptial{2} (g_3)_2 = g_2, \quad x_1 \in \bbr, \ x_2 = 0,\ x_3 = 0.
	\end{equation*}
	To solve \eqref{AppendixEqs: Heat-quarter}, we reduce it to the case $ (u_0, f, g_3) = 0 $. Extend 
	\begin{gather*}
		u_0 \in \W{2 - \frac{2}{q}}(\bbr \times \bbr_+^2)^3, \quad
		f \in \Lq{q}(J; \Lq{q}(\bbr \times \bbr_+^2)^3), \\
		g_3 \in \W{1 - \frac{1}{2q}}(J; \Lq{q}(\bbr \times \bbr_+ \times \{0\})^3) \cap \Lq{q}(J; \W{2 - \frac{1}{q}}(\bbr \times \bbr_+ \times \{0\})^3)
	\end{gather*}
	with respect to $ x_2 $ by the similar extension as \eqref{Extension:Sobolev}
	to some functions
	\begin{gather*}
		\tu_0 \in \W{2 - \frac{2}{q}}(\bbr^2 \times \bbr_+)^3, \quad
		\tf \in \Lq{q}(J; \Lq{q}(\bbr^2 \times \bbr_+)^3), \\
		\tg_3 \in \W{1 - \frac{1}{2q}}(J; \Lq{q}(\bbr^2 \times \{0\})^3) \cap \Lq{q}(J; \W{2 - \frac{1}{q}}(\bbr^2 \times \{0\})^3).
	\end{gather*}
	Solving
	\begin{equation*}
		\begin{alignedat}{3}
			\rho \pt u - \mu \Delta u & = \tf, && \tin \bbr \times \bbr \times \bbr_+ \times J, \\
			u & = \tg_3, && \ton \bbr \times \bbr \times\{0\}  \times J, \\
			u(0) & = \tu_0, && \tin \bbr \times \bbr \times \bbr_+,
		\end{alignedat}
	\end{equation*}
	yields a unique solution 
	\begin{gather*}
		\bar{u} \in \W{1}(J; \Lq{q}(\bbr^2 \times \bbr_+)^3) \cap \Lq{q}(J; \W{2}(\bbr^2 \times \bbr_+)^3),
	\end{gather*}
	by \cite[Section 6.2]{PS2016}. Then $ (u - \bar{u}) $ (restricted) solves \eqref{AppendixEqs: Heat-quarter} with $ (u_0, f, g_3) = 0 $ and modified data $ g_1 $, $ g_2 $ (not to be relabeled) having vanishing trace at $ t = 0 $ and satisfying $ \rv{g_1}_{x_3 = 0} = 0 $ and $ \rv{g_2}_{x_3 = 0} = 0 $. Now by odd reflection	with regard to $ x_3 $, one can extend 
	\begin{gather*}
		g_1 \in \WO{1 - \frac{1}{2q}}(J; \Lq{q}(\bbr \times \{0\} \times \bbr_+)^2) \cap \Lq{q}(J; \W{2 - \frac{1}{q}}(\bbr \times \{0\} \times \bbr_+)^2), \\
		g_2 \in \WO{\onehalf - \frac{1}{2q}}(J; \Lq{q}(\bbr \times \{0\} \times \bbr_+)) \cap \Lq{q}(J; \W{1 - \frac{1}{q}}(\bbr \times \{0\} \times \bbr_+)),
	\end{gather*}
	to some functions
	\begin{gather*}
		\tg_1 \in \WO{1 - \frac{1}{2q}}(J; \Lq{q}(\bbr \times \{0\} \times \bbr)^2) \cap \Lq{q}(J; \W{2 - \frac{1}{q}}(\bbr \times \{0\} \times \bbr)^2), \\
		\tg_2 \in \WO{\onehalf - \frac{1}{2q}}(J; \Lq{q}(\bbr \times \{0\} \times \bbr)) \cap \Lq{q}(J; \W{1 - \frac{1}{q}}(\bbr \times \{0\} \times \bbr)).
	\end{gather*}
	Then we solve the half-space heat equation
	\begin{equation} \label{AppendixEqs: Heat-quarter-half}
		\begin{alignedat}{3}
			\rho \pt u - \mu \Delta u & = 0, && \tin \bbr \times \bbr_+ \times \bbr \times J, \\
			\tran{(u_1, u_3)} & = \tg_1, && \ton \bbr \times \{0\} \times \bbr \times J, \\
			2 \mu \ptial{2} u_2 & = \tg_2, && \ton \bbr \times \{0\} \times \bbr \times J, \\
			u(0) & = 0, && \tin \bbr \times \bbr_+ \times \bbr, 
		\end{alignedat}
	\end{equation}
	to obtain a unique solution 
	\begin{gather*}
		\tu \in \WO{1}(J; \Lq{q}(\bbr \times \bbr_+ \times \bbr)^3) \cap \Lq{q}(J; \W{2}(\bbr \times \bbr_+ \times \bbr)^3).
	\end{gather*}
	We remark here that \eqref{AppendixEqs: Heat-quarter-half} is actually assembled by two separated heat equations with Dirichlet boundary condition and Neumann boundary condition respectively. Namely,
	\begin{equation*} 
		\begin{alignedat}{5}
			\rho \pt \tran{(u_1, u_3)} - \mu \Delta \tran{(u_1, u_3)} & = 0, \quad & \rho \pt u_2 - \mu \Delta u_2 & = 0, && \tin \bbr \times \bbr_+ \times \bbr \times J, \\
			\tran{(u_1, u_3)} & = \tg_1, \quad & 2 \mu \ptial{2} u_2 & = \tg_2, && \ton \bbr \times \{0\} \times \bbr \times J, \\
			\tran{(u_1, u_3)}(0) & = 0, \quad & u_2(0) & = 0 && \tin \bbr \times \bbr_+ \times \bbr, 
		\end{alignedat}
	\end{equation*}
	Hence \eqref{AppendixEqs: Heat-quarter-half} is uniquely solved by \cite[Section 6.2]{PS2016}. By the symmetry, one knows that 
	\begin{equation*}
		\baru(t, x_1, x_2, x_3)
		:= \tran{(- \tu_1 (t, x_1, x_2, - x_3), - \tu_2 (t, x_1, x_2, - x_3), - \tu_3 (t, x_1, x_2, - x_3))}
	\end{equation*}
	is also a solution to \eqref{AppendixEqs: Heat-quarter-half}. Then the uniqueness implies that $ \baru (t, x_1, x_2, x_3) = \tu (t, x_1, x_2, x_3) $, i.e., $ \tu (t, x_1, x_2, 0) = 0 $. Consequently, the restricted function 
	\begin{equation*}
		(\bar{u} + \tu) \in \W{1}(J; \Lq{q}(\bbr \times \bbr_+^2)^3) \cap \Lq{q}(J; \W{2}(\bbr \times \bbr_+^2)^3)
	\end{equation*}
	solves \eqref{AppendixEqs: Heat-quarter} with compatibility condition at the contact line.
	
	Analogously, there is a solution operator with regard to the data and solutions for the Heat equation in a bent quarter space as in Section \ref{Section: Stokes-Bent-Quarter}.
	
	\textbf{Heat transmission problem in a half space.} For a half space $ \bbr \times \bbr_+ \times \bbr $,
	we consider the problem
	\begin{equation}\label{AppendixEqs: Heat-transmission-half}
		\begin{alignedat}{3}
			\rho \pt u - \mu \Delta u & = f, && \tin \bbr \times \bbr_+ \times \dR \times J, \\
			\jump{u} & = g_1, && \ton \bbr \times \bbr_+ \times \{0\} \times J, \\
			\jump{\mu(\ptial{3} u + \nabla u_3)} & = g_2, && \ton \bbr \times \bbr_+ \times \{0\} \times J, \\
			\tran{(u_1, u_3)} & = g_3, && \ton \bbr \times \{0\} \times \dR \times J, \\
			2 \mu \ptial{2} u_2 & = g_4, && \ton \bbr \times \{0\} \times \dR \times J, \\
			u(0) & = u_0, && \tin \bbr \times \bbr_+ \times \dR.
		\end{alignedat}
	\end{equation}
	Then the data $ (u_0, g_j) $, $ j = 1,...,4 $, satisfy the corresponding initial compatibility condition
	\begin{gather*}
		\jump{u_0} = \rv{g_1}_{t = 0}, \quad
		\jump{\ptial{3} u_0 + \nabla (u_0)_3} = \rv{g_2}_{t = 0}, \\
		\rv{\tran{((u_0)_1, (u_0)_3)}}_{x_2 = 0} = \rv{g_3}_{t = 0}, \quad 
		\rv{2 \mu \ptial{2} (u_0)_2}_{x_2 = 0} = \rv{g_4}_{t = 0},
	\end{gather*}
	and compatibility conditions at the contact line $ \{x \in \bbr^3: x_2 = 0, x_3 = 0\} $
	\begin{gather*}
		\tran{((g_1)_1, (g_1)_3)} = \jump{g_3}, \quad
		2 \ptial{2} (g_1)_2 = \jump{\frac{g_4}{\mu}}, \\
		(g_2)_1 = \jump{\mu (\ptial{3} (g_3)_1 + \ptial{1} (g_3)_2)}, \quad
		(g_2)_3 = \jump{2 \mu \ptial{3} (g_3)_2}.
	\end{gather*}
	Following the argument in Section \ref{Section: Twophase-halfspace}, one can easily reduce \eqref{AppendixEqs: Heat-transmission-half} to the case $ (u_0, f, g_3, g_4) = 0 $ with modified data (not to be relabeled) having vanishing trace at $ t = 0 $ and satisfying 
	\begin{equation*}
		\rv{\tran{((g_1)_1, (g_1)_3)}}_{x_2 = 0} = 0, \quad \rv{\ptial{2} (g_2)_2}_{x_2 = 0} = 0, \quad
		\rv{(g_2)_1}_{x_2 = 0} = 0, \quad
		\rv{(g_2)_3}_{x_2 = 0} = 0,
	\end{equation*}
	Subsequently, one can extend $ ((g_1)_1, (g_1)_3, (g_2)_1, (g_2)_3 ) $ and $ ((g_1)_2, (g_2)_2 ) $ by odd and even reflection respectively	to some functions
	\begin{gather*}
		\tg_1 \in \WO{1 - \frac{1}{2q}}(J; \Lq{q}(\bbr^2 \times \{0\})^3) \cap \Lq{q}(J; \W{2 - \frac{1}{q}}(\bbr^2 \times \{0\})^3), \\
		\tg_2 \in \WO{\onehalf - \frac{1}{2q}}(J; \Lq{q}(\bbr^2 \times \{0\})^3) \cap \Lq{q}(J; \W{1 - \frac{1}{q}}(\bbr^2 \times \{0\})^3).
	\end{gather*}
	Then we solve the full-space heat transmission problem with a flat interface
	\begin{equation*}
		\begin{alignedat}{3}
			\rho \pt u - \mu \Delta u & = 0, && \ton \bbr \times \bbr \times \dR \times J, \\
			\jump{u} & = \tg_1, && \ton \bbr \times \bbr \times \{0\} \times J, \\
			\jump{\mu(\ptial{3} u + \nabla u_3)} & = \tg_2, && \ton \bbr \times \bbr \times \{0\} \times J, \\
			u(0) & = 0, && \ton \bbr \times \bbr \times \dR, \\
		\end{alignedat}
	\end{equation*}
	to obtain a unique solution 
	\begin{gather*}
		\tu \in \WO{1}(J; \Lq{q}(\bbr^2 \times \bbr)^3) \cap \Lq{q}(J; \W{2}(\bbr^2 \times \dR)^3).
	\end{gather*}
	with the help of \cite[Section 6.5]{PS2016}. Again by symmetry, we conclude that 
	\begin{equation*}
		\hatu(t, x_1, x_2, x_3)
		:= \tran{(- \tu_1 (t, x_1, x_2, - x_3), \tu_2 (t, x_1, x_2, - x_3), - \tu_3 (t, x_1, x_2, - x_3))}
	\end{equation*}
	is a solution to \eqref{AppendixEqs: Heat-transmission-half} as well. From the uniqueness, that is, $ \hatu(t, x_1, x_2, x_3) = \tu(t, x_1, x_2, x_3) $, we know that 
	\begin{equation*}
		(\tu_1, \tu_3)(t, x_1, 0, x_3) = 0, \quad
		\ptial{2} \tu_2(t, x_1, 0, x_3) = 0.
	\end{equation*}
	Consequently, the restricted function 
	\begin{gather*}
		\tu \in \W{1}(J; \Lq{q}(\bbr \times \bbr_+ \times \bbr)^3) \cap \Lq{q}(J; \W{2}(\bbr^2 \times \dR)^3).
	\end{gather*}
	solves \eqref{AppendixEqs: Heat-transmission-half} with $ (u_0, f, g_3, g_4) = 0 $.
	As in Section \ref{Section: TwoPhase-Bent-halfspace}, one concludes that there is also a solution operator with regard to the data and solutions for the heat transmission problem in a half space with a bent interface.
	
	\textbf{Heat transmission problem with a Dirichlet boundary.} Suppose that $ \OM \subset \bbr^3 $ is a bounded domain with $ C^2 $-boundary, consisting of two parts $ \OM^\pm $ which are also open and such that $ \Bar{\OM^+} \subset \OM $. Let $ \Sigma = \partial \OM^+ $ be the interface separating $ \OM^\pm $ such that $ \OM = \OM^+ \cup \Sigma \cup \OM^- $. Consider the problem
	\begin{equation} \label{Appendix: Parabolic transmission with Dirichlet}
		\begin{alignedat}{3}
			\rho \pt u - \mu \Delta u & = f, && \tin \OM \backslash \Sigma \times J, \\
			\jump{u} & = g_1, && \ton \Sigma \times J, \\
			\jump{\mu (\nabla u + \tran{\nabla} u) \nu} & = g_2, && \ton \Sigma \times J, \\
			u & = g_3, && \ton \partial \OM \times J, \\
			u(0) & = u_0, && \tin \OM \backslash \Sigma,
		\end{alignedat}
	\end{equation}
	where $ \rho, \mu > 0 $, $ \nu $ denotes the outer unit normal vector on $ \Sigma $ pointing from $ \OM^+ $ into $ \OM^- $. Note that a similar transmission problem with a Neumann boundary condition was investigated in \cite[Section 6.5]{PS2016}, one could solve \eqref{Appendix: Parabolic transmission with Dirichlet} by a truncation technique. Readers are referred to \cite[Section 3.1.2]{AL2021} for more details, where we handled the two-phase Stokes problem. In this position, we only give a sketch of the procedure. To this end, we choose $ \psi(x) \in C_0^\infty (\OM) $ as a cutoff function over $ \OM $ such that 
	\begin{equation}
		\label{cutoff function}
		\psi(x) = 
		\left\{
		\begin{aligned}
			& 1, \quad \text{in a neighborhood of}\  \OM^+,\\
			& 0, \quad \text{in a neighborhood of}\  \partial \OM.
		\end{aligned}
		\right.
	\end{equation}
	Define $ u := \psi u_1 + (1 - \psi) u_2 $,
	where $ u_1 $ is the solution of a parabolic transmission problem with a Neumann boundary condition in $ \OM $, while $ u_2 $ solves the heat equation in $ \OM^- $ with a Dirichlet boundary condition. Then $ u $ solves \eqref{Appendix: Parabolic transmission with Dirichlet} with some remainders. Since systems for $ u_1 $ and $ u_2 $ are all solvable and hence, following the procedure in \cite[Section 3.1.2]{AL2021} completes the proof of this part.
	
	Finally, by a standard localization procedure as in Section \ref{Subsection: Localization procedure} (see also \cite{PS2016}), one can finish the proof analogously, with the help of Proposition \ref{AppendixPropositon: POU-Omega}.
\end{proof}

\end{document}